\documentclass[10pt,paper=a4,oneside,notitlepage]{report}

\usepackage{wrapfig}
\usepackage[T1]{fontenc}
\usepackage{titlesec, blindtext, color}
\usepackage{typearea}
\areaset[6mm]{140mm}{230mm}
\usepackage[paper=a4paper,left=32mm,right=32mm,top=40mm,bottom=40mm]{geometry}
\usepackage{caption}
\usepackage{subcaption}
\usepackage[hyphens]{url}
\usepackage{changepage}
\usepackage[utf8]{inputenc}	% Wir wollen UTF-8 (=keine Probleme mit Umlauten etc.)
\usepackage{ae,aecompl} %bessere Schrift
\usepackage{amsmath}
\usepackage{amsthm}
\usepackage{esint}        % Mittelwert Integral
\usepackage{amssymb} 
\usepackage[ngerman,british]{babel}  	% Deutsches Wörterbuch usw.
\usepackage{graphicx}	% wird für Figur Umgebung benötigt
\usepackage{color}	% ermöglicht einbinden von Farben
\usepackage{epstopdf}   % Wandelt .eps Dateien automatisch um
\usepackage{url}	% für URL mit \url{.....}
\usepackage{bbm}	% für einheitsmatrix \mathbbm{1}
\usepackage{listings}	% für code (gnuplot, C etc.)
\usepackage[font=small,labelfont=bf]{caption}  % Optionen für Bild- und Codeunterschriften
\usepackage[hidelinks]{hyperref}         % damit Links in der PDF anklickbar werden
\usepackage{booktabs}	% bessere Tabellen mit Abstand zur hline
\usepackage{float}
\usepackage{bibgerm}	% Bibstyle deutsches Literaturverzeichnis
\usepackage{tikz}
\usetikzlibrary{shapes.misc,shadows} % Für Schatten unter der box
\usepackage{stmaryrd}
\usepackage{tcolorbox}
\usepackage{tikz}              %Diagramme
\usepackage{multirow,tabularx}
\usepackage{changepage}
\usepackage{nameref} % zitiert Theorem, Formel etc. mit Namen
\usepackage{pifont} %Herz
\usepackage{theoremref} 
\usepackage{etoolbox}    %Kapitel beginnnt auf derselben Seite
\usepackage{fixmath} % for \mathbold
\usepackage[misc]{ifsym}

\makeatletter
\patchcmd{\chapter}{\if@openright\cleardoublepage\else\clearpage\fi}{}{}{} %Kapitel beginnnt auf derselben Seite
\makeatother

\theoremstyle{break}

\newtheorem{theorem}{Theorem} [chapter]
\newtheorem{defi}[theorem]{Definition} 
\newtheorem{prop}[theorem]{Proposition}

\newtheorem{lemma}[theorem]{Lemma}

\allowdisplaybreaks
\numberwithin{equation}{chapter}
\usepackage{tocloft}
\cftsetindents{chapter}{0em}{2em}
\cftsetindents{section}{0em}{3em}
\cftsetindents{subsection}{0em}{4em}
\cftsetindents{subsubsection}{0em}{5em}

\newcommand*{\toccontents}{\@starttoc{toc}}

\makeatletter
\def\@makechapterhead#1{%
  \vspace*{50\p@}%
  {\parindent \z@ \raggedright \normalfont
    \interlinepenalty\@M
  \center  \LARGE\bfseries  \thechapter. #1\par\nobreak %Größe, fett,...
    \vskip 20\p@                                       % Abstand Kapitelüberschrift zu nachfolgendem Text
  }}
 
  \titleformat{\section}
  {\normalfont\bfseries\Large}{\thesection}{1em}{} % Statt scshape-> \bfseries?

  \titleformat{\subsection}
  {\normalfont\bfseries\large}{\thesubsection}{1em}{}

\addto{\captionsbritish}{%
  
}

\addto{\captionsbritish}{%
  
}

\makeatother %Wörter verlinken

\makeatletter
\newcommand{\subalign}[1]{%
  \vcenter{%
    \Let@ \restore@math@cr \default@tag
    \baselineskip\fontdimen10 \scriptfont\tw@
    \advance\baselineskip\fontdimen12 \scriptfont\tw@
    \lineskip\thr@@\fontdimen8 \scriptfont\thr@@
    \lineskiplimit\lineskip
    \ialign{\hfil$\m@th\scriptstyle##$&$\m@th\scriptstyle{}##$\hfil\crcr
      #1\crcr
    }%
  }%
}
\makeatother

\binoppenalty=\maxdimen %verhindert Zeilenumbrüche in $ $-Formeln
\relpenalty=\maxdimen

\DeclareMathOperator*{\esssup}{ess\,sup}
\DeclareMathOperator{\divs}{div}
\DeclareMathOperator{\curl}{curl}
\DeclareMathOperator{\Ima}{Im}
\DeclareMathOperator{\Id}{Id}
\DeclareMathOperator{\tr}{tr}
\DeclareMathOperator{\supp}{supp}

\newcommand{\axi}{\ensuremath{a_{(\xi)}} \xspace}

  \newcommand{\wpr}{\ensuremath{w_{q+1}^{(p)}} \xspace} 
  \newcommand{\wc}{\ensuremath{w_{q+1}^{(c)}} \xspace}
  \newcommand{\wt}{\ensuremath{w_{q+1}^{(t)}} \xspace}

 \newcommand{\phixi}{\ensuremath\phi_{(\xi)}} 
  \newcommand{\Phixi}{\ensuremath\Phi_{(\xi)}}
  \newcommand{\psixi}{\ensuremath\psi_{(\xi)}}
 
 \newcommand{\Wxi}{\ensuremath{W_{(\xi)}} \xspace}
 \newcommand{\Wcxi}{\ensuremath{W^{(c)}_{(\xi)}} \xspace}
 \newcommand{\Vxi}{\ensuremath{V_{(\xi)}} \xspace}

\begin{document}

\begin{center}
\boldmath
~\\
\LARGE{\textbf{On the $3$D Navier-Stokes Equations with a Linear Multiplicative Noise and Prescribed Energy}}
\unboldmath
\end{center}

~\\
\begin{center}
\large{Stefanie Elisabeth Berkemeier}\\
\end{center}
\begin{center}
\today
\end{center}
\,\\
\begin{adjustwidth}{25pt}{25pt}
\textbf{Abstract.} For a prescribed deterministic kinetic energy we use convex integration to construct analytically weak and probabilistically strong solutions to the $3$D incompressible Navier-Stokes equations driven by a linear multiplicative stochastic forcing. These solutions are defined up to an arbitrarily large stopping time and have deterministic initial values, which are part of the construction. Moreover, by a suitable choice of different kinetic energies which coincide on an interval close to time $0$, we obtain non-uniqueness.
\end{adjustwidth}
~\\ \, \\ \, \\
\textbf{Keywords.} stochastic Navier-Stokes equations, multiplicative noise, kinetic energy, analytically weak solutions, probabilistically strong solutions, non-uniqueness, convex integration.
~\\  \,
\tableofcontents

%1.Chapter on same page as table of contents
{\let\clearpage\relax

  \chapter{Introduction}\label{Chapter 1}}
 \section{Motivation and Previous Works} \label{Section 1.1}
Proving existence and smoothness of strong solutions to the incompressible Navier-Stokes equations is a longstanding open problem in the field of fluid dynamics. It is the subject of one of the Millennium Prize Problems and caused, especially in the recent years, worldwide much attention.\\
In 2009 De Lellis and Sz\'ekelyhidi developed the method of convex integration which permits them to construct infinitely many weak solutions to the incompressible Euler equations which dissipate the total kinetic energy and satisfy the global and local energy inequality \cite{DLS09}, \cite{DLS10}, \cite{DLS13}. Together with Isett, Buckmaster and Vicol, they used this scheme to prove Onsager's conjecture in 2016 and 2017, respectively \cite{Is16}, \cite{BDLSV17}. Their work was a groundbreaking success regarding the theory of weak solutions and was seminal for many further results. \\
Buckmaster and Vicol applied the technique of convex integration in 2019 to establish existence and non-uniqueness of weak solutions to the incompressible Navier-Stokes equations with finite kinetic energy \cite{BV19a}, \cite{BV19b}. One year later, in July 2020, a similar result for power law flows in the deterministic setting  follows by Burczak, Modena and Sz\'ekelyhidi \cite{BMS21}. \\
Meanwhile the methods of convex integration even found their way in the theory of stochastic partial differential equations as well. A possible advantage by adding suitable stochastic perturbations into partial differential equations consists in regularizing deterministically ill-posed problems. Considering in particular a linear multiplicative noise provides additionally a certain stabilization effect on the three dimensional Navier-Stokes equations see e.g. Röckner, Zhu and Zhu \cite{RZZ13}. However this phenomena does not help when it comes to the question of uniqueness of probabilistically strong solutions to the Navier-Stokes equations, as also shown in the present paper.\\ 
A stochastic counterpart to \cite{BV19a}, \cite{BV19b} was obtained by Hofmanová, Zhu and Zhu, who were able to show existence and non-uniqueness of analytically weak and probabilistically strong solutions to the incompressible Navier-Stokes equations also with a prescribed energy but additionally driven by a stochastic additive noise \cite{HZZ21a}. \\
 Chen, Dong, Hofmanová, Zhu and Zhu developed further stochastic versions of convex integration to prove dissipative martingale solutions to 3D stochastic Euler equations, global existence and non-uniqueness for 3D stochastic Navier-Stokes equations with space time white noise, non-unique ergodic solutions for 3D Navier-Stokes equations and Euler equations as well as sharp non-uniqueness of solutions to stochastic Navier-Stokes equations \cite{HZZ21b}, \cite{HZZ22a}, \cite{HZZ21c}, \cite{CDZ10}. Very recently first results regarding the $3$D Euler equations with transport noise, power-law equations with additive noise and surface quasi-geostrophic equations with irregular spatial perturbations could already be achieved \cite{HLP22}, \cite{LZ22}, \cite{HZZ22b}.\\
In their previous work \cite{HZZ19} Hofmanová, Zhu and Zhu were concerned with the incompressible Navier-Stokes equations perturbed by three different stochastic forces: one with a multiplicative, one with an additive and one with a non-linear noise. In all these cases they proved that the law of analytically weak and probabilistically strong solutions is not unique and that the solutions violate the corresponding energy inequality.\\ 
A series of further results regarding non-uniqueness in law for several stochastic partial differential equations, such as the transport-diffusion equation, the $3$D magnetohydrodynamics system or also the $3$D Navier-Stokes equations,  was attained by Yamazaki and Rehmeier and Schenke \cite{Ya21a}, \cite{Ya21b}, \cite{Ya22}, \cite{RS22}. All these equations are perturbed by different kinds of random noise but without any statement concerning a prescribed energy. \\
In the present paper we follow the ideas of \cite{HZZ19} and \cite{HZZ21a} to prove the existence and non-uniqueness of solutions to the $3$D incompressible Navier-Stokes equations driven by a multiplicative noise up to an arbitrarily large stopping time. Opposed to \cite{HZZ19} we are now able to deduce that the kinetic energy of the constructed solution equals to a prescribed energy profile. To this end we make use of the transformation and convex integration technique therein, but had to reformulate the main iteration. For a detailed introduction into the Fourier analysis of convex integration we refer to \cite{Be23}.
   
  \section{Main Result}\label{Section 1.2}
  \noindent
We consider the three-dimensional incompressible Navier-Stokes equations perturbed by a linear multiplicative forcing
\begin{align}
\begin{split}
du+\divs(u \otimes u) \,dt+ \nabla P \,dt&= \nu\Delta u \,dt+ u\,dB, \\
\divs u&=0, \label{1.1}
\end{split}
\end{align}
posed on $[0,T]\times\mathbb{T}^3\times \Omega $ for some $T>0$, where $\left(B_t\right)_{t \in [0,T]}$ is a real-valued Brownian motion on an given probability space $\left( \Omega, \mathcal{F}, \mathcal{P}\right)$ and $\mathbb{T}^3= [0,2\pi]^3 $ denotes the three-dimensional torus. Moreover, let $\left(\mathcal{F}_{t}\right)_{t \in [0,T]}$ be the natural filtration generated by $\left(B_t\right)_{t \in [0,T]}$.\\
For some pressure $P\colon [0,T]\times\mathbb{T}^3\times \Omega \to \mathbb{R}$ the system governs the time evolution of the velocity $u \colon [0,T]\times\mathbb{T}^3\times \Omega \to \mathbb{R}^3$ of an incompressible fluid with viscosity $\nu$. Throughout the paper the viscosity is, for the sake of simplicity, assumed to be $1$, which physically corresponds to water of $20^\circ$C (cf. \cite{AT74}, p.1238, Table 3) and moreover we will often deal with the $\mathbb{T}^3$-periodic extensions of $u$ and $P$ on $\mathbb{R}^3$, which can be identified with functions on the three dimensional flat torus $\mathbb{R}^3\setminus \left(2\pi \mathbb{Z}\right)^3$ (cf. \thref{Lemma A.1}).\\
In this paper we are concerned with finding solutions in the following sense:
\begin{defi} \thlabel{Definition 1.1}
An $\left(\mathcal{F}_t\right)_{t \in [0,T]}$-adapted solution $u$ to \eqref{1.1} is said to be probabilistically strong and analytically weak if 
\begin{itemize}

\item[i)]it belongs for some $\gamma \in (0,1)$  to $C\left([0,T];H^\gamma\left(\mathbb{T}^3\right)\right)$ a.s., 
\item[ii)] it satisfies 
\begin{align*} 
 &\int_0^t \int_{\mathbb{T}^3} u(s,x,\omega)\cdot \varphi(x) \,dx \,dB_s\\
 &\hspace{0.2cm}= \int_{\mathbb{T}^3}   \Big(u(t,x,\omega)-u(0,x,\omega)\Big)\cdot\varphi(x)\,dx-\int_0^t \int_{\mathbb{T}^3} (u \otimes u)(s,x, \omega): \nabla \varphi^T(x)\,dx \,ds\\
 &\hspace{0.7cm}+ \int_0^t \int_{\mathbb{T}^3}  u(s,x,\omega)\cdot\Delta \varphi(x)\,dx\,ds 
\end{align*} 
for every divergence free test function $\varphi \in C^\infty\left(\mathbb{T}^3;\mathbb{R}^3\right)$, any $t \in [0,T]$ and almost all $\omega \in \Omega$,
\item[iii)] it is weakly divergence free, i.e. it obeys
\begin{align*} 
 \int_{\mathbb{T}^3} \Big(u(t,x,\omega)\cdot \nabla \Big)\phi(x) \,dx = 0
\end{align*} 
for all $\phi \in C^\infty\left(\mathbb{T}^3;\mathbb{R}\right)$, almost all $\omega \in \Omega$ at any time $t \in [0,T]$.
\end{itemize}  

\end{defi}
\noindent Note that working with divergence free test functions in the definition above allows to eliminate the pressure term, which can be reconstructed after a weak solution has been found. We can now formulate our main result:

\begin{theorem}\thlabel{Theorem 1.2}
For any $L>0$ arbitrarily large and every energy $e\in C^1_b\left([0,L];[\underline{e},\infty)\right)$, satisfying
\begin{align*}
\|e\|_{C}\leq \bar{e} \text{ \ and \ } \Big\|\frac{d}{dt} e\Big\|_{C}\leq \widetilde{e}
\end{align*}
for some constants $4< \underline{e}\leq \bar{e}$ and $\widetilde{e}>0$ a probabilistically strong and analytically weak solution $u$ to \eqref{1.1}, depending explicitly on the given energy $e$, can be constructed up to a $\mathcal{P}$-a.s. strictly positive stopping time
\begin{align*}
\tau:= \inf \{t>0: |B_t|\geq L\} \wedge \inf\{t>0: \|B\|_{C_{[0,t]}^{0,\iota}}\geq L \} \wedge L
\end{align*}
with $\iota \in \Big(\frac{1}{3}, \frac{1}{2}\Big)$. This solution has deterministic initial value $u_0$ and belongs to $C\left([0,\tau];H^\gamma\left(\mathbb{T}^3\right)\right)$ a.s. for some $\gamma \in (0,1)$. It obeys 
\begin{align*}
\esssup_{\omega \in \Omega} \sup_{t\in [0,\tau]} \|u(t,\omega)\|_{H^\gamma}< \infty
\end{align*}
and its kinetic energy is given by $e$, i.e.
\begin{align*}
\|u(t)\|^2_{L^2}=e(t),
\end{align*}
as long as $t \in [0,\tau]$.\\
Moreover the following consistency result holds:\\
If two energies with the same bounds $\underline{e}, \bar{e}, \widetilde{e}$ coincide for some $t \in [0,L]$ everywhere on $[0,t]$, then so do the corresponding solutions on $[0,t \wedge \tau]$.
\end{theorem}

 \noindent The proof of \thref{Theorem 1.2} is based on a convex integration scheme. That is, after transforming \eqref{1.1} to a random PDE, we develop an iteration procedure and apply it to the just received equation.\\
More precisely, if $u$ solves \eqref{1.1}, the function $v:=e^{-B}u$ is by Itô's formula a solution to the ensuing system
\begin{align}
\label{1.2}
\begin{split}
\partial_tv+\frac{1}{2} v- \Delta v+ \Theta\divs(v \otimes v)+\Theta^{-1}\nabla P&=0,\\
\divs v&=0,
\end{split}
\end{align}
where $\Theta$ is the stochastic process given by 
\begin{align*}
\Theta \colon [0,T]\times \Omega \to \mathbb{R}, \qquad\Theta(t,\omega):=e^{B_t(\omega)}
\end{align*}
and the converse is also true. In fact, applying Itô's formula to the smooth function \linebreak $f\colon [0,T]\times \mathbb{R} \to \mathbb{R}$, $f(t,y):= e^{-y}$ yields $de^{-B_t}=\frac{1}{2}e^{-B_t}\, dt- e^{-B_t}\, dB_t$ and by Itô's product rule and \eqref{1.1} we conclude \eqref{1.2}.

\section{Organization of the Paper}
We organize the present paper as follows: Chapter \ref{Chapter 2} is devoted to the collection of basic notations used throughout this paper. In Chapter \ref{Chapter 3} we outline the convex integration technique to prove \thref{Proposition 4.1}. This is the core of the proof of our main \thref{Theorem 1.2}, which is established in Chapter \ref{Chapter 4}. The Appendix \ref{Appendix} covers some lemmata used in the previous chapters.

\chapter{Preliminaries}\label{Chapter 2}
In order to define several function spaces and operators we need the Fourier transform and the inverse Fourier transform of a function $u$ on $\mathbb{T}^3$ given by
\begin{align*}
(\mathcal{F}u)(n)&:=\hat{u}_n:=(2\pi)^{-3} \int_{\mathbb{T}^3} u(y)e^{-in\cdot y}\, dy \qquad
\text{and} \qquad
 (\mathcal{F}^{-1}\hat{u}_n)(x)=u(x)=\sum_{n \in \mathbb{Z}^3}\hat{u}_n e^{in\cdot x}
\end{align*}
for any $n \in \mathbb{Z}^3$ and $x\in \mathbb{T}^3$, respectively, where the series shall be understood as the limit of partial sums with square-cut off
\begin{align*}
u_N(x):=\sum_{n \in [-N,N]^3}\hat{u}_n e^{in\cdot x}.
\end{align*}
Moreover, for $d \in \mathbb{N}$ we will often deal with the spaces of symmetric or traceless $d \times d$-matrices $A$, designated by $\mathbb{R}_{\text{sym}}^{d\times d}$ and $\mathring{\mathbb{R}}^{d\times d}$, respectively. As usual we consider the Frobenius norm on them and in order to elucidate that the matrix itself is traceless we will frequently write $\mathring{A}$ instead of $A$.

\section{Function Spaces}\label{Section 2.1}
For some $N\in \mathbb{N}_0\cup \{\infty\},\, T\in [0,\infty)$ and a Banach space $\left(Y,\|\cdot\|_Y\right)$ we denote the space of all $N$-times continuously differentiable functions from $X\in \{ \mathbb{T}^3,\mathbb{R}^3\}$ to $Y$ by $C_{XY}^N$ and from $X=(-\infty,T]$ to $Y$ we shorten $C_{TY}^N$. Endowed with the natural norms
\begin{align*}
\|u\|_{C_{XY}^N}:= \sum_{\substack{0 \leq |\alpha| \leq N\\  \alpha \in \mathbb{N}_0^3}}\sup_{x\in X} \|D^\alpha u(x) \|_Y  \qquad \text{ and } \qquad \|u\|_{C_{TY}^N}:= \sum_{\substack{0 \leq n \leq N\\  n \in \mathbb{N}_0}}\sup_{t \in(-\infty,T]}\| \partial_t^n u(t) \|_Y ,
\end{align*}
respectively, it is well known that they become Banach spaces.\\
Sometimes we will omit writing $X$ or $Y$ if it is clear from the context which domain or codomain is considered and moreover we will frequently write $C_{XY}$ instead of $C^0_{XY}$. In particular $\|u\|_C$ is the usual supremum norm, which will also be used for more general normed spaces $X$.\\
There should be no misunderstanding when we talk about $C_c^N$- and $C_b^N$-functions, meaning the $N$-times continuous differentiable ones with compact support and the ones that are either bounded from above, bounded from below or even both. \\
Besides it is customary to introduce the space of test functions 
\begin{align*}
\mathcal{D}\left(X\right):=\left\{
\begin{array}{ll}
C^\infty\left(X \right),& X \text{ is compact},\\
C_c^\infty\left(X\right),& \text{otherwise},
\end{array}\right.
\end{align*}
and we moreover label by $C_{T,x}^N$ the space of all $N$-times continuously differentiable functions on $(-\infty,T]\times \mathbb{T}^3$, equipped with the corresponding norm
\begin{align*}
\|u\|_{C_{T,x}^N}:= \sum_{\substack{0 \leq n+|\alpha| \leq N\\ n \in \mathbb{N}_0, \alpha \in \mathbb{N}_0^3}}\|\partial_t^n D^\alpha u\|_{C_TL^\infty}.
\end{align*} 
We will speak of Hölder-continuous functions $u$ on $X\in \{\mathbb{T}^3,\mathbb{R}^3\}$ of exponent $N+\iota$ with $N\in \mathbb{N}_0$ and $\iota \in (0,1]$, whenever
\begin{align*}
\|u\|_{C_{XY}^{N,\iota}}:=\sum_{\substack{0 \leq |\alpha| \leq N\\  \alpha \in \mathbb{N}_0^3}}\| D^\alpha u\|_{C_{XY}}+\sum_{\substack{ |\alpha| = N\\  \alpha \in \mathbb{N}_0^3}}\sup_{x \in X}[ D^\alpha u(x)]_{C_{XY}^{0,\iota}}<\infty
\end{align*}
holds, where
\begin{align*}
[u(x)]_{C_{XY}^{0,\iota}}:= \sup_{\substack{x_1,x_2 \in X\\ x_1 \neq x_2}} \frac{\|u(x_1)-u(x_2)\|_Y}{|x_1-x_2|^\iota},
\end{align*}
is the $\iota^\text{th}$-Hölder seminorm. For functions on $(-\infty,T]$ we define the space in the same manner. \\
The space of Bochner-integrable functions $L^p\left(X;Y\right)$ consists of the equivalence classes of all functions $u\colon X \to Y$, which coincide almost everywhere and for which the usual $L^p$-norm $\|\cdot\|_{L^p}$ is finite, whereas
$\left(W^{k,p}\big( X;\mathbb{R}^d\big),\|\cdot\|_{W^{k,p}}\right)$ should be the usual Sobolev space for all $k \in \mathbb{N}_0$ and $1\leq p\leq \infty$. Moreover we denote by $W^{k,p}_0\left(X\right)$ the closure of $\mathcal{D}\left(X\right)$ in $\left(W^{k,p}\big(X\big),\|\cdot\|_{W^{k,p}}\right)$ and by $W^{-k,q}\left(X\right)$ the dual space of $W^{k,p}_0\left(X\right)$ with $\frac{1}{p}+ \frac{1}{q}=1$.\\
The Bessel potential spaces are for any $p,q\in(1,\infty)$, $\frac{1}{p}+\frac{1}{q}=1$ on $\mathbb{T}^3$ defined in the spirit of \cite{Tr83}, \cite{Tr92}:\\
We set
\begin{align*}
W^{s,p}\left(\mathbb{T}^3\right)&:=\bigg\{u \in L^p\left(\mathbb{T}^3\right): \Big\|\mathcal{F}^{-1}\big[(1+|\cdot|^2)^{s/2}\mathcal{F}u\big] \Big\|_{L^p}<\infty\bigg\}
\intertext{for $s\geq 0$, whereas we define} 
W^{-s,p}\left(\mathbb{T}^3\right)&:=\bigg\{u ^\prime\in \left(\mathcal{D}\left(\mathbb{T}^3\right)\right)^\prime: \Big\|\mathcal{F}^{-1}\big[(1+|\cdot|^2)^{-s/2}\mathcal{F}u^\prime\big]  \Big\|_{(L^q)^\prime}<\infty\bigg\} 
\end{align*}
whenever $s>0$, which can be identified with the dual space of $W^{s,q}\left(\mathbb{T}^3\right)$ (see also \cite{DNPV11}).\\
Endowed with the canonical norms
\begin{align*}
\|u\|_{W^{s,p}}:= \Big\|\mathcal{F}^{-1}\big[(1+|\cdot|^2)^{s/2}\mathcal{F}u\big] \Big\|_{L^p}
 \quad \text{and} \quad
\|u^\prime\|_{W^{-s,p}}:= \Big\|\mathcal{F}^{-1}\big[(1+|\cdot|^2)^{-s/2}\mathcal{F}u^\prime\big] \Big\|_{(L^q)^\prime},
\end{align*}
respectively, these spaces become Banach spaces and furthermore we stipulate $\displaystyle H^s\left(X\right) :=\linebreak W^{s,2}\left(X\right)$.\\
For $p\in[1,\infty], \, d\in \mathbb{N}$ and $X\in \{\mathbb{N}^d, \mathbb{Z}^d\}$ we denote by $\ell^p(X)$ the usual space of sequences for which the corresponding norm $\|\cdot\|_{\ell^p}$ is finite.\\
We will often deal with functions of zero mean. So for convenience we set
\begin{align*}
X_{\neq 0}:=\mathbb{P}_{\neq 0} \left( X\right),
\end{align*}
for any function space $X$, where $\mathbb{P}_{\neq 0}$ is the projection onto functions with non-zero frequencies, which will be introduced in \thref{Definition 2.2} below.

\section{Operators} \label{Section 2.2}

We will deal with the extended Leray projection $\mathbb{P}=\Id-\nabla\Delta^{-1}\divs$ on the Bessel potential spaces $W_{\neq 0}^{s,p}\left(\mathbb{T}^3\right)$ for every real $s\geq 0$ and $p \in (1,\infty)$, which enjoys the following useful property. 
\begin{lemma}\thlabel{Lemma 2.1}
The Leray projection commutes almost everywhere with any partial derivative on $W_{\neq 0}^{1,p}\left(\mathbb{T}^3\right)$ and if we work with time depended functions $u$, satisfying 
\begin{align*}
|\partial_t u(t,x)|\leq g(x)
\end{align*}
for some $g\in L^1\left( \mathbb{T}^3\right)$ and all $x \in \mathbb{T}^3$, even on $C^1\left((-\infty,T];C_{\neq 0}\left(\mathbb{T}^3\right)\right)$ with $T>0$. In other words 
\begin{align*}
\partial_{x_j}(\mathbb{P}u)(t,x)=(\mathbb{P}\partial_{x_j} u)(t,x) \qquad \text{ and } \qquad \partial_t(\mathbb{P}u)(t,x)=(\mathbb{P}\partial_t u)(t,x)
\end{align*}
holds for all $x\in \mathbb{T}^3$, each $t\in \mathbb{R} $ and $j=1,2,3$.
\end{lemma}
 \noindent Moreover, it is very common and practical to introduce the Fourier multiplier operators, which project a function onto its null mean frequencies and onto its frequencies $\leq \kappa$ in absolute value.

\begin{defi}\thlabel{Definition 2.2}
For any $u \in L^p\left(\mathbb{T}^3\right)$ with $1<p\leq \infty$ we denote by
\begin{align*}
\left(\mathbb{P}_{\neq 0}u\right)(x):=\mathcal{F}^{-1}\Big[ \mathbbm{1}_{\{|\cdot|\neq 0 \}}\mathcal{F}u\Big](x)=u(x)-(\mathcal{F}u)(0)=u(x)-\fint_{ \mathbb{T}^3}u(x)\,dx
\end{align*}
the projection onto its non-zero frequencies. \\
Furthermore for all real $\kappa \geq 1$ we define the operators $\mathbb{P}_{\leq \kappa}$ and $\mathbb{P}_{\geq \kappa}$ as
\begin{align*}
\left(\mathbb{P}_{\leq \kappa}u\right)(x)&:= \mathcal{F}^{-1}\Big[ \chi_\kappa\, \mathcal{F}u\Big](x)=\sum_{n \in \mathbb{Z}^3} \chi_\kappa(n)\hat{u}_n e^{in\cdot x}
\intertext{and}
\left(\mathbb{P}_{\geq \kappa}u\right)(x)&:= \mathcal{F}^{-1}\Big[ \big(1-\chi_\kappa\big)\, \mathcal{F}u\Big](x)=\sum_{n \in \mathbb{Z}^3} \big(1-\chi_\kappa(n)\big)\hat{u}_ne^{in\cdot x},
\end{align*}
respectively, where we set $\chi_\kappa:= \chi \left( \frac{\cdot}{\kappa}\right)$ for the smooth compactly supported function $\chi \in C_c^\infty \left(\mathbb{R}^3\right)$ given by
\begin{align*}
\chi(x):=\left\{\begin{array}{ll}
1,& |x|\leq \frac{1}{2}, \\ [1.3ex] 
\frac{1}{\exp\left(\frac{1}{1/2-|x|}+\frac{1}{1-|x|} \right)+1},& \frac{1}{2}<|x|<1,\\
0,& |x|\geq 1.\\  
\end{array}\right. 
\end{align*}
\end{defi}

\begin{lemma} \thlabel{Lemma 2.3}
The above operators $\mathbb{P}_{\neq 0},\, \mathbb{P}_{\leq \kappa}$ and $\mathbb{P}_{\geq \kappa}$ are for each real $\kappa \geq 1$ and $1<p\leq \infty$ continuous on $L^p\left(\mathbb{T}^3\right)$, where the implicit constants do not depend on $\kappa$.
\end{lemma}
\begin{lemma} \thlabel{Lemma 2.4}
For all real $\kappa\geq 1$ it holds
\begin{align*}
\|\left(-\Delta\right)^{-1/2} \mathbb{P}_{\geq \kappa}\|_{L^p\to L^p} \lesssim \frac{1}{\kappa}
\end{align*}
whenever $1<p<\infty$.
\end{lemma}

\begin{lemma} \thlabel{Lemma 2.5}
For all $\left(\frac{\mathbb{T}}{L}\right)^3$-periodic functions $u \in L^p \left( \mathbb{T}^3\right)$ with $L \in \mathbb{N}$ and $1< p\leq \infty$ the operators $\mathbb{P}_{\leq \kappa}$ and $\mathbb{P}_{\geq \kappa}$ can be written as 
\begin{align*}
\left(\mathbb{P}_{\leq \kappa}u\right)(x)&= \sum_{n \in (L\mathbb{Z})^3} \chi_\kappa(n)\hat{u}_{n}e^{in\cdot x}
\intertext{and}
\left(\mathbb{P}_{\geq \kappa}u\right)(x)&= \sum_{n \in (L\mathbb{Z})^3} \big(1-\chi_\kappa(n)\big)\hat{u}_{n}e^{in\cdot x},
\end{align*}
respectively. If additionally $L>\kappa$, we have that
\begin{align*}
\mathbb{P}_{\geq \kappa}u=\mathbb{P}_{\neq 0}u.
\end{align*}
\end{lemma}

\chapter{Outline of the Convex Integration Scheme}
 \label{Chapter 3}
The convex integration scheme is an iterative procedure giving rise to solutions to several deterministic and stochastic PDE's. Also in the present paper we construct, based on a suitable starting point, a solution $v_q$ to \eqref{1.2} on $(-\infty,\tau]$ perturbed by an error term $\mathring{R}_q$, called Reynolds stress, on the level $q\in \mathbb{N}_0$. While the iterations $v_q$ approach the desired velocity $v$, solving \eqref{1.2} on $[0,\tau]$, the stress tensor $\mathring{R}_q$ becomes step by step infinitesimally small. The convex integration technique provides typically a way to construct even an infinite number of such solutions, as it is also the case in the present paper.  \\
For the construction of our sequence $(v_q,\mathring{R}_q)_{q\in \mathbb{N}_0}$ we previously have to fix some parameters, done in the following section, Section \ref{Section 3.1}. Section \ref{Section 3.2} is concerned with the key bounds that each pair $(v_q,\mathring{R}_q)$ has to fulfill. In Section \ref{Section 3.3} we recall the definitions of intermittent jets used to give an explicit expression of $v_{q+1}$. Section \ref{Section 3.4} is then concerned with the verification of the key bounds of the next iteration step $v_{q+1}$, culminating in the proof of convergence of the sequence $\left(v_q\right)_{q\in \mathbb{N}_0}$. In Section \ref{Section 3.5} we decompose the Reynolds stress and in Section \ref{Section 3.6} we verify the key bound for this tensor on the level $q+1$. We close this chapter in Section \ref{Section 3.7} by proving that our constructed sequence is adapted and deterministic at any time $t\leq 0$.

\section{Choice of Parameters}\label{Section 3.1}
  \addtocontents{toc}{\par\noindent\rule{\textwidth}{0.4pt}
 \hspace*{-10pt} 
\begin{minipage}{0.15\textwidth}
\includegraphics[width=\textwidth]{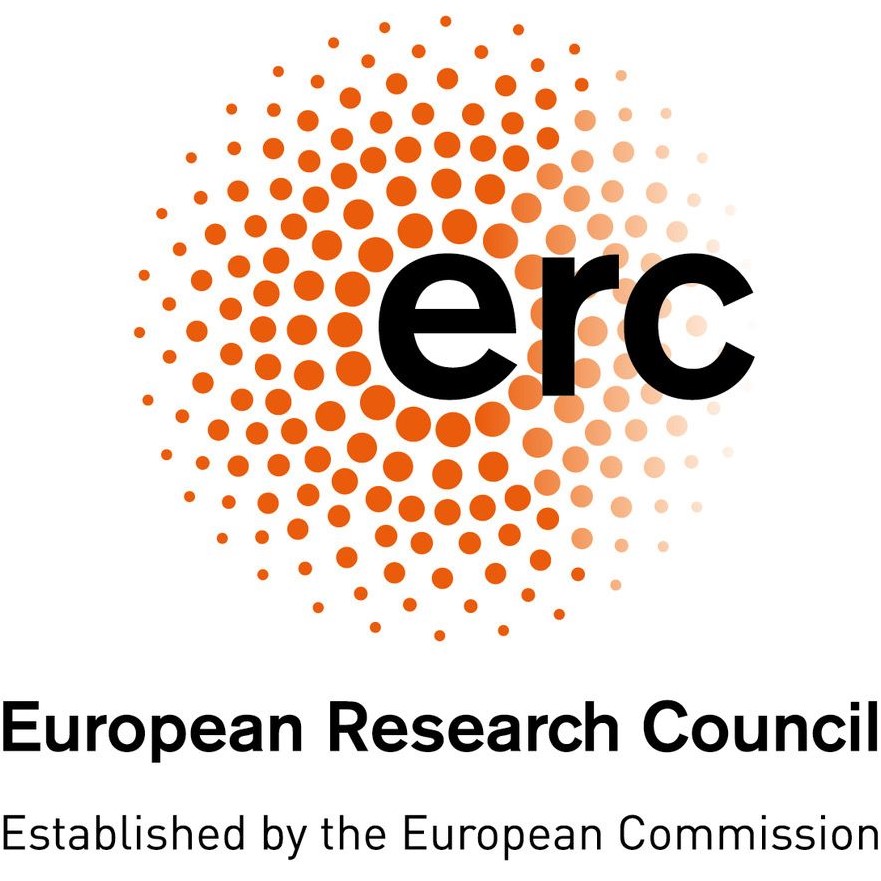}
\end{minipage}
\begin{minipage}[c]{0.85\textwidth}
This project has received funding from the European Research Council (ERC) under the European Union's Horizon 2020 research and innovation programme (grant agreement No. 949981). 
\end{minipage} \\
Faculty of mathematics, University of Bielefeld, D-33501 Bielefeld, Germany,\\
\Letter \, \href{mailto:sberkeme@math.uni-bielefeld.de}{sberkeme@math.uni-bielefeld.de}
  }

To simplify the upcoming computations we assume $L=1$ in \thref{Theorem 1.2} and for sufficiently large $a \in \mathbb{N}, b\in 7\mathbb{N}$ and sufficiently small $\alpha,\beta \in (0,1)$, we require  \\
\begin{tabular}{ll}
 \parbox{0.5\textwidth}{
\begin{itemize}
\item $161\alpha<\frac{1}{7}$,
\item $\alpha \iota> 4\beta b^2$,
\item $\alpha b\geq \frac{10}{\iota}-4$,
\end{itemize} }
 &
 \parbox{0.5\textwidth}{
 \begin{itemize}
\item $a\geq 3600$,
\item $a^{\beta b}\geq 2$,
\item $a^{\frac{3\alpha}{2}b+2}\geq m^2\bar{e}$
\end{itemize}}
\end{tabular} \\
and define \\

\begin{tabular}{ll}
 \parbox{0.5\textwidth}{
\begin{itemize}
\item[]$\lambda_q:=a^{(b^q)}$,
\item[] $\delta_q:=\lambda_1^{2\beta}\lambda_q^{-2\beta}$,
\end{itemize} }
 &
 \parbox{0.5\textwidth}{
 \begin{itemize}
\item[] $\ell:= \lambda_{q+1}^{-\frac{3\alpha}{2}}\lambda_q^{-2}$,
\item[] $m:=\exp{(4)}$
\end{itemize}}
\end{tabular} \\
for $q\in \mathbb{N}_0$. As a consequence the following estimates hold

\begin{subequations}\label{3.1}
  \begin{align}
  \ell \lambda_q^{5/\iota}\leq \lambda_{q+1}^{-\alpha}, \label{3.1a}\\
  m^2\bar{e}\leq \ell^{-1}\leq \lambda_{q+1}^{2\alpha}. \label{3.1b}
  \end{align}
  \end{subequations}
  
\noindent Particularly, that means $\ell \in (0,1)$ and we have developed an increasing and a decreasing sequence $\left(\lambda_q\right)_{q \in \mathbb{N}_0}\subseteq \mathbb{N}$, $\left(\delta_q\right)_{q\geq 2}\subseteq (0,1)$, which diverges to $\infty$ and converges to $0$, respectively. We therefore find some $q_0 \in \mathbb{N}$ so that the above parameters additionally fulfill

\hspace{-1.3cm}
\begin{tabular}{lr}
 \parbox{0.51\textwidth}{
\begin{itemize}
\item $10\, M_0 \bar{e}\lambda_{q+1}^{-\alpha/2+2\beta b^2}\leq 1$,
\item $\frac{M}{4|\Lambda|}\lambda_{q+1}^{33\alpha-1/7}\leq 1$,
\item $33(2\pi)^{3/2}M_0 m^{9/4}(\bar{e}+\widetilde{e})  \lambda_{q+1}^{-\alpha \big(\frac{3}{2}\iota-\frac{1}{2}\big)} \leq \frac{1}{1500m^{1/2}}$,
\item $S\left(\frac{M}{4|\Lambda
|}+\left(\frac{M}{4|\Lambda|}\right)^2 \right) \lambda_{q+1}^{-100\alpha}\leq \frac{1}{1500m^{1/2}}$,
\item $3\widetilde{S}\max\limits_{n\in\{2,3,4\}}\left(\frac{M}{4|\Lambda|}\right)^n\lambda_{q+1}^{-68 \alpha}\leq \frac{1}{1500m^{1/2}}$,
\item $\hat{S}\left( \frac{M}{4|\Lambda|}\right)^2  \lambda_{q+1}^{-111 \alpha}\leq \frac{1}{1500m^{1/2}}$,
\end{itemize} }
 &
 \parbox{0.5\textwidth}{
 \begin{itemize}
\item$2K\frac{M}{4|\Lambda|} \lambda_{q+1}^{-12/7}+K\left(\frac{M}{4|\Lambda|}\right)^2  \lambda_{q+1}^{-6/7} \leq \frac{1}{2}$, 
\item $ \widetilde{K}\lambda_q^5\lambda_{q+1}^{-5}\leq \frac{1}{2}$,
\item $\hat{K}\lambda_{q+1}^{-147\alpha}\leq \frac{1}{80m^{3/4}}$,
\item $K^\ast \, 2\pi \, \frac{M}{4|\Lambda|} \lambda_{q+1}^{-5/21} \leq \frac{1}{5m^{1/2}}$,

\item $ K^\prime \left(M_0+\frac{M}{4|\Lambda|} \right)^3 \lambda_{q+1}^{-\frac{1}{14}}\leq \frac{1}{5m^{1/2}}$,
\item $K^{\prime \prime} \left(\left(\frac{M}{4|\Lambda|} \right)^2+\left(\frac{M}{4|\Lambda|} \right)^4\right)  \lambda_{q+1}^{-1/7}\leq \frac{1}{5m^{1/2}}$
\end{itemize}}
\end{tabular} \\
for every $q\geq q_0$, where $K, \,\widetilde{K}, \, \hat{K}, \, K^\ast, \, K^\prime, \, K^{\prime \prime}, \, S, \, \widetilde{S},\,\hat{S}, \, \frac{M}{4|\Lambda|}$ and $M_0\geq 1$ are universal constants determined by \eqref{3.28}, \eqref{3.29}, \eqref{3.38}, \eqref{3.39}, \eqref{3.40}, \eqref{3.41}, \eqref{3.45}, \eqref{3.46}, \eqref{3.47}, \eqref{3.7} and \eqref{3.26}, respectively.\\
Note that the assumption $33(2\pi)^{3/2}M_0 m^{9/4}(\bar{e}+\widetilde{e})  \lambda_{q+1}^{-\alpha \big(\frac{3}{2}\iota-\frac{1}{2}\big)} \leq \frac{1}{1500m^{1/2}}$ requires $\iota >\frac{1}{3}$. Choosing furthermore $a,b$ sufficiently large and $\alpha,\beta$ sufficiently small enough, permits to suppose $q_0=1$.

\section{Start of the Iteration}\label{Section 3.2}
In view of \eqref{1.2} we are concerned with an adapted velocity field $v_q$ and a symmetric traceless matrix $\mathring{R}_q$, solving the transformed Navier-Stokes-Reynolds system reads 
\begin{align}
\begin{split}
\partial_t v_q+ \frac{1}{2}v_q-\Delta v_q+ \Theta \divs(v_q \otimes v_q)+\nabla p_q&=\divs(\mathring{R}_q),\\ \label{3.2}
\divs v_q&=0
\end{split}
\end{align}
and that obey 
\begin{subequations}\label{3.3}
\begin{align} 
&\|v_q\|_{C_tL^2}\leq M_0 \Big( 1+\sum_{r=1}^q \delta_r^{1/2}\Big)m\bar{e}^{1/2},\label{3.3a}\\
&\|v_q\|_{C_{t,x}^1}\leq  \lambda_q^5 m\bar{e}^{1/2},\label{3.3b} \\
&\|\mathring{R}_q\|_{C_tL^1}\leq \frac{1}{1500}  \delta_{q+2}\Theta^{-2}(t)e(t)\label{3.3c}
\end{align}
\end{subequations}
for $q\in \mathbb{N}_0$, any $t\in (-\infty,\tau]$ and some universal constant $M_0\geq 1$. Here we have to include also negative times in order to avoid several problems by decomposing the Reynolds stress in Section \ref{Section 3.5}. For this purpose the energy $e$ and the Brownian motion $B$ are continuously extended to functions on $(-\infty,\tau]$ by taking them equal to the value at $t=0$. We furthermore set $\sum_{r=1}^0:=0$ and point out that $\sum_{r=1}^q \delta_r^{1/2}$ is bounded by $2$. Indeed, thanks to the assumed $a^{\beta b} \geq 2$ it holds
\begin{align} \label{3.4}
\sum_{r=1}^q \delta_r^{1/2}\leq \sum_{r=1}^\infty a^{\beta b-\beta rb}=\sum_{r=0}^\infty (a^{-\beta b})^r =\frac{1}{1-a^{-\beta b}}\leq 2.
\end{align}
Moreover we require
\begin{align}\label{3.5}
\frac{3}{4}\delta_{q+1} \Theta^{-2}(t)e(t)\leq \Theta^{-2}(t)e(t)- \|v_q(t)\|^2_{L^2}\leq \frac{5}{4}\delta_{q+1} \Theta^{-2}(t)e(t)
\end{align}
at any time $t$ up to the stopping time $\tau$. In other words, the given energy $e$ will be gradually approximated by the kinetic energies of the iterations $e^Bv_q$. \\
If we set $\mathcal{F}_t:=\mathcal{F}_0$ whenever $t<0$, the pair $\big(v_0,\mathring{R}_0\big):=(0,0)$ is evidently a deterministic, hence $\left(\mathcal{F}_t\right)_{t\in \mathbb{R}}$-adapted, weak solution to \eqref{3.2}, satisfying \eqref{3.3} as well as \eqref{3.5}. Therefore we may start our iteration procedure with this pair.

\boldmath
\section{Construction of $v_{q+1}$} \label{Section 3.3}
\unboldmath
In order to obtain more regularity we refrain in defining the next iteration $v_{q+1}$ in mere dependence of $v_q$. Instead we intend to define $v_{q+1}$ in terms of the mollified velocity field $v_\ell$ and a perturbation $w_{q+1}$, pointed out in the two subsequent sections. For short, the new velocity field will be given as
\begin{align*}
v_{q+1}:=v_{\ell}+w_{q+1}.
\end{align*}
\subsection{Mollification} \label{Section 3.3.1}
Let us start with the mollification of the velocity $v_q$. For this end we consider the standard  mollifier
\begin{align*}
\phi(x):=\left\{\begin{array}{ll} C_{\text{space}} \exp{\left(\frac{1}{|x|^2-1}\right)}, &\text{if } |x|<1, \\
         0, & \text{if } |x| \geq 1, \end{array}\right.
\end{align*}
on $\mathbb{R}^3$ and the shifted mollifier
\begin{align*}
\varphi(t):=\left\{\begin{array}{ll} C_{\text{time}} \exp{\left(\frac{1}{|t-1/2|^2-1/4}\right)}, &\text{if } t \in(0,1), \\
         0, & \text{if } t\notin (0,1), \end{array}\right.
\end{align*}
on $\mathbb{R}$, where $C_{\text{space}},C_{\text{time}}>0$ are chosen, such that $\int_{\mathbb{R}^3}\phi(x)\, dx=1$ and $\int_{\mathbb{R}}\varphi(t)\, dt=1$, as usual. Thus the convolution of $v_q$ with the rescaled mollifiers $\phi_\ell:=\frac{1}{\ell^3}\phi(\frac{\cdot}{\ell})$ and $\varphi_\ell:=\frac{1}{\ell}\varphi(\frac{\cdot}{\ell})$ yield a smooth function in space and time, and of $\Theta$ with $\varphi_\ell$ just in time.\\
For short, we consider
\begin{align*}
v_\ell:=(v_q \ast_t \varphi_\ell)\ast_x\phi_\ell, \quad \mathring{R}_\ell:= (\mathring{R}_q \ast_t \varphi_\ell)\ast_x\phi_\ell, \quad \Theta_\ell:=\Theta \ast_t \varphi_\ell,
\end{align*}
where $\mathring{R}_\ell$ remains traceless, because the mollification acts only componentwise. \\
It is worth noting that the adaptedness of $v_\ell,\, \mathring{R}_\ell$ and $\Theta_\ell$ follows from the fact that the support of $\varphi$ lives in $\mathbb{R}^+$. In fact, let $\Pi_\ell$ be the set that contains all partitions of $[0,\ell]$ of the form $\{0=s_0<s_1<\ldots<s_n=\ell\}$ for some $n\in \mathbb{N}$. Then for any $t\in (-\infty,\tau]$ and $x\in \mathbb{T}^3$ the function 
\begin{align*}
\omega \mapsto (v_q \ast_t \varphi_\ell)(t,x,\omega)&=\int_0^\ell v_q(t-s,x,\omega)\varphi_\ell(s) \,ds\\
&=\inf_{\Pi_\ell}\sum_{i=1}^n(s_i-s_{i-1})\sup_{s\in(s_{i-1},s_i)} \underbrace{v_q(t-s,x,\omega)}_{\mathcal{F}_{t-s}\subseteq \mathcal{F}_t\text{-meas.}}\hspace{-0.35cm}\overbrace{\varphi_\ell(s)}^{\text{deterministic}}
\end{align*}
is $\mathcal{F}_t$-measurable.\\
Taking into account that the convolution in space does not influence the behavior of $(v_q \ast_t \varphi_\ell)$ in time, $v_\ell$ inherits the $\left(\mathcal{F}_t\right)_{t\in \mathbb{R}}$-adaptedness and so do $\mathring{R}_\ell$ and $\Theta_\ell$. \\
It is easy to see that $v_\ell$ is close to $v_q$ w.r.t. $\|\cdot\|_{L^2}$ at any time up to the stopping time $\tau$ and fulfills \eqref{3.3} as well. More precisely, the subsequent Lemma holds.

\begin{lemma} \thlabel{Lemma 3.1}
The mollification $v_\ell$, defined above, enjoys the following bounds
\begin{subequations}\label{3.6}
\begin{align}
&\|v_q-v_\ell\|_{C_tL^2}\leq (2\pi)^{3/2}\ell\|v_q\|_{C^1_{t,x}}\leq (2\pi)^{3/2} \delta_{q+1}^{1/2}m \bar{e}^{1/2},\label{3.6a}\\
&\|v_\ell\|_{C_tL^2}\leq \|v_q\|_{C_tL^2}\leq M_0\Big(1+\sum_{r=1}^q \delta_r^{1/2}\Big)m\bar{e}^{1/2},\label{3.6b}\\
&\|v_\ell\|_{C_{t,x}^N}\lesssim \ell^{-N+1}\|v_q\|_{C_{t,x}^1}\lesssim  \ell^{-N}\lambda_{q+1}^{-\alpha} m \bar{e}^{1/2}\label{3.6c}
\end{align}
\end{subequations}
for any $t \in (-\infty,\tau]$ and $N\geq 1$.
\end{lemma} 
 \noindent The proof is rather straightforward, so that we do not pursue this here.

\subsection{Perturbation} \label{Section 3.3.2}
Let us now have a closer look at the perturbation $w_{q+1}$. We will decompose it into three parts: the principle part $\wpr$, the incompressibility corrector $\wc$ and the temporal corrector $\wt$. Each of them will be defined in terms of the amplitude functions and the intermittent jets, introduced and worked out in \cite{BV19b} and \cite{BV19a}, respectively. In what follows we will give a short review of the necessary facts.\\
First, we recall the essential geometric lemma from \cite{BV19a}.

\begin{lemma} [Geometric Lemma]\thlabel{Lemma 3.2} 
There exists a family of smooth real-valued functions $\left(\gamma_\xi\right)_{\xi \in \Lambda}$, where $\Lambda$ is a set of finite directions, contained in $\mathbb{S}^2\cap\mathbb{Q}^3$, so that each symmetric $3\times 3$ matrix $R$, satisfying $\|R-\Id\|_F\leq \frac{1}{2}$, admits the representation
\begin{align*}
R=\sum_{\xi\in\Lambda}\gamma_\xi^2(R)(\xi \otimes \xi).
\end{align*}
\end{lemma}
\vspace{.5cm}
 \noindent Second, based on this lemma, we define for all $N\in \mathbb{N}$ the constant
\begin{align}
M:=8 |\Lambda|(1+8\pi^3)^{1/2} \sup_{\xi \in \Lambda}\Big(\|\gamma_\xi\|_{C}+\sum_{|\alpha|\leq N}\|D^\alpha\gamma_\xi\|_C\Big). \label{3.7}
\end{align}

\subsubsection{Amplitude Functions}\label{Section 3.3.2.1}
The function $\gamma_\xi$ in the geometric \thref{Lemma 3.2} is used to define the amplitude functions as 
\begin{align*}
a_{(\xi)}(t,x,\omega)&:=a_{\xi,q+1}(t,x,\omega):=(2\pi)^{3/2} \rho^{1/2}(t,x,\omega)\gamma_\xi\left(\Id-\frac{\mathring{R}_\ell(t,x,\omega)}{\rho(t,x,\omega)}\right)
\end{align*}
with 
\begin{align*}
\rho(t,x,\omega)&:=2\sqrt{\ell^2+\|\mathring{R}_\ell(t,x,\omega)\|_F^2}+\Theta_\ell(t,\omega)\eta_\ell(t,\omega),\\
\\
\eta_q(t,\omega)&:=  \frac{1}{3 (2\pi)^3}\Big[\Theta^{-2} (t,\omega)e(t)(1-\delta_{q+2})-\|v_q(t)\|^2_{L^2}\Big]
\end{align*}
for any $t \in (-\infty, \tau],\, x\in \mathbb{R}^3,\, \omega \in \Omega$, where $\eta_\ell$ denotes the mollification in time of $\eta_q$.\\
Since we start our iteration procedure with $\mathring{R}_0=0$, we include a small perturbation $\ell$ in the definition of $\rho$ to avoid its degeneracy, whereas the function $\eta_q$ should pump energy into the system. This enables us to confirm the key bounds \eqref{3.3} at level $q+1$ in Section \ref{Section 3.4.2} and \ref{Section 3.6.1}, respectively. \\
Moreover, we point out that \eqref{3.5} and our choice of parameters $a^{\beta b}\geq 2, b\geq 7$ (cf. Section \ref{Section 3.1}) ensure
\begin{align*}
3(2\pi)^3\eta_q &\geq  \Theta^{-2} e \left(\frac{3}{4}\delta_{q+1}-\delta_{q+2}\right) \geq \Theta^{-2} e (a^{-2\beta b (b-1)}\delta^b_{q+1}-\delta_{q+2})=0,
\end{align*}
which entails
\begin{align*}
\rho \geq 2 \ell \quad \text{ and } \quad \rho \geq 2\|\mathring{R}_\ell\|_F.
\end{align*}
As a consequence $\Id-\frac{\mathring{R}_\ell}{\rho}$ fulfills the condition $\big\|\Id-\frac{\mathring{R}_\ell}{\rho}-\Id\big\|_F\leq \frac{1}{2}$ in the geometric \thref{Lemma 3.2}, so that the amplitude functions are actually well-defined.\\
Now we would already like to sum up some properties of these functions here.
 
\begin{lemma}\thlabel{Lemma 3.3}
The amplitude functions enjoy the following bounds
\begin{subequations} \label{3.8}
\begin{align}
&\|a_{(\xi)}\|_{C_tL^2}\lesssim \frac{M}{4 |\Lambda|} \delta_{q+1}^{1/2}m^{3/8}\bar{e}^{1/2},\label{3.8a}\\
&\|a_{(\xi)}\|_{C_{t,x}^N}\lesssim \frac{M}{4 |\Lambda| } \ell^{-8-7N} \delta_{q+1}^{1/2} m^{3/8}\bar{e}^{1/2}\label{3.8b}
\end{align}
\end{subequations}
for any $t \in (-\infty, \tau]$ and $N\geq 0$.
\end{lemma}
\noindent The proofs of this Lemma and of \thref{Lemma 3.6} are collected in Section \ref{Section 3.4.1}.\\
\noindent
In view of \cite{HZZ19} one might heuristically think that it makes sense to define the amplitude function as $a_{(\xi)}(t,x,\omega):=a_{\xi,q+1}(t,x,\omega):=(2\pi)^{3/2} \rho^{1/2}(t,x,\omega)\gamma_\xi\left(\Id-\frac{\mathring{R}_\ell(t,x,\omega)}{\rho(t,x,\omega)}\right)$ with $\rho(t,x,\omega):=$ $2\Theta_\ell^{-1}\sqrt{\ell^2+\|\mathring{R}_\ell(t,x,\omega)\|_F^2}+\Theta_\ell^{-1}\eta_\ell(t)$ and $\eta_q(t,\omega):= \frac{1}{3 (2\pi)^3}\Big[e(t)(1-\delta_{q+2})-\|v_q(t)\|^2_{L^2}\Big]$. In this case the first statement of \thref{Theorem 1.2} remains true but there would appear several problems in order to deduce the energy equality later on.\\
The factor $\Theta^{-2}$ in the definition of $\eta_q$ is thereby essential to make use of \eqref{3.5}, so that we can derive practical bounds for amplitude functions later on (cf. e.g. \eqref{3.21}), whereas $\Theta_\ell$ in front of $\eta_\ell$ in the definition of $\rho$ is needed to get a suitable cancellation in the first term of \eqref{3.30}.

\subsubsection{Intermittent Jets} \label{Section 3.3.2.2}
Let us now proceed with the construction of the intermittent jets. To this end consider two smooth functions
\begin{align*}
\Phi \colon \mathbb{R}^2 \to \mathbb{R}, \quad  \quad \psi \colon \mathbb{R} \to \mathbb{R}
\end{align*}
with support in a ball of radius $1$ and center $0$, where $\Phi$ should solve the Poisson equation $\phi=-\Delta \Phi$. We require $\phi$ and $\psi$ to admit the normalizations
\begin{align}\label{3.9}
\frac{1}{4 \pi^2} \int_{B_1(0)} \phi^2(x_1,x_2) \,d(x_1,x_2)=1, \qquad \qquad \qquad \frac{1}{2\pi} \int_{B_1(0)} \psi^2(x_3) \,dx_3=1
\end{align}
and $\psi$ to have zero mean. Note that Green's identity implies $\int_{\mathbb{T}^2} \phi\,dx=\int_{\mathbb{T}^2}-\Delta \Phi \,dx=0$ as well. \\
Moreover, we define

\begin{align*}
\lambda:=\lambda_{q+1}, \qquad r_\perp:=\lambda_{q+1}^{-6/7}, \qquad r_\parallel:=\lambda_{q+1}^{-4/7}, \qquad \mu:=\frac{r_\parallel}{r_\perp}\lambda_{q+1}=\lambda_{q+1}^{9/7}, 
\end{align*}
so that we find by our choice of parameters (c.f. Section \ref{Section 3.1})
\begin{align*}
0<\lambda^{-1} \ll r_\perp \ll r_\parallel \ll 1,\quad \quad \quad \lambda r_\perp \in \mathbb{N},\quad \quad \quad \mu>0.
\end{align*}
Since the rescaled cut-off functions
\begin{align*}
\phi_{r_\perp}(x_1,x_2):=\frac{1}{r_\perp}\phi\left(\frac{x_1}{r_\perp}, \frac{x_2}{r_\perp} \right), \hspace{0.15cm} \Phi_{r_\perp}(x_1,x_2):=\frac{1}{r_\perp}\Phi\left(\frac{x_1}{r_\perp}, \frac{x_2}{r_\perp} \right), \hspace{0.15cm} \psi_{r_\parallel}(x_3):= \frac{1}{r_\parallel^{1/2}}\psi\left(\frac{x_3}{r_\parallel}\right)
\end{align*}
remain compactly supported, we will henceforth identify them with their $\mathbb{T}^2, \mathbb{T}^2$ and $\mathbb{T}$-periodic versions.\\
The vectors $\xi \in \Lambda$ in the geometric \thref{Lemma 3.2} are used to construct the building blocks \eqref{3.10} of our intermittent jets. Strictly speaking, let us select $\left(\alpha_\xi\right)_{\xi \in \Lambda} \subseteq \mathbb{R}^3$ in such a way that
\begin{align*}
\sqrt{|\alpha_{\xi}\cdot A_{\xi}-\alpha_{\xi^\prime}\cdot  A_{\xi^\prime}-2 \pi z_1|^2+|\alpha_{\xi}\cdot (\xi \times A_{\xi})-\alpha_{\xi^\prime}\cdot (\xi^\prime \times A_{\xi^\prime})-2 \pi z_2|^2}>\frac{2}{n_\ast \lambda} 
\end{align*}
for each $\xi,\xi^\prime \in \Lambda$ and $z_1,z_2 \in \mathbb{Z}$, which forces the families $\left(\phixi\right)_{\xi \in \Lambda}$ and $\left(\Phixi\right)_{\xi \in \Lambda}$ given by 
\begin{subequations}\label{3.10}
\begin{align} 
\phixi(x)&:=\phi_{\xi,r_\perp,\lambda}(x):= \phi_{r_\perp}(n_\ast r_\perp \lambda(x-\alpha_\xi)\cdot A_\xi,n_\ast r_\perp\lambda(x-\alpha_\xi)\cdot(\xi \times A_\xi)),\label{3.10a}\\
\Phixi(x)&:=\Phi_{\xi,r_\perp,\lambda}(x):= \Phi_{r_\perp}(n_\ast r_\perp \lambda(x-\alpha_\xi)\cdot A_\xi,n_\ast r_\perp\lambda(x-\alpha_\xi)\cdot(\xi \times A_\xi)),\label{3.10b}\\
\psixi(t,x)&:=\psi_{\xi,r_\perp,r_\parallel,\lambda,\mu}(t,x):=\psi_{r_\parallel}(n_\ast r_\perp \lambda (x\cdot \xi+\mu t)) \label{3.10c}
\end{align}
\end{subequations}
to have mutually disjoint support. Here we consider $\psixi$ at any time $t \in \mathbb{R}$. The vector $A_\xi \in \mathbb{S}^2 \cap \mathbb{Q}^3$ should be orthogonal to $\xi$, so that $\{\xi,A_\xi,\xi \times A_\xi\}\subseteq \mathbb{S}^2\cap \mathbb{Q}^3$ forms an orthonormal basis for $\mathbb{R}^3$ and $n_\ast\in \mathbb{N}$ denotes the least common multiple of the denominators of the rational numbers $\xi_i, (A_\xi)_i$ and $(\xi \times A_\xi)_i,\,  i=1,2,3$, in other words $\big\{n_\ast \xi, n_\ast A_\xi,n_\ast \big(\xi \times A_\xi \big)\big\}\subseteq \mathbb{Z}^3$.\\
With these preparations in hand, we introduce the intermittent jet
\begin{align*}
W_{(\xi)}(t,x):=W_{\xi,r_\perp,r_\parallel,\lambda, \mu}(t,x):=\xi \psixi(t,x)\phixi(x)
\end{align*}
and its incompressibility corrector
\begin{align} \label{3.11}
W_{(\xi)}^{(c)}(t,x):=\frac{1}{n_\ast^2 \lambda^2} \nabla \psixi(t,x) \times \curl{(\Phixi (t,x)\xi)}=\curl{\curl{V_{(\xi)}(t,x)}}-W_{(\xi)}(t,x),
\end{align}
where $V_{(\xi)}(t,x):=\frac{1}{n_\ast^2\lambda^2}\xi \psixi(t,x)\Phixi(x)$, so that $W_{(\xi)}+W_{(\xi)}^{(c)}$ becomes divergence free.
Their spatial support is then contained in some cylinders of radius $\frac{r _\parallel+r_\perp}{n_\ast r_\perp \lambda}$ and axis being the line passing through $(-\mu t, \alpha_\xi \cdot A_\xi, \alpha_\xi \cdot (\xi \times A_\xi))^T$ with direction $\xi$. More precisely one has
\begin{align*}
&\supp (\Wxi(t,\cdot))\subseteq B_\frac{r_\parallel+r_\perp}{n_\ast r_\perp \lambda}(0)+ 
\left(\begin{matrix}
-\mu t\\ \alpha_\xi \cdot A_\xi \\ \alpha_\xi \cdot (\xi \times A_\xi)
\end{matrix} \right)
+\{s \xi\}_{s \in \mathbb{R}} + 2\pi \mathbb{Z}^3
\end{align*}
at each time $t\in \mathbb{R}$.
So by possibly shifting $\alpha_{\xi^\prime}$ in direction of $A_\xi, \xi \times A_\xi$, \thref{Lemma A.2} guarantees that their supports are still disjoint for all distinct $\xi,\, \xi^\prime \in \Lambda$. As a consequence of the orthogonal directions of oscillations for the functions defined in \eqref{3.10}, we may deduce
\begin{lemma}  \thlabel{Lemma 3.4}
The building blocks $\psixi$ and $\phixi$ obey
\begin{align}
\|D^\alpha \partial_t^N \psi_{(\xi)}^n D^\beta  \phi_{(\xi)}^m\|_{CL^p}= (2\pi)^{-3/p}\|D^\alpha \partial_t^N \psi_{(\xi)}^n\|_{CL^p} \|D^\beta \phi_{(\xi)}^m\|_{L^p} \label{3.12}
\end{align}
for each $n,m,N \in \mathbb{N}_0$, $p \in [1,\infty)$ and all multi-indices $\alpha,\beta \in \mathbb{N}_0^3$.
\end{lemma}
\noindent as well as the following fundamental bounds
\begin{lemma} \thlabel{Lemma 3.5}
For any $N,M \in \mathbb{N}_0 $ and $p \in [1,\infty]$ it holds
\begin{subequations}\label{3.13}
\begin{align}
\sum_{|\alpha|\leq N}\|D^\alpha\partial_t^M\psixi\|_{CL^p} &\lesssim r_{\parallel}^{1/p-1/2}\left(\frac{r_\perp \lambda}{r_\parallel}\right)^N \left(\frac{r_\perp \lambda \mu}{r_\parallel }\right)^M\label{3.13a}\\
\sum_{|\alpha|\leq N}\|D^\alpha\phixi\|_{L^p}+\sum_{|\beta|\leq N}\|D^\beta\Phixi\|_{L^p}&\lesssim r_\perp^{2/p-1}\lambda^N \label{3.13b}\\
\sum_{|\alpha|\leq N} \|D^\alpha\partial_t^M W_{(\xi)}\|_{CL^p}+\frac{r_\parallel}{r_\perp}\sum_{|\beta|\leq N} \|D^\beta\partial_t^M W_{(\xi)}^{(c)}&\|_{CL^p}+\lambda^2 \sum_{|\gamma|\leq N} \|D^\gamma \partial_t^M V_{(\xi)}\|_{CL^p}\notag\\ &\lesssim r_\perp^{2/p-1}r_\parallel^{1/p-1/2}\lambda^N \left(\frac{r_\perp \lambda \mu}{r_\parallel}\right)^M\label{3.13c},
\end{align}
\end{subequations}
where the implicit constants merely depend on $M,\, N$ and $p$.\\
In the special case $N=M=0,\, p=2$ \thref{Lemma 3.4} and the normalizations in \eqref{3.9} even entail $\|\Wxi\|_{CL^2}=1$.
\end{lemma}
\vspace{.5cm} \noindent
Based on this, we are now able to define the principle part
\begin{align*}
\wpr:= \Theta^{-1/2}_\ell\sum_{\xi \in \Lambda} \axi\Wxi,
\end{align*}
the temporal corrector
\begin{align*}
\wt:&= - \mu^{-1} \sum_{\xi \in \Lambda} \mathbb{P}\mathbb{P}_{\neq 0} \left(a^2_{(\xi)}  \psi^2_{(\xi)} \phi^2_{(\xi)} \xi \right),
\end{align*}
which will provide a better handling of the oscillation error later on, and the incompressibility corrector
\begin{align*}
\wc:&= \Theta^{-1/2}_\ell\sum_{\xi \in \Lambda} \Big( \curl \left(\nabla \axi \times \Vxi \right)+ \nabla \axi \times \curl \Vxi+\axi\Wcxi\Big),
\end{align*}
whose purpose is to ensure that $\wpr+\wc$ is divergence free and to have zero mean. In fact, the expression
\begin{align}\label{3.14}
\wpr+\wc= \Theta^{-1/2}_\ell\sum_{\xi \in \Lambda} \curl \curl \left(\axi \Vxi \right) 
\end{align}
can be easily verified by a direct computation, so that $\wpr+\wc$ is obviously divergence free and since $\axi\Vxi$ is a smooth function with periodic boundary conditions, it has additionally zero mean. Notably, these properties carry over to the total perturbation
\begin{align*}
w_{q+1}:=\wpr+\wc+\wt
\end{align*}
and we may bound each part of it as follows. 

\begin{lemma}\thlabel{Lemma 3.6}
At any time $t \in (-\infty,\tau]$, each component of the perturbation $w_{q+1}$ can be estimated
\begin{itemize}
\item[a)] in $C_tL^p$ for any $p \in (1,\infty)$ as
\begin{subequations}\label{3.15}
\begin{align}
&\|\wpr\|_{C_tL^p}\lesssim  \frac{M}{4|\Lambda|} \ell^{-8} \delta_{q+1}^{1/2} m \bar{e}^{1/2}  r_\perp^{2/p-1}r_\parallel^{1/p-1/2}, \label{3.15a}\\
&\|\wc\|_{C_tL^p}\lesssim \frac{M}{4|\Lambda|} \ell^{-22} \delta_{q+1}^{1/2}m\bar{e}^{1/2} r_\perp^{2/p} r_\parallel^{1/p-3/2},\label{3.15b}\\
&\|\wt\|_{C_tL^p}\lesssim \left(\frac{M}{4|\Lambda|}\right)^2\ell^{-16}\delta_{q+1}m^2\bar{e} r_\perp^{2/p-1}r_\parallel^{1/p-2} \lambda_{q+1}^{-1}. \label{3.15c}
\end{align}
In the specific case $p=2$ the principle part admits the stronger bound
\begin{align}
\|\wpr\|_{C_tL^2}\lesssim \frac{M}{4|\Lambda|} \delta_{q+1}^{1/2}m\bar{e}^{1/2} \label{3.15d}.
\end{align}
\end{subequations}

\item[b)] in $C_{t,x}^1$ as
\begin{subequations} \label{3.16}
\begin{align}
&\|\wpr\|_{C_{t,x}^1}\lesssim \frac{M}{4|\Lambda|}  \ell^{-16}\delta_{q+1}^{1/2}m\bar{e}^{1/2}r_\perp^{-1}r_\parallel^{-1/2}\lambda_{q+1}^2,\label{3.16a}\\
&\|\wc\|_{C_{t,x}^1}\lesssim  \frac{M}{4|\Lambda|}  \ell^{-30} \delta_{q+1}^{1/2}m\bar{e}^{1/2} r_\parallel^{-3/2}\lambda_{q+1}^2,\label{3.16b}\\
&\|\wt\|_{C_{t,x}^1}\lesssim \bigg(\frac{M}{4 |\Lambda|}\bigg)^2 \ell^{-30}\delta_{q+1}m^2\bar{e}r_\perp^{-1}r_\parallel^{-2} \lambda_{q+1}^2. \label{3.16c}
\end{align}
\end{subequations}

\item[c)] in $C_tW^{1,p}$ for any $p\in (1,\infty)$ as
\begin{subequations} \label{3.17}
\begin{align}
\|\wpr+\wc\|_{C_tW^{1,p}}&\lesssim\frac{M}{4|\Lambda|} \ell^{-29}\delta_{q+1}^{1/2} m\bar{e}^{1/2} r_\perp^{2/p-1}r_\parallel^{1/p-1/2}\lambda_{q+1}, \label{3.17a}\\
\|\wt\|_{C_tW^{1,p}} &\lesssim \left(\frac{M}{4|\Lambda|}\right)^2 \ell^{-23}\delta_{q+1}m^2\bar{e} r_\perp^{2/p-1} r_\parallel^{1/p-2}. \label{3.17b}
\end{align}
\end{subequations}
\end{itemize}
\end{lemma}

\noindent The proof is postponed to Section \ref{Section 3.4.1}.\\
 It is worth mentioning that the factor $\Theta_\ell^{-1/2}$ in the principle part $\wpr$ of the perturbation is needed to establish \eqref{3.31}, which in turn is essential to deduce a handy expression of the oscillation error \eqref{3.43}. We also define the incompressibility corrector $\wc$ with the factor $\Theta_\ell^{-1/2}$ ahead, in order to guarantee \eqref{3.14}, whereas the temporal corrector $\wt$ does not contain it.

\boldmath
\section{Inductive Estimates for $v_{q+1}$}\label{Section 3.4}
\unboldmath
So far we have developed a sequence $\left(v_q\right)_{q \in \mathbb{N}_0}$, solving \eqref{3.2} on the level $q \in \mathbb{N}_0$ and with the corresponding Reynolds stress constructed below, also on the level $q+1$. As we will see in Section \ref{Section 3.4.4} this sequence converges in $C\left((-\infty,\tau];L^2\left(\mathbb{T}^3\right)\right)$, so that the limit function will be our desired weak solution to \eqref{1.2} on $[0,\tau]$. However, let us take one step after the other. We start by aiming $\left(v_q\right)_{q \in \mathbb{N}_0}$ to admit the bounds \eqref{3.3a} and \eqref{3.3b} on the level $q+1$.\\

\subsection{Preparations} \label{Section 3.4.1}
For this purpose we need to control the amplitude functions and intermittent jets, so that we go back to \thref{Lemma 3.3}, and \thref{Lemma 3.6}. The key ingredient of the proof of \thref{Lemma 3.3} is the ensuing result from \cite{BDLIS15}, p.163, where we set $D_t^\alpha:=\partial^{\alpha_1}_{x_1}\ldots\partial^{\alpha_n}_{x_n} \partial_t^{\alpha_{n+1}}$ for all multi-indices  $\alpha=(\alpha_1,\ldots,\alpha_{n+1})\in \mathbb{N}_0^{n+1}$.

\begin{prop}\thlabel{Proposition 3.7}
For any $k \in \mathbb{N}$ the composition of  $f \in C^\infty\left(\Ima(u); \mathbb{R} \right)$ and $u \in C^\infty \left( \mathbb{R}\times\mathbb{R}^n; \mathbb{R}^N\right)$  can be estimated as 
\setlength\jot{-0.4cm}
\begin{align*}
\sum_{0<|\alpha|\leq k}& \|D_t^\alpha f(u)\|_{C_tL^\infty} \\
& \hspace{-0.5cm}\lesssim \sum_{0<|\alpha|\leq k} \Bigg(\max_{|\beta|=1} \|D_t^\beta f \|_{C_tL^\infty} \max_{|\beta|=|\alpha|}\|D_t^\beta u\|_{C_tL^\infty}+\sum_{j=0}^{|\alpha|-1} \max_{|\beta|=j+1}\|D_t^{\beta}f \|_{C_tL^\infty} \max_{|\beta|=1} \|D_t^\beta u\|^{|\alpha|}_{C_tL^\infty}\Bigg). 
\end{align*}
\end{prop}
 \,\\

\begin{proof}[Proof of \thref{Lemma 3.3}]~\\
\underline{\eqref{3.8a}:} Keeping in mind  that $\alpha>4\beta b^2>\frac{4}{3} \beta$ and $\bar{e}> 4$ imply
\begin{align}\label{3.18}
\ell  \leq \delta_{q+1}\bar{e},
\end{align} 
and taking \eqref{3.7}, \eqref{3.3c}, $\eta_q \geq 0$ and \eqref{3.5} into account, leads to 
\begin{align*}
\|a_{(\xi)}\|_{C_tL^2}& \lesssim \sup_{\substack{R \in \mathbb{R}_{\text{sym}}^{3\times 3}\\ \|R-\Id\|_F\leq 1/2}}|\gamma_\xi (R)| \, \|\rho\|_{C_tL^1}^{1/2}\\
&\lesssim \frac{M}{4|\Lambda|} \left( \delta_{q+1}\bar{e} +\|\mathring{R}_q\|_{C_tL^1} +m^{1/4} \|\eta_q\|_{C_t} \right)^{1/2}\\
&\lesssim \frac{M}{4|\Lambda|}\left(\delta_{q+1}\bar{e} + \delta_{q+2}m^{2/4} \bar{e}+m^{1/4} \|\Theta^{-2} e-\|v_q\|_{L^2}^2\|_{C_t}\right)^{1/2}\\
&\lesssim \frac{M}{4|\Lambda|}\left(\delta_{q+1}\bar{e} + \delta_{q+2}m^{2/4}\bar{e}+\delta_{q+1}m^{3/4}\bar{e}\right)^{1/2}\\
&\lesssim \frac{M}{4 |\Lambda|}\delta_{q+1}^{1/2}m^{3/8}\bar{e}^{1/2}.
\end{align*}
\noindent \underline{\eqref{3.8b}:} To estimate the $C_{t,x}^N$-norm of $a_{(\xi)}$, we make use of Leibniz's rule 
\begin{align*}
\|a_{(\xi)}\|_{C_{t,x}^N} &\lesssim \sum_{|\alpha|\leq N} \sum_{\beta \leq \alpha} \Big\|D_t^\beta\big( \rho ^{1/2}\big)\Big\|_{C_tL^\infty} \bigg\|D_t^{\alpha-\beta} \gamma_\xi \bigg( \Id- \frac{\mathring{R}_\ell}{\rho}\bigg)\bigg\|_{C_tL^\infty}\\
&\lesssim \sum_{k\leq N} \big\|\rho ^{1/2}\big\|_{C^k_{t,x}} \bigg\|\gamma_\xi \bigg( \Id- \frac{\mathring{R}_\ell}{\rho}\bigg)\bigg\|_{C_{t,x}^{N-k}}.
\end{align*}
Let us proceed with a bound for $\big\|\rho ^{1/2}\big\|_{C_{t,x}^k}$. We will verify

\begin{align} \label{3.19}
\big\|\rho^{1/2}\big\|_{C_{t,x}^k} \lesssim \left\{\begin{array}{ll} \ell^{-2} \delta^{1/2}_{q+1}m^{3/8}\bar{e}^{1/2}, & \text{if } k=0, \\
         \ell^{1-7k}\delta_{q+1}^{1/2}m^{3/8}\bar{e}^{1/2}, & \text{if } k>0. \end{array}\right. 
\end{align}

\noindent For this purpose we need
\begin{adjustwidth}{20pt}{20pt}
\begin{flalign}\label{3.20}
\text{\underline{1.Claim:}}\ \  &\|\mathring{R}_\ell\|_{C_{t,x}^k} \lesssim \ell^{-4-k} \delta_{q+2}m^{1/2} \bar{e}& 
\end{flalign}
\begin{proof}
Owing to the embedding $W^{4,1}\subseteq L^\infty$ (cf. \cite{Ev10}, p.284, Theorem 6), Fubini's theorem and \eqref{3.3c} we may deduce 
\begin{align*}
\|\mathring{R}_\ell\|_{C_{t,x}^k} &\lesssim \sum_{n+|\alpha| \leq k} \|\partial_t^n D^\alpha \mathring{R}_\ell\|_{C_tW^{4,1}}\\
& \lesssim \sum_{\substack{n+|\alpha| \leq k\\ |\beta| \leq 4}}\Big\|\Big( \mathring{R}_q \ast_t \partial_t^n \varphi_\ell \Big) \ast_x D^{\alpha+\beta} \phi_\ell\Big\|_{C_tL^1}\\
& \lesssim \sum_{\substack{n+|\alpha| \leq k\\ |\beta| \leq 4}} \sup_{s \in (-\infty,t]} \int_{\mathbb{T}^3} \int_0^\ell \int_{|y|\leq \ell} |\mathring{R}_q(s-u,x-y)| \ell^{-n} \frac{1}{\ell} \varphi^{(n)} \left( \frac{u}{\ell} \right)\\
&\hspace{6.4cm}\cdot \ell^{-(|\alpha|+|\beta|)} \frac{1}{\ell^3} D_{\frac{y}{\ell}}^{\alpha+\beta} \phi \left(\frac{y}{\ell}\right) \,dy \,du \,dx\\
&\lesssim  \ell^{-4-k}  \|\mathring{R}_q\|_{C_tL^1} \sum_{\substack{n+|\alpha| \leq k\\ |\beta| \leq 4}}  \int_0^1  \varphi^{(n)} \left( u \right)\,du \int_{|y|\leq 1} D^{\alpha+\beta} \phi(y) \,dy \\
&\lesssim \ell^{-4-k} \delta_{q+2} m^{2/4}\bar{e}.
\end{align*}
\end{proof}
\end{adjustwidth}
\vspace{0.5cm}
to derive
\begin{adjustwidth}{20pt}{20pt}
\begin{flalign} \label{3.21}
\text{\underline{2.Claim:}}\ \  &\|\rho\|_{C_{t,x}^k} \lesssim \left\{\begin{array}{ll} \ell^{-4} \delta_{q+1}m^{3/4}\bar{e}, & \text{if } k=0, \\
         \ell^{2-7k}\delta_{q+1}m^{3/4}\bar{e}, & \text{if } k>0.\end{array}\right. &
\end{flalign}
\begin{proof}
Combining \eqref{3.18}, \eqref{3.20} and $\eta_q\geq 0$ with \eqref{3.5}, results in
\begin{align*}
\|\rho\|_{C_{t,x}^0} \lesssim \delta_{q+1}\bar{e}+\ell^{-4}\delta_{q+2}m^{1/2}\bar{e} +m^{1/4}\big\|\Theta^{-2} e-\|v_q\|_{L^2}^2 \big\|_{C_t}\lesssim \ell^{-4}\delta_{q+1} m^{3/4}\bar{e}.
\end{align*}
For $k>0$ we introduce the smooth function
\begin{align*}
f \colon \mathbb{R} \to \mathbb{R}, \quad f(z):= \sqrt{\ell^2+z^2},
\end{align*}
satisfying 
\begin{align}
|f^{(k)}(z)| \lesssim \ell^{-k+1}. \label{3.22}
\end{align}
Keeping \eqref{3.20} and \eqref{3.1b} in mind, \thref{Proposition 3.7} teaches us
\begin{align*}
\sum_{0<|\alpha|\leq k}\Big\|D_t^\alpha\Big( \sqrt{\ell^2+\|\mathring{R}_\ell\|_F^2}\Big)\Big\|_{C_tL^\infty} \lesssim \ell^{2-7k} \delta_{q+2}m^{1/2}\bar{e}.
\end{align*}
Moreover, we use Leibniz's rule together with the fact $\eta_q \geq 0$ and \eqref{3.5} to calculate 
\begin{align*}
&\sum_{0<|\alpha|\leq k} \|D_t^\alpha (\Theta_\ell \eta_\ell)\|_{C_tL^\infty} \\
&\hspace{.3cm}\lesssim \sum_{0<n\leq k} \sum_{j\leq n}\|\partial_t^{n-j} \Theta_\ell\|_{C_t}\|\partial_t^j\eta_\ell\|_{C_t}\\
&\hspace{.3cm}\lesssim \sum_{0<n\leq k}  \sum_{j\leq n}m^{1/4}\ell^{-n+j+i} \int_0^1 \varphi^{(n-j)}(u)\,du\,   \|\eta_q\|_{C_t} \ell^{-j} \int_0^1 \varphi^{(j)}(u)\,du \\
&\hspace{.3cm}\lesssim \ell^{-k}m^{1/4} \|\Theta^{-2}e-\|v_q\|^2_{L^2}\|_{C_t}\\
&\hspace{.3cm}\lesssim \ell^{-k} \delta_{q+1} m^{3/4}\bar{e}.
\end{align*}
As a result of these three bounds
\begin{align*}
\|\rho\|_{C_{t,x}^k} &\lesssim \|\rho\|_{C_{t,x}^0}+\sum_{0<|\alpha|\leq k}\Big\|D_t^\alpha \Big(\sqrt{\ell^2+\|\mathring{R}_\ell\|^2}\Big)\Big\|_{C_tL^\infty} +\sum_{0<|\alpha|\leq k} \|D_t^\alpha (\Theta_\ell\eta_\ell)\|_{C_tL^\infty}\\
&\lesssim \ell^{2-7k}\delta_{q+1}m^{3/4}\bar{e}.
\end{align*}
\end{proof}
\end{adjustwidth}
~\\
In order to find a bound for $\|\rho^{1/2}\|_{C_{t,x}^k}$ we intend to make use of \thref{Proposition 3.7} again. This time, however, applied to the function
\begin{align*}
\widetilde{f} \colon \Ima (\rho) \to \mathbb{R}, \quad \widetilde{f}(z)=z^{1/2}.
\end{align*}
and $\rho$. Taking into account that $\rho \geq  \ell$ entails $|\widetilde{f}^{(k)}(z)|\lesssim |z|^{1/2-k}\lesssim \ell^{1/2-k}$, we deduce from \eqref{3.21} and \eqref{3.1b} that
\begin{align*}
\sum_{0<|\alpha|\leq k}\|D_t^\alpha \big( \rho^{1/2}\big)\|_{C_tL^\infty} &\lesssim \ell^{-1/2} \ell^{2-7k} \delta_{q+1} m^{3/4}\bar{e}+\ell^{1/2-k}\left(\ell^{2-7} \delta_{q+1}m^{3/4}\bar{e}\right)^k\\
&\lesssim \ell^{1-7k}\delta_{q+1}^{1/2} \bar{e}^{1/2}
\end{align*}
holds, provided $k>0$. As a consequence and accordingly \eqref{3.21}
\begin{align*}
\|\rho^{1/2}\|_{C_{t,x}^k} \lesssim \left(\ell^{-4}\delta_{q+1}m^{3/4}\bar{e}\right)^{1/2}+\ell^{1-7k}\delta_{q+1}^{1/2} \bar{e}^{1/2}\lesssim \ell^{1-7k}\delta_{q+1}^{1/2} m^{3/8}\bar{e}^{1/2}.
\end{align*}
~\\
Let us now have a closer look at $\Big\|\gamma_\xi \left( \Id- \frac{\mathring{R}_\ell}{\rho}\right) \Big\|_{C_{t,x}^{N-k}}$. Our aim is to verify
\begin{align} \label{3.23}
\Big\|\gamma_\xi \left( \Id- \frac{\mathring{R}_\ell}{\rho}\right) \Big\|_{C_{t,x}^{N-k}} \lesssim \left\{\begin{array}{ll} \frac{M}{4|\Lambda|}, & \text{if } k=N, \\
\frac{M}{4|\Lambda|}\ell^{-6-7(N-k)}  , & \text{if } 0\leq k<N.\end{array}\right. 
\end{align}
~\\
The case $k=N$ is trivial, whereas \thref{Proposition 3.7} and \eqref{3.7} again imply
\begin{align*}
& \Big\|\gamma_\xi \left( \Id- \frac{\mathring{R}_\ell}{\rho}\right) \Big\|_{C_{t,x}^{N-k}}-\Big\|\gamma_\xi \left( \Id- \frac{\mathring{R}_\ell}{\rho}\right) \Big\|_{C_tL^\infty}\\
&\hspace{0.5cm}\lesssim \sum_{0<|\alpha|\leq N-k} \frac{M}{8 |\Lambda|(1+8 \pi^3)^{1/2}} \left(\|\Id\|_{C_{t,x}^{N-k}}+ \bigg\|\frac{\mathring{R}_\ell}{\rho} \bigg\|_{C^{N-k}_{t,x}} +\|\Id\|_{C_{t,x}^1}^{|\alpha|}+\bigg\|\frac{\mathring{R}_\ell}{\rho} \bigg\|^{|\alpha|}_{C^1_{t,x}} \right)
\end{align*}
and we assert
\begin{adjustwidth}{20pt}{20pt}
\begin{flalign*} 
\text{\underline{3.Claim:}}\ \  &\bigg\| \frac{\mathring{R}_\ell}{\rho}\bigg\|_{C_{t,x}^{N-k}} \lesssim  \ell^{-6-7(N-k)}.&
\end{flalign*}
\begin{proof}
Thanks to Leibniz's formula
\begin{align*}
\bigg\| \frac{\mathring{R}_\ell}{\rho}\bigg\|_{C_{t,x}^{N-k}} &\lesssim \sum_{|\alpha|\leq N-k} \sum_{\beta\leq \alpha} \|D_t^\beta \mathring{R}_\ell\|_{C_tL^\infty} \bigg\|D_t^{\alpha-\beta}\frac{1}{\rho}\bigg\|_{C_tL^\infty}\lesssim \sum_{j\leq N-k} \|\mathring{R}_\ell\|_{C_{t,x}^{j}} \bigg\|\frac{1}{\rho} \bigg\|_{C_{t,x}^{N-k-j}}. 
\end{align*}
Applying \thref{Proposition 3.7} to the functions
\begin{align*}
\hat{f} \colon \Ima(\rho) \to \mathbb{R}, \quad \hat{f}(z)=\frac{1}{z}, \quad |\hat{f}^{(k)}(z)|\lesssim |z|^{-1-k}\lesssim \ell^{-1-k}
\end{align*}
and $\rho$ yields according to \eqref{3.21} and \eqref{3.1b}
\begin{align*}
\sum_{0<|\alpha|\leq N-k-j} \bigg\| D_t^\alpha \frac{1}{\rho}\bigg\|_{C_tL^\infty}&\lesssim \ell^{-7(N-k-j)}\delta_{q+1}m^{3/4} \bar{e}+\ell^{-1-6(N-k-j)} \delta_{q+1}m^{\frac{3(N-k-j)}{4}}\bar{e}^{N-k-j}\\
& \lesssim \ell^{-1-7(N-k-j)}.
\end{align*}
Consequently
\begin{align*}
\bigg\| \frac{\mathring{R}_\ell}{\rho}\bigg\|_{C_{t,x}^{N-k}}&  \lesssim \sum_{j\leq N-k-1} \|\mathring{R}_\ell\|_{C_{t,x}^{j}} \left(\bigg\|\frac{1}{\rho}\bigg\|_{C_tL^\infty}+\sum_{0<|\alpha|\leq N-k-j} \bigg\| D_t^\alpha \frac{1}{\rho}\bigg\|_{C_tL^\infty} \right)\\
&\hspace{1.7cm}+\|\mathring{R}_\ell\|_{C_{t,x}^{N-k}} \bigg\|\frac{1}{\rho}\bigg\|_{C_{t,x}^0}\\
& \lesssim \sum_{j\leq N-k-1} \ell^{-4-j}\delta_{q+2}m^{1/2}\bar{e}\left(\ell^{-1}+ \ell^{-1-7(N-k-j)}\right)\\
&\hspace{1.7cm}+\ell^{-4-(N-k)}\delta_{q+2}m^{1/2}\bar{e}\ell^{-1}\\
&\lesssim \ell^{-6-7(N-k)}.
\end{align*}
 The penultimate step additionally follows from \eqref{3.20} and $\rho \geq \ell$, whereas the last step also holds due to \eqref{3.1b}.\\
 \end{proof}
 
 \noindent If $k=N-1$ we obtain a stronger bound
 \begin{flalign*} \
\text{\underline{4.Claim:}}\ \  &\bigg\| \frac{\mathring{R}_\ell}{\rho}\bigg\|_{C_{t,x}^1} \lesssim \ell^{-7}.&
\end{flalign*}
\begin{proof}
Remembering that $\rho \geq  \ell$ and $\rho \geq \|\mathring{R}_\ell\|_F $ holds and taking \eqref{3.20}, \eqref{3.21} and \eqref{3.1b} into account we compute
\begin{align*}
\bigg\| \frac{\mathring{R}_\ell}{\rho}\bigg\|_{C_{t,x}^1}&\lesssim \bigg\| \frac{\mathring{R}_\ell}{\rho}\bigg\|_{C_tL^\infty}\hspace{-0.5cm}+\bigg\|\frac{\partial_t \mathring{R}_\ell}{\rho}\bigg\|_{C_tL^\infty}\hspace{-0.5cm}+ \bigg\|\frac{\mathring{R}_\ell\partial_t \rho}{\rho^2}\bigg\|_{C_tL^\infty}+ \sum_{k=1}^3 \left(\bigg\| \frac{\partial_{x_k}\mathring{R}_\ell}{\rho}\bigg\|_{C_tL^\infty}\hspace{-0.5cm}+\bigg\|\frac{\mathring{R}_\ell\partial_{x_k}\rho}{\rho^2}\bigg\|_{C_tL^\infty}\right) \\
&\lesssim \ell^{-1} \left( \|\mathring{R}_\ell\|_{C_tL^\infty}+\|\partial_t \mathring{R}_\ell\|_{C_tL^\infty}+\sum_{k=1}^3 \|\partial_{x_k}\mathring{R}_\ell\|_{C_tL^\infty} \right)\\
& \hspace{.5cm}+ \ell^{-1} \left(\|\partial_t \rho\|_{C_tL^\infty}+\sum_{k=1}^3 \|\partial_{x_k}\rho\|_{C_tL^\infty} \right)\\
& \lesssim \ell^{-1} \left(\ell^{-4-1} \delta_{q+2} m^{1/2}\bar{e} +\ell^{2-7} \delta_{q+1}m^{3/4}\bar{e}\right)\\
&\lesssim \ell^{-7}.
\end{align*}
\end{proof}
\end{adjustwidth}
\vspace{0.5cm}
As a result 
\begin{align*}
\Big\|\gamma_\xi \bigg( \Id- \frac{\mathring{R}_\ell}{\rho}\bigg) \Big\|_{C_{t,x}^{N-k}}&\lesssim  \sum_{0<|\alpha|\leq N-k} \frac{M}{8|\Lambda|(1+8\pi^3)^{1/2}}\left(1+\ell^{-6-7(N-k)}+1+\ell^{-7|\alpha|}\right)\\
&\hspace{.5cm} +\Big\|\gamma_\xi \left( \Id- \frac{\mathring{R}_\ell}{\rho}\right) \Big\|_{C_tL^\infty}\\
&\lesssim \frac{M}{4|\Lambda|}\ell^{-6-7(N-k)}.
\end{align*}
Altogether we therefore bound \eqref{3.8b} as
\begin{align*}
\|a_{(\xi)}\|_{C_{t,x}^N}&\lesssim \|\rho ^{1/2}\|_{C^0_{t,x}} \bigg\|\gamma_\xi \bigg( \Id- \frac{\mathring{R}_\ell}{\rho}\bigg)\bigg\|_{C_{t,x}^N}+\sum_{k=1}^{N-1} \|\rho ^{1/2}\|_{C^k_{t,x}} \bigg\|\gamma_\xi \bigg( \Id- \frac{\mathring{R}_\ell}{\rho}\bigg)\bigg\|_{C_{t,x}^{N-k}}\\
&\hspace{.5cm} +\|\rho ^{1/2}\|_{C^N_{t,x}} \bigg\|\gamma_\xi \bigg( \Id- \frac{\mathring{R}_\ell}{\rho}\bigg)\bigg\|_{C_{t,x}^0}\\
&\lesssim \frac{M}{4 |\Lambda|}\delta_{q+1}^{1/2}m^{3/8}\bar{e}^{1/2}\left(\ell^{-8-7N}+\ell^{-5-7N}+\ell^{1-7N} \right)\\
&\lesssim \frac{M}{4|\Lambda|}\ell^{-8-7N} \delta_{q+1}^{1/2}m^{3/8}\bar{e}^{1/2},
\end{align*}
provided $N>0$. However, the final bound is due to \eqref{3.7} and \eqref{3.19} even valid for $N=0$. The proof of \thref{Lemma 3.3} is therefore complete.
\end{proof}

\begin{proof}[Proof of \thref{Lemma 3.6}]~\\
\underline{\eqref{3.15a}:} follows readily from \eqref{3.8b} and \eqref{3.13c}.\\
\underline{\eqref{3.15b}:} Thanks to \eqref{3.8b} and \eqref{3.13c} again, we obtain
\begin{align*}
\|\wc\|_{C_tL^p} &\lesssim m^{1/8}\sum_{\xi \in \Lambda} \bigg\{ \sum_{\ell ij=1}^3 \Big\| \partial_{x_\ell} \partial_{x_i} \axi \Vxi^j\Big\|_{C_tL^p}+  \sum_{\ell ij=1}^3\|\partial_{x_i} \axi \partial_{x_\ell} \Vxi^j\|_{C_tL^p} \bigg\}\\
& \hspace{.3cm} +m^{1/8}\sum_{\xi \in \Lambda} \bigg\{ \sum_{\ell ij=1}^3\Big\|\partial_{x_\ell} \axi \partial_{ x_i} \Vxi^j \Big\|_{C_tL^p}+   \Big\|\axi \Wcxi \Big\|_{C_tL^p} \bigg\} \\
&\lesssim m^{1/8}\sum_{\xi \in \Lambda} \bigg\{ \|\axi\|_{C^2_{t,x}}\|\Vxi\|_{C_tL^p} +\|\axi\|_{C^1_{t,x}}\sum_{|\gamma|\leq 1 } \| D^\gamma \Vxi\|_{C_tL^p}\bigg\}\\
&\hspace{0.3cm}+m^{1/8}\sum_{\xi \in \Lambda}\|\axi\|_{C^0_{t,x}}\|\Wcxi\|_{C_tL^p}\\
&\lesssim \frac{M}{4 |\Lambda|} \ell^{-22} \delta_{q+1}^{1/2} m \bar{e}^{1/2} r_\perp^{2/p}r_\parallel^{1/p-3/2}.
\end{align*}

\noindent
\underline{\eqref{3.15c}:}
Remembering that $\mathbb{P}\mathbb{P}_{\neq 0}$ is bounded on $L^p$ and keeping \eqref{3.12}, \eqref{3.8b}, \eqref{3.13a} and \eqref{3.13b} in mind we may compute
\begin{align*}
\|\wt\|_{C_tL^p} 
&\lesssim \mu^{-1}   \sum_{\xi \in \Lambda} \|\axi^2\|_{C^0_{t,x}} \|\psi_{(\xi)}^2 \phi_{(\xi)}^2\|_{C_tL^p}\\
&\lesssim \mu^{-1}   \sum_{\xi \in \Lambda} \|\axi\|_{C_{t,x}^0}^2 \|\psi_{(\xi)}\|^2_{C_tL^{2p}} \|\phi_{(\xi)}\|^2_{L^{2p}}\\
&\lesssim \left( \frac{M}{4 |\Lambda|}\right)^2 \ell^{-16} \delta_{q+1} m^2 \bar{e} r_\perp^{2/p-1} r_\parallel^{1/p-2} \lambda_{q+1}^{-1}.
\end{align*}
\noindent
\underline{\eqref{3.15d}:} Moreover, if $p=2$, we obtain according to \eqref{3.8a} and \eqref{3.8b}
\begin{align*}
\sum_{|\alpha|\leq j}\| D^\alpha \axi\|_{C_tL^2} \lesssim  \frac{M}{4|\Lambda|} \delta_{q+1}^{1/2}m^{3/8}\bar{e}^{1/2}\ell^{-15j}
\end{align*}
for all $j\geq 0$. Choosing $N\in \mathbb{N}$ in such a way that $N \geq \frac{60\ln(\ell)-\ln(16)-12}{\ln(12\pi)+3-\ln(n_\ast r_\perp \lambda_{q+1})-15 \ln(\ell)}$ holds, ensures
\begin{align*}
16 \exp(12)\ell^{-60} \left( 3\frac{2 \pi}{n_\ast r_\perp\lambda_{q+1}} \,2\exp(3)\ell^{-15} \right)^N \leq 1,
\end{align*}
whereas $a\geq 3600,\, b\geq 7,\, 161\alpha <\frac{1}{7 }$ and \eqref{3.1b} imply
\begin{align*}
3\frac{2 \pi }{n_\ast r_\perp\lambda_{q+1}} \, \ell^{-15}\leq \frac{1}{41}.
\end{align*}
Here we need in particular that $a\geq 3600$. Alternatively one could also chose a smaller $a$ but then we have to increase $b\in 7\mathbb{N}$.\\
So all requirements of Lemma 3.7 from \cite{BV19b}, recalled in Appendix \ref{Appendix A.1}, \thref{Lemma A.3}, are satisfied. Invoking additionally \eqref{3.13c} entails
\begin{align*}
\|\wpr\|_{C_tL^2} \lesssim m^{1/8} \frac{M}{4|\Lambda|} \delta_{q+1}^{1/2}m^{3/8}\bar{e}^{1/2}\sum_{\xi \in \Lambda} \| \Wxi\|_{C_tL^2}\lesssim \frac{M}{4|\Lambda|} \delta_{q+1}^{1/2}m\bar{e}^{1/2}.
\end{align*}
 b)\underline{\eqref{3.16a}:} follows from \eqref{3.8b}, \eqref{3.13c} and 
 \begin{align}
 \|\partial_t \Theta_\ell^{-1/2}\|_{C_t}\lesssim \|\partial_t \Theta_\ell\|_{C_t}\| \Theta_\ell^{-3/2}\|_{C_t}\lesssim \ell^{-1}m^{1/4}m^{3/8}. \label{3.24}
 \end{align} 
 Namely, 
 \begin{align*}
 \|\wpr\|_{C_{t,x}^1}
 &\lesssim \|\Theta_\ell^{-1/2}\|_{C_t}\|\Theta_\ell^{1/2}\wpr\|_{C_tL^\infty}+\|\partial_t\Theta_\ell^{-1/2}\|_{C_t}\|\Theta_\ell^{1/2}\wpr\|_{C_tL^\infty}\\
 &\hspace{.5cm}+\|\Theta_\ell^{-1/2}\|_{C_t}\|\partial_t (\Theta_\ell^{1/2}\wpr)\|_{C_tL^\infty}+\|\Theta_\ell^{-1/2}\|_{C_t}\sum_{|\alpha|= 1}\|\Theta_\ell^{1/2}D^\alpha\wpr\|_{C_tL^\infty} \\
 &\lesssim\ell^{-1}m^{5/8} \sum_{\xi \in \Lambda} \bigg( \|\axi\|_{C_tL^\infty}\|\Wxi\|_{C_tL^\infty}+ \|\partial_t \axi\|_{C_tL^\infty} \|\Wxi\|_{C_tL^\infty}\bigg)\\
 &\hspace{0.5cm}+\ell^{-1}m^{5/8}\sum_{\xi \in \Lambda}\bigg(\|\axi\|_{C_tL^\infty}\|\partial_t \Wxi\|_{C_tL^\infty}+\sum_{|\alpha|=1}\|D^\alpha\axi\|_{C_tL^\infty}\|\Wxi\|_{C_tL^\infty}\bigg)\\
&\hspace{0.5cm}+\ell^{-1}m^{5/8}\sum_{\xi \in \Lambda}\sum_{|\alpha|=1}\|\axi\|_{C_tL^\infty}\|D^\alpha\Wxi\|_{C_tL^\infty}  \\
 &\lesssim \ell^{-1}m^{5/8}\sum_{\xi \in \Lambda} \|\axi\|_{C_{t,x}^0} \bigg(\|\Wxi\|_{C_tL^\infty}+\|\partial_t\Wxi\|_{C_tL^\infty}+\sum_{|\alpha|=1} \|D^\alpha\Wxi\|_{C_tL^\infty} \bigg)\\
 &\hspace{0.5cm}+\ell^{-1}m^{5/8}\sum_{\xi \in \Lambda}\|\axi\|_{C_{t,x}^1}\|\Wxi\|_{C_tL^\infty} \\
 &\lesssim \frac{M}{4 |\Lambda|}  \ell^{-16}\delta_{q+1}^{1/2} m\bar{e}^{1/2} r_\perp^{-1}r_\parallel^{-1/2}\lambda_{q+1}^2.
  \end{align*}
\underline{\eqref{3.16b}:} We will estimate each involved term of
\begin{align*}
\|\wc\|_{C_{t,x}^1} \overset{\eqref{3.24}}{\leq} \ell^{-1}m^{5/8} \sum_{\xi \in \Lambda}\Big(\|\curl (\nabla \axi \times \Vxi)\|_{C_{t,x}^1} +\|\nabla \axi \times \curl \Vxi\|_{C_{t,x}^1}+\|\axi \Wcxi\|_{C_{t,x}^1}\Big)
\end{align*}
by using \eqref{3.8b} and \eqref{3.13c} in order of their appearance.\\
Firstly, by using the Levi-Civita-symbol,
\begin{align*}
\|\curl & (\nabla \axi \times \Vxi)\|_{C_{t,x}^1}\\
&\lesssim \sum_{\ell, k,m,i,j} \|\mathcal{E}_{\ell k m} \mathcal{E}_{ijk}\partial_{x_\ell} (\partial_{x_i} \axi V^j_{(\xi)})\vec{e}_m\|_{C_{t,x}^1}\\
& \lesssim \sum_{\ell, i,j} \|\partial_{x_\ell} \partial_{x_i} \axi V^j_{(\xi)}\|_{C_{t,x}^1}+\| \partial_{x_i} \axi \partial_{x_\ell} V^j_{(\xi)}\|_{C_{t,x}^1}\\
&\lesssim \sum_{|\alpha|\leq 2} \sum_{|\beta|\leq 1} \|D^\alpha \axi D^\beta \Vxi\|_{C_{t,x}^1}\\
&\lesssim  \sum_{|\alpha|\leq 2} \sum_{|\beta|\leq 1} \bigg\{ \|D^\alpha\axi\|_{C_tL^\infty}\|D^\beta \Vxi\|_{C_tL^\infty}+\|\partial_t D^\alpha\axi\|_{C_tL^\infty}\|D^\beta \Vxi\|_{C_tL^\infty}\bigg\}\\
&\hspace{.5cm}+\sum_{|\alpha|\leq 2} \sum_{|\beta|\leq 1}\bigg\{\|D^\alpha\axi\|_{C_tL^\infty}\|\partial_t D^\beta \Vxi\|_{C_tL^\infty}+\sum_{|\gamma|=1}\|D^{\alpha+\gamma}\axi\|_{C_tL^\infty}\|D^\beta \Vxi\|_{C_tL^\infty}\bigg\}\\
& \hspace{.5cm}+\sum_{|\alpha|\leq 2} \sum_{|\beta|\leq 1}\sum_{|\gamma|=1}\|D^\alpha\axi\|_{C_tL^\infty}\|D^{\beta+\gamma} \Vxi\|_{C_tL^\infty}\\
& \lesssim \|\axi\|_{C_{t,x}^3} \bigg( \sum_{|\beta|\leq 1} \|\partial_tD^\beta \Vxi \|_{C_tL^\infty}+ \sum_{|\beta|\leq 2} \|D^\beta \Vxi\|_{C_tL^\infty} \bigg)\\
& \lesssim \frac{M}{4|\lambda|}  \ell^{-29} \delta_{q+1}^{1/2} m^{3/8}\bar{e}^{1/2} r_\perp^{-1} r_\parallel^{-1/2} \lambda_{q+1},
\end{align*}
secondly 
\begin{align*}
\|\nabla \axi \times \curl \Vxi\|_{C_{t,x}^1}
\lesssim \frac{M}{4|\lambda|}  \ell^{-22} \delta_{q+1}^{1/2} m^{3/8} \bar{e}^{1/2} r_\perp^{-1} r_\parallel^{-1/2} \lambda_{q+1}
\end{align*}
and thirdly
\begin{align*}
\|\axi \Wcxi\|_{C_{t,x}^1} \lesssim \frac{M}{4 |\Lambda|}  \ell^{-15} \delta_{q+1}^{1/2} m^{3/8}\bar{e}^{1/2} r_\parallel^{-3/2} \lambda_{q+1}^2,
\end{align*}
verifying the claim.\\
\underline{\eqref{3.16c}:} 
Let $E_{n,\beta}\partial^n_t D^\beta w_{q+1}^{(t)}$ for any $n\in \mathbb{N}_0,\, \beta \in \mathbb{N}_0^3$ be the canonical extension of $\partial_t^n D^\beta w_{q+1}^{(t)}$ to $\mathbb{R}^3$, meaning that $E_{n,\beta}\colon W^{1,p}\left(\mathbb{T}^3 \right)\to W^{1,p}\left(\mathbb{R}^3 \right)$ denotes a linear bounded operator with \linebreak ${E_{n,\beta}\partial_t^n D^\beta w_{q+1}^{(t)}}_{|\mathbb{T}^3}=\partial_t^nD^\beta w_{q+1}^{(t)}$ almost everywhere (see \cite{Ev10}, p.268, Theorem 1 for instance). It then holds according to Morrey's inequality (see \cite{Ev10}, p.280, Theorem 4) 
\begin{align*}
\|\wt\|_{C_{t,x}^1(\mathbb{T}^3)}&\lesssim \sum_{n+|\beta|\leq 1} \|E_{n,\beta}\partial^n_t D^\beta \wt\|_{C_tC^{0,\gamma}(\mathbb{R}^3)}\\
& \lesssim \sum_{n+|\beta|\leq 1} \|E_{n,\beta}\partial^n_t D^\beta \wt\|_{C_tW^{1,p}(\mathbb{R}^3)} \\
&\lesssim \mu^{-1}  \sum_{\xi \in \Lambda} \|\mathbb{P}\mathbb{P}_{\neq 0}(a_{(\xi)}^2\psi_{(\xi)}^2\phi_{(\xi)}^2\xi)\|_{C_tW^{1,p}(\mathbb{T}^3)}\\
&\hspace{.5cm}+\mu^{-1}  \sum_{\xi \in \Lambda}\|\partial_t\mathbb{P}\mathbb{P}_{\neq 0}(a_{(\xi)}^2\psi_{(\xi)}^2\phi_{(\xi)}^2\xi)\|_{C_tW^{1,p}(\mathbb{T}^3)}  \\
& \hspace{.5cm} + \mu^{-1}  \sum_{\xi \in \Lambda}\sum_{|\beta|=1}\|D^\beta\mathbb{P}\mathbb{P}_{\neq 0} (a_{(\xi)}^2 \psi_{(\xi)}^2\phi_{(\xi)}^2\xi)\|_{C_tW^{1,p}(\mathbb{T}^3)}
\end{align*}
for any $p \in \big(3, \infty\big)$ and $\gamma:=1-\frac{3}{p}$.\\
Note that $x\mapsto (\axi^2\psi_{(\xi)}^2\phi_{(\xi)}^2 \xi)(x)$ and $x\mapsto \partial_t(\axi^2\psi_{(\xi)}^2\phi_{(\xi)}^2 \xi)(x)$ are as smooth functions on $\mathbb{T}^3$ bounded, so that they are particularly dominated by some integrable constant function. This allows us to compute
\begin{align*}
\partial_t \mathbb{P}_{\neq 0}\big(\axi^2\psi_{(\xi)}^2\phi_{(\xi)}^2 \xi\big)= \mathbb{P}_{\neq 0}\partial_t\big(\axi^2\psi_{(\xi)}^2\phi_{(\xi)}^2 \xi\big), 
\end{align*}
which implies together with \thref{Lemma 2.1} 
 \begin{align}
 \partial_t \mathbb{P}\mathbb{P}_{\neq 0}\big(\axi^2\psi_{(\xi)}^2\phi_{(\xi)}^2 \xi\big)= \mathbb{P}\mathbb{P}_{\neq 0}\partial_t\big(\axi^2\psi_{(\xi)}^2\phi_{(\xi)}^2 \xi\big)  \label{3.25}
\end{align}
and 
\begin{align*}
 \partial_{x_i} \mathbb{P}\mathbb{P}_{\neq 0}\big(\axi^2\psi_{(\xi)}^2\phi_{(\xi)}^2 \xi\big)= \mathbb{P}\partial_{x_i} \mathbb{P}_{\neq 0} \big(\axi^2\psi_{(\xi)}^2\phi_{(\xi)}^2 \xi\big)=\mathbb{P}\partial_{x_i}  \big(\axi^2\psi_{(\xi)}^2\phi_{(\xi)}^2 \xi\big). 
 \end{align*}
for each $i=1,2,3$.
Taking into account that the Leray projection is bounded on $W_{\neq 0}^{1,p}\left(\mathbb{T}^3\right)$ and the projection onto zero mean functions $\mathbb{P}_{\neq 0}$ is a bounded operator on $W^{1,p}\left(\mathbb{T}^3\right)$, the above expression amounts to
\begin{align*} 
\|\wt\|_{C_{t,x}^1}
&\lesssim \mu^{-1}  \sum_{\xi \in \Lambda} \bigg(\overbrace{\| a_{(\xi)}^2 \psi_{(\xi)}^2\phi_{(\xi)}^2\|_{C_tW^{1,p}}}^{=:\text{I}}+\overbrace{\|\partial_t   ( a_{(\xi)}^2 \psi_{(\xi)}^2 \phi_{(\xi)}^2)\|_{C_tW^{1,p}}}^{=:\text{II}}\bigg) \\
& \hspace{.5cm}+\mu^{-1}  \sum_{\xi \in \Lambda} \underbrace{\sum_{|\beta|=1}\|D^\beta ( a_{(\xi)}^2 \psi_{(\xi)}^2\phi_{(\xi)}^2)\|_{C_tW^{1,p}}}_{=:\text{III}}.
\end{align*}
Employing the formula
\begin{align*}
\Big\|\prod_{n=1}^N f_n\Big\|_{W^{1,p}} \lesssim \sum_{k=1}^N \prod_{\substack{n=1\\ n\neq k}}^N \|f_n\|_{L^\infty} \|f_k\|_{W^{1,p}},  
\end{align*}
which holds for all functions $f_n \in W^{1,p}\left(\mathbb{T}^3\right)\cap L^\infty \left( \mathbb{T}^3\right),\,n=1,\ldots,N,\, N\geq 2$, permits to deduce
\begin{align*}
\text{I}&\lesssim \|\axi\|^2_{C_tL^\infty}\|\psixi\|_{C_tL^\infty}^2\|\phixi\|_{L^\infty} \|\phixi\|_{W^{1,p}}\\
&\hspace{.5cm}+\|\axi\|^2_{C_tL^\infty}\|\psixi\|_{C_tL^\infty}\|\psixi\|_{C_tW^{1,p}}\|\phixi\|^2_{L^\infty}\\
&\hspace{.5cm}+ \|\axi\|_{C_tL^\infty}\|\axi\|_{C_tW^{1,p}}\|\psixi\|_{C_tL^\infty}^2\|\phixi\|^2_{L^\infty}.
\end{align*}
Due to the embedding $W^{1,\infty}\left(\mathbb{T}^3\right)\subseteq W^{1,p}\left(\mathbb{T}^3\right)$ and by invoking \eqref{3.8b}, \eqref{3.13a}, \eqref{3.13b} we obtain
 \begin{align*}
 \text{I} \lesssim \bigg(\frac{M}{4 |\Lambda|}\bigg)^2 \ell^{-23}\delta_{q+1}m^{3/4}\bar{e}r_\perp^{-2}r_\parallel^{-1}\lambda_{q+1}.
 \end{align*}
In the same manner we estimate
\begin{align*}
\text{II}&\lesssim \|\partial_t \axi \axi \psi_{(\xi)}^2\phi^2_{(\xi)}\|_{C_tW^{1,p}}+ \|a_{(\xi)}^2\partial_t\psixi\psixi\phi_{(\xi)}^2\|_{C_tW^{1,p}}\\
&\lesssim \bigg(\frac{M}{4 |\Lambda|}\bigg)^2 \ell^{-30}\delta_{q+1}m^{3/4}\bar{e}r_\perp^{-2}r_\parallel^{-1}\lambda_{q+1} +\bigg(\frac{M}{4 |\Lambda|}\bigg)^2 \ell^{-23}\delta_{q+1}m^{3/4}\bar{e}r_\perp^{-2}r_\parallel^{-1}\lambda_{q+1}^3\\
&\lesssim \bigg(\frac{M}{4 |\Lambda|}\bigg)^2 \ell^{-30}\delta_{q+1}m^{3/4}\bar{e}r_\perp^{-2}r_\parallel^{-1}\lambda_{q+1}^3  
\end{align*}
and
\begin{align*}
\text{III}&\lesssim \sum_{|\beta|\leq 1}\Big\{ \|D^\beta\axi \axi \psi_{(\xi)}^2\phi^2_{(\xi)}\|_{C_tW^{1,p}}+ \|a_{(\xi)}^2D^\beta\psixi\psixi\phi_{(\xi)}^2\|_{C_tW^{1,p}}\Big\}\\
&\hspace{0.5cm}+\sum_{|\beta|\leq 1}\|a_{(\xi)}^2\psi_{(\xi)}^2 D^\beta \phixi \phixi \|_{C_tW^{1,p}}\\
&\lesssim \bigg(\frac{M}{4 |\Lambda|}\bigg)^2 \ell^{-30}\delta_{q+1}m^{3/4}\bar{e} r_\perp^{-2}r_\parallel^{-1} \lambda_{q+1}+ \bigg(\frac{M}{4 |\Lambda|}\bigg)^2 \ell^{-23}\delta_{q+1}m^{3/4}\bar{e}r_\perp^{-1}r_\parallel^{-2}\lambda_{q+1}^2\\
&\hspace{0.5cm}+\bigg(\frac{M}{4 |\Lambda|}\bigg)^2 \ell^{-23}\delta_{q+1}m^{3/4}\bar{e}r_\perp^{-2}r_\parallel^{-1} \lambda_{q+1}^2\\
& \lesssim \bigg(\frac{M}{4 |\Lambda|}\bigg)^2 \ell^{-30}\delta_{q+1}m^{3/4}\bar{e} r_\perp^{-2}r_\parallel^{-1} \lambda_{q+1}^2.
\end{align*}
Therefore we conclude
\begin{align*}
\|\wt\|_{C_{t,x}^1} \lesssim \bigg(\frac{M}{4 |\Lambda|}\bigg)^2 \ell^{-30}\delta_{q+1}m^2\bar{e}r_\perp^{-1}r_\parallel^{-2} \lambda_{q+1}^2.
\end{align*}

\noindent\underline{\eqref{3.17a}:} In view of \eqref{3.14}, \eqref{3.8b} and \eqref{3.13c} we infer again with help of Levi-Civita's-symbol 
\begin{align*}
&\|\wpr+\wc\|_{C_tW^{1,p}}\\
&\hspace{.5cm} \lesssim m^{1/8} \sum_{\xi \in \Lambda} \sum_{\ell, k, n, i, j=1}^3 \|\mathcal{E}_{\ell k n} \mathcal{E}_{i j k} \partial_{x_\ell} \big(\partial_{x_i} \big(\axi \Vxi^j\big)\big)\vec{e}_n\|_{C_tW^{1,p}}\\
& \hspace{0.5cm}\lesssim m^{1/8}\sum_{\xi \in \Lambda} \bigg\{\sum_{|\alpha|=2} \sum_{|\gamma|\leq 1} \|D^\gamma \big(D^\alpha \axi \Vxi \big)\|_{C_tL^p}+ \sum_{\substack{|\alpha|=1\\ |\beta|=1}} \sum_{|\gamma|\leq 1}  \|D^\gamma \big(D^\alpha\axi D^\beta \Vxi \big)\|_{C_tL^p}\bigg\}\\
&\hspace{1cm}+ m^{1/8}\sum_{\xi \in \Lambda}\sum_{|\beta|=2} \sum_{|\gamma|\leq 1} \|D^\gamma \big( \axi D^\beta\Vxi \big)\|_{C_tL^p}\\
&\lesssim m^{1/8} \sum_{\xi \in \Lambda} \bigg\{ \|\axi\|_{C_{t,x}^3} \|\Vxi\|_{C_tL^p}+ \|\axi\|_{C_{t,x}^2} \sum_{|\beta|\leq 1} \|D^\beta\Vxi\|_{C_tL^p}\bigg\} \\
&\hspace{0.5cm} +  m^{1/8} \sum_{\xi \in \Lambda}\bigg\{ \|\axi\|_{C_{t,x}^1} \sum_{|\beta| \leq 2} \|D^\beta\Vxi\|_{C_tL^p} +\|\axi\|_{C_{t,x}^0} \sum_{|\beta|\leq 3} \|D^\beta\Vxi\|_{C_tL^p} \bigg\}\\
&\hspace{.5cm}\lesssim \frac{M}{4|\Lambda|} \delta_{q+1}^{1/2} m^{1/2}\bar{e}^{1/2} r_\perp^{2/p-1}r_\parallel^{1/p-1/2} (\ell^{-29}\lambda_{q+1}^{-2}+\ell^{-22}\lambda_{q+1}^{-1}+ \ell^{-15}+\ell^{-8}\lambda_{q+1})\\
&\hspace{.5cm}\lesssim \frac{M}{4|\Lambda|} \ell^{-29}\delta_{q+1}^{1/2} m \bar{e}^{1/2} r_\perp^{2/p-1}r_\parallel^{1/p-1/2} \lambda_{q+1}.
\end{align*}

\noindent \underline{\eqref{3.17b}:} Bearing in mind that $\mathbb{P}$ and $\mathbb{P}_{\neq 0}$ are both bounded operators on $W_{\neq 0}^{1,p}\left(\mathbb{T}^3\right)$ and $W^{1,p}\left(\mathbb{T}^3\right)$, respectively, we make use of \thref{Lemma 3.4} in order to obtain
\begin{align*}
\|\wt\|_{C_tW^{1,p}} &\lesssim \mu^{-1}   \sum_{\xi \in \Lambda} \bigg\{ \|a_{(\xi)}^2\psi_{(\xi)}^2\phi_{(\xi)}^2\|_{C_tL^p}+ \sum_{|\alpha|=1}\|D^\alpha\axi\axi \psi_{(\xi)}^2\phi_{(\xi)}^2\|_{C_tL^p}\bigg\}\\
&\hspace{0.5cm}+\mu^{-1}   \sum_{\xi \in \Lambda}\bigg\{\sum_{|\alpha|=1}\|a_{(\xi)}^2 D^\alpha \psi^2_{(\xi)}\phi_{(\xi)}^2\|_{C_tL^p}+\sum_{|\alpha|=1}\|a_{(\xi)}^2 \psi_{(\xi)}^2 D^\alpha \phi_{(\xi)}^2  \|_{C_tL^p} \bigg\}\\
&\lesssim \mu^{-1}  \sum_{\xi \in \Lambda} \|\axi\|^2_{C_{t,x}^0} \|\psixi\|_{C_tL^{2p}}^2 \|\phixi\|_{L^{2p}}^2\\
&\hspace{.5cm}+\mu^{-1}  \sum_{\xi \in \Lambda}\|\axi\|_{C_{t,x}^1}\|\axi\|_{C_{t,x}^0} \|\psixi\|_{C_tL^{2p}}^2\|\phixi\|_{L^{2p}}^2\\
&\hspace{.5cm}+\mu^{-1}  \sum_{\xi \in \Lambda} \|\axi\|_{C_{t,x}^0}^2\sum_{|\alpha |\leq 1 } \|D^\alpha \psixi \psixi\|_{C_tL^p}\|\phixi\|_{L^{2p}}^2\\
&\hspace{0.5cm}+\mu^{-1}  \sum_{\xi \in \Lambda}\|\axi\|_{C_{t,x}^0}^2 \|\psixi\|_{C_tL^{2p}}^2 \sum_{|\alpha |\leq 1}\|D^\alpha \phixi \phixi\|_{L^p} .
\end{align*}
Since Cauchy–Schwarz inequality implies
\begin{align*}
\sum_{|\alpha |\leq 1} \|D^\alpha \psixi \psixi\|_{C_tL^p} \lesssim \sum_{|\alpha |\leq 1} \|D^\alpha \psixi\|_{C_tL^{2p}} \|\psixi\|_{C_tL^{2p}},
\end{align*}
and 
\begin{align*}
\sum_{|\alpha |\leq 1} \|D^\alpha \phixi \phixi\|_{L^p} \lesssim \sum_{|\alpha |\leq 1} \|D^\alpha \phixi\|_{L^{2p}} \|\phixi\|_{L^{2p}} 
\end{align*}
respectively, we appeal to \eqref{3.8b}, \eqref{3.13a} and \eqref{3.13b} to conclude
\begin{align*}
\|\wt\|_{C_tW^{1,p}} \lesssim \left(\frac{M}{4|\Lambda|}\right)^2 \ell^{-23}\delta_{q+1} m^2\bar{e}r_\perp^{2/p-1}r_{\parallel}^{1/p-2}. 
\end{align*}

\end{proof}

\pagebreak

\boldmath
\subsection{Verifying the Key Bounds on the Level $q+1$} \label{Section 3.4.2}
\unboldmath
With these preparations in hands, we are now able to justify \eqref{3.3a} and \eqref{3.3b} on the level $q+1$.
\subsubsection{First Key Bound (\ref{3.3a})} \label{Section 3.4.2.1}
By \eqref{3.15d}, \eqref{3.15b}, \eqref{3.15c} and \eqref{3.1b} we find
\begin{align*}
\|w_{q+1}\|_{C_tL^2}\leq  \delta_{q+1}^{1/2} m \bar{e}^{1/2} \left(C\frac{M}{4|\Lambda|}+C\frac{M}{4|\Lambda|}\lambda_{q+1}^{44\alpha-2/7}+C\frac{M}{4|\Lambda|}\frac{M}{4|\Lambda|} \lambda_{q+1}^{33\alpha-1/7}\right)
\end{align*}
for some constant $C>\frac{4|\Lambda|}{3M}$. In order to absorb the subsequent implicit constants we chose
\begin{align}
M_0:=  3C\frac{M}{4|\Lambda|}  \label{3.26}.
\end{align}
Recalling the requirements $\frac{M}{4|\Lambda|}\lambda_{q+1}^{33\alpha-1/7}\leq 1$ and $161\alpha<\frac{1}{7},\, a\geq 3600,\, b\geq 7$, ensuring $\lambda_{q+1}^{44\alpha-2/7}\leq 1$, permits to achieve
\begin{align}
\|w_{q+1}\|_{C_tL^2} \leq  M_0 \delta_{q+1}^{1/2} m\bar{e}^{1/2}\label{3.27}, 
\end{align}
which combined with \eqref{3.6b}, actually confirms 
\begin{align*}
\|v_{q+1}\|_{C_tL^2}\leq \|v_\ell\|_{C_tL^2}+ \|w_{q+1}\|_{C_tL^2} \leq M_0\left(1+\sum_{r=1}^{q+1} \delta_r^{1/2}\right)m \bar{e}^{1/2}.
\end{align*}
\subsubsection{Second Key Bound (\ref{3.3b})} \label{Section 3.4.2.2}
In the same manner as above, \eqref{3.16a}, \eqref{3.16b}, \eqref{3.16c}, \eqref{3.1b}, the fact $161 \alpha<\frac{1}{7}$, \eqref{3.6c} and \eqref{3.3b} furnish the existence of some constants $K, \widetilde{K}>0$, so that 
\begin{align}
\|w_{q+1}\|_{C_{t,x}^1}
&\leq \lambda_{q+1}^5 m \bar{e}^{1/2}\left(K\frac{M}{4|\Lambda|} \lambda_{q+1}^{-12/7}+K\frac{M}{4|\Lambda|}\lambda^{-2}_{q+1}+K\left(\frac{M}{4|\Lambda|}\right)^2  \lambda_{q+1}^{-6/7} \right)
 \label{3.28}
\end{align}
and
\begin{align}
\|v_\ell\|_{C_{t,x}^1}\leq \widetilde{K} \|v_q\|_{C_{t,x}^1} \leq \widetilde{K} \lambda_q^5 \lambda_{q+1}^{-5}\lambda_{q+1}^5 m \bar{e}^{1/2}, \label{3.29}
\end{align}
respectively. As pointed out in Section \ref{Section 3.1}, we increase $a$ and $b$ in a fashion that \linebreak $K\frac{M}{4|\Lambda|} \lambda_{q+1}^{-12/7}+K\frac{M}{4|\Lambda|}\lambda^{-2}_{q+1}+K\left(\frac{M}{4|\Lambda|}\right)^2  \lambda_{q+1}^{-6/7} \leq \frac{1}{2}$ as well as $\widetilde{K}\lambda_q^5\lambda_{q+1}^{-5}\leq \frac{1}{2}$. That means \eqref{3.3b} stays true on the level $q+1$.

\subsection{Control of the Energy} \label{Section 3.4.3}
It remains to affirm \eqref{3.5} at level $q+1$, which is equivalent in showing 
\begin{align*}
|\Theta^{-2}(t) e(t)(1- \delta_{q+2})-\|v_{q+1}(t)\|_{L^2}^2|\leq \frac{1}{4}\delta_{q+2}\Theta^{-2}(t)e(t),
\end{align*}
or expressed in terms of the function $\eta_q$
\begin{align*}
|3(2\pi)^3\eta_q(t)+\|v_q(t)\|_{L^2}^2-\|v_{q+1}(t)\|_{L^2}^2|\leq  \frac{1}{4}\delta_{q+2}\Theta^{-2}(t)e(t)
\end{align*}
whenever $t \in (-\infty, \tau]$.\\
As we will see it is not enough to require only boundedness of the energy $e$ here, rather it is necessary to ask for a uniform bound of its derivative $e^\prime$. From a physical point of view it means that the change of kinetic energy and therefore the acceleration of a fluid can not become arbitrary large. For example, if we consider a river that flows uphill, the gravity  will influence its flow rate, so that the gradient can only attend limited values.\\
Anyway, lets come back to the mathematical computations: 
\begin{align}\label{3.30}
\hspace{-0.3cm}|3&(2\pi)^3\eta_q(t)+\|v_q(t)\|_{L^2}^2-\|v_{q+1}(t)\|_{L^2}^2| \notag \\
&\leq \Big\|3(2\pi)^3\eta_q+\|v_q\|^2_{L^2}-\||v_\ell+\wpr+\wc+\wt|^2\|_{L^1}\Big\|_{C_t} \notag \\
\begin{split}
&\leq \overbrace{\Big\|3(2\pi)^3\eta_q- \|\wpr\|_{L^2}^2\Big\|_{C_t}}^{\text{=:I}} +\overbrace{\Big\| \|v_q\|^2_{L^2}-\|v_\ell\|^2_{L^2}\Big\|_{C_t}}^{\text{=:II}}+\overbrace{2\|v_\ell \cdot \wpr\|_{C_tL^1}}^{\text{=:III}}  \\
&\hspace{0.5cm}+\underbrace{2\|v_\ell \cdot(\wc+\wt)\|_{C_tL^1}+2\|\wpr \cdot(\wc+\wt)\|_{C_tL^1}}_{\text{=:IV}}+\underbrace{\|\wc+\wt\|^2_{C_tL^2}}_{\text{=:V}}\\
\end{split}
\end{align}
and we proceed with a bound for I. \\
\begin{tikzpicture}[baseline=(char.base)]
\node(char)[draw,fill=white,
  shape=rounded rectangle,
  drop shadow={opacity=.5,shadow xshift=0pt},
  minimum width=.8cm]
  {\Large I};
\end{tikzpicture} 
For this purpose we first assert
\begin{adjustwidth}{20pt}{20pt}
\begin{flalign} 
\text{\underline{1.Claim:}}\ \  & \Theta_\ell \left(\wpr \otimes \wpr\right)=\sum_{\xi \in \Lambda} \axi^2 \mathbb{P}_{\neq 0}\big(\Wxi \otimes \Wxi \big)+\rho \Id-\mathring{R}_\ell.& \label{3.31}
\end{flalign}
\begin{proof}
Keeping in mind that the mutually disjoint supports of $\left( \Wxi\right)_{\xi \in \Lambda}$ causes \linebreak $\Wxi \otimes W_{(\xi^\prime)} \equiv 0$ for $\xi \neq \xi^\prime$, we invoke \thref{Lemma 3.4} to deduce
\begin{align*}
&\wpr \otimes  \wpr\\
&\hspace{.5cm}= \Theta_\ell^{-1}\sum_{\xi \in \Lambda} \axi^2 \big(\Wxi \otimes \Wxi\big)\\
&\hspace{.5cm}= \Theta_\ell^{-1}\sum_{\xi \in \Lambda} \axi^2 \mathbb{P}_{\neq 0}\big(\Wxi \otimes \Wxi\big)+\Theta_\ell^{-1}\sum_{\xi \in \Lambda} \axi^2 \big(\xi \otimes \xi\big) (2\pi)^{-6}\|\psi_{(\xi)}^2\|_{L^1}\|\phi_{(\xi)}^2\|_{L^1}
\end{align*}
and appealing to the geometric \thref{Lemma 3.2} as well as to the normalizations \eqref{3.9}, we find
\begin{align*}
\wpr \otimes \wpr=\Theta_\ell^{-1} \sum_{\xi \in \Lambda} \axi^2 \mathbb{P}_{\neq 0}\big(\Wxi \otimes \Wxi \big)+\Theta_\ell^{-1}\left(\rho \Id-\mathring{R}_\ell\right). 
\end{align*}
\end{proof}
\end{adjustwidth}
to deduce  
\begin{align*}
|\wpr|^2-3\eta_q&= \tr\Big(\wpr \otimes \wpr\Big)-3\eta_q\\
&= \Theta_\ell^{-1}\bigg[\sum_{\xi \in \Lambda} \axi^2 \mathbb{P}_{\neq 0} \tr\Big(\Wxi \otimes \Wxi\Big) +\tr(\Id)\rho-\tr(\mathring{R}_\ell)\bigg]-3 \eta_q\\
&= 6 \Theta_\ell^{-1}\sqrt{\ell^2+\|\mathring{R}_\ell\|_F^2}+ 3(\eta_\ell-\eta_q)+ \Theta_\ell^{-1}\sum_{\xi \in \Lambda} \axi^2 \mathbb{P}_{\neq 0}|\Wxi|^2,
\end{align*}
yielding
\begin{align*}
&\Big\| \|\wpr\|_{L^2}^2-3(2\pi)^3 \eta_q \Big\|_{C_t}\\ &\hspace{.5cm}\leq \int_{\mathbb{T}^3} \Big\| 6\Theta_\ell^{-1}\sqrt{\ell^2+\|\mathring{R}_\ell\|_F^2} \Big\|_{C_t} \,dx +3 \int_{\mathbb{T}^3} \|\eta_\ell-\eta_q\|_{C_t}\, dx+\bigg\|\Theta_\ell^{-1}\sum_{\xi \in \Lambda}\int_{\mathbb{T}^3}\axi^2 \mathbb{P}_{\neq 0} |\Wxi|^2\,dx\bigg\|_{C_t}\\
&\hspace{.5cm}\leq 6m^{1/4}(2\pi)^3 \ell+6m^{1/4}\|\mathring{R}_\ell\|_{C_tL^1}+3(2\pi)^3 \|\eta_\ell-\eta_q\|_{C_t}+m^{1/4} \sum_{\xi \in \Lambda}\bigg\| \int_{\mathbb{T}^3} \axi^2 \mathbb{P}_{\neq 0} |\Wxi|^2\,dx\bigg\|_{C_t}.
\end{align*}
We continue by proving
\begin{adjustwidth}{20pt}{20pt}
\begin{flalign} 
\text{\underline{2.Claim:}}\ \  & 6 m^{1/4}(2\pi)^3 \ell \leq \frac{1}{80} \delta_{q+2} \Theta^{-2}(t) e(t). & \label{3.32}
\end{flalign}
\begin{proof}
Using $2\beta b< \frac{3 \alpha}{2}$, implied by the assumptions $4 \beta b^2< \alpha$ and $b \geq 7$, together with $e(t)\geq \underline{e}> 4$ and $\Theta^{-2}(t)\geq m^{-1/2}$ and also $\lambda_q>3600^7>\sqrt{120 }m^{3/4}\, (2\pi)^{3/2}$, we compute 
\begin{align*}
6m^{1/4}(2\pi)^3 \ell \leq \frac{3}{2} e(t)m^{1/4}(2\pi)^3 \lambda_{q+1}^{-2\beta b} \lambda_{q}^{-2}\leq \frac{1}{80m^{1/2}} e(t)\lambda_{q+2}^{-2\beta}  \leq \frac{1}{80} \delta_{q+2} \Theta^{-2}(t) e(t).
\end{align*}
\end{proof}
\end{adjustwidth}
We further need
\begin{adjustwidth}{20pt}{20pt}
\begin{flalign} 
\text{\underline{3.Claim:}}\ \  &\|\mathring{R}_\ell\|_{C_tL^1} \leq \frac{1}{480 m^{1/4}} \delta_{q+2}\Theta^{-2}(t) e(t),& \label{3.33}
\end{flalign}
\begin{proof}
It follows immediately from  Fubini's theorem, the normalization of the mollifiers and \eqref{3.3c}.
\end{proof}
\end{adjustwidth}

\begin{adjustwidth}{20pt}{20pt}
\begin{flalign} 
\text{\underline{4.Claim:}}\ \  &3(2\pi)^3 \|\eta_\ell-\eta_q\|_{C_t} \leq \frac{1}{80} \delta_{q+2}\Theta^{-2}(t) e(t)& \label{3.34}
\end{flalign}
\begin{proof}
Both terms of
\begin{align} \label{3.35}
3(2&\pi)^3 \|\eta_\ell-\eta_q\|_{C_t} \notag\\ 
&\leq (1-\delta_{q+2}) \|\Theta^{-2} e-\left( \Theta^{-2}e\right)\ast \varphi_\ell\|_{C_t}+ \Big\| \|v_q\|^2_{L^2}-\|v_q\|_{L^2}^2\ast_t \varphi_\ell \Big\|_{C_t}
\end{align}
will be estimated separately.\\ 
First, for any $s\in(-\infty,t]$ and $u \in [0,\ell]$ it holds by virtue of the product rule and It\^{o}'s formula  
\begin{align*}
&|\Theta^{-2}(s)e(s)-\Theta^{-2}(s-u)e(s-u)|\\
&\hspace{.5cm}\leq 2 \int_{(s\vee u)-u}^s|\Theta^{-2}(r)e(r)|\,dr+ 2\int_{(s\vee u)-u}^s|\Theta^{-2}(r)e(r)|\,dB_r+\int_{(s\vee u)-u}^s|\Theta^{-2}(r)e^\prime(r)|\,dr\\
&\hspace{.5cm}\leq 2m^{1/2}\bar{e}(s\wedge u)+2 m^{1/2}\bar{e}(B_s-B_{(s\vee u)-u})+ m^{1/2}\widetilde{e}(s\wedge u)\\
&\hspace{.5cm}\leq 2 m^{1/2} \bar{e}\ell+2 m^{1/2}\bar{e} \|B\|_{C_t^{0,\iota}}\ell^\iota+  m^{1/2}\widetilde{e}\ell\\
&\hspace{.5cm}\leq 5 m^{1/2}(\bar{e}+ \widetilde{e})\ell^\iota.
\end{align*}
Hence,
\begin{align} \label{3.36}
\|\Theta^{-2}e-\left( \Theta^{-2} e\right) \ast_t \varphi_\ell\|_{C_t} &\leq \sup_{s\in(-\infty,t]} \int_0^\ell |\Theta^{-2}(s)e(s)-\Theta^{-2}(s-u)e(s-u)| \varphi_\ell(u) \,du \notag\\
&\leq 5 m^{1/2}(\bar{e}+\widetilde{e})\ell^\iota,
\end{align}
where we exploited the normalization of $\varphi$ in both steps again.\\
Estimating the second term of \eqref{3.35} follows by standard mollification estimates
\begin{align*}
&\Big\|\|v_q\|_{L^2}^2-\|v_q\|_{L^2}^2\ast_t \varphi_\ell \Big\|_{C_t}\\
&\hspace{.5cm}\leq \sup_{s \in (-\infty,t]} \int_0^\ell \Big|\|v_q(s)\|^2_{L^2}-\|v_q(s-u)\|^2_{L^2} \Big| \varphi_\ell(u) \,du \\
&\hspace{.5cm}\leq \sup_{s \in (-\infty,t]} \int_0^\ell \bigg|\int_0^1 \partial_\theta \|v_q(s-\theta u)\|^2_{L^2} \,d\theta \bigg| \varphi_\ell (u) \,du \\
&\hspace{.5cm}\leq 2\sup_{s \in (-\infty,t]}\int_0^\ell \bigg| \int_0^1 \| v_q(s-\theta u)\cdot \partial_\theta v_q(s-\theta u)\|_{L^1} \,d\theta \bigg| \varphi_\ell (u) \,du \\
&\hspace{.5cm}\leq 2(2\pi)^{3/2}\sup_{s \in (-\infty,t]}\int_0^\ell \bigg|\int_0^1  \|\partial_\theta v_q(s-\theta u)\|_{L^\infty} \,d\theta \bigg| \varphi_\ell (u) \,du \|v_q\|_{C_tL^2} \\
&\hspace{.5cm} =2(2\pi)^{3/2}\sup_{s \in (-\infty,t]} \int_0^\ell \bigg| \int_0^1 u\,  \|\partial_s  v_q(s-\theta u)\|_{L^\infty} \,d\theta \bigg| \varphi_\ell (u)\, du  \, \|v_q\|_{C_tL^2}\\
&\hspace{.5cm}\leq 2(2\pi)^{3/2} \ell \|\partial_s v_q\|_{C_tL^\infty} \|v_q\|_{C_tL^2}.
\end{align*}
Therefore, it holds owing to \eqref{3.3a}, \eqref{3.3b}, \eqref{3.4} and \eqref{3.1a}
\begin{align*}
3(2\pi)^3 \|\eta_\ell-\eta_q\|_{C_t} &\leq (1-\delta_{q+2})5m^{1/2}(\bar{e}+\widehat{e})\ell^\iota  +6(2\pi)^{3/2}M_0  \ell \lambda_{q}^5m^2\bar{e}\\
&\leq 5\lambda_{q+1}^{-\frac{3\alpha}{2}\iota}m^{1/2}(\bar{e}+ \widetilde{e})+6(2\pi)^{3/2}M_0\lambda_{q+1}^{-\alpha}m^2\bar{e}\\
&\leq 12(2\pi)^{3/2}M_0 \lambda_{q+1}^{-\alpha\big(\frac{3}{2}\iota-\frac{1}{2}\big)}m^2(\bar{e}+\widetilde{e})\lambda_{q+1}^{-\frac{\alpha}{2}}.
\end{align*}
In view of $12(2\pi)^{3/2}M_0 \lambda_{q+1}^{-\alpha\big(\frac{3}{2}\iota-\frac{1}{2}\big)}m^2(\bar{e}+\widetilde{e})\leq \frac{1}{20m^{1/2}} \leq \frac{1}{80}\Theta^{-2}(t) e(t)$ and $\alpha>4\beta b$, yielding $\lambda_{q+1}^{-\alpha/2}<\lambda_{q+1}^{-2\beta b}\leq \delta_{q+2}$, we get the desired bound.

\end{proof}
\end{adjustwidth}
\vspace{0.5cm}
and lastly

\begin{adjustwidth}{20pt}{20pt}
\begin{flalign} 
\text{\underline{5.Claim:}}\ \  &\sum_{\xi \in \Lambda} \Big\|\int_{\mathbb{T}^3}\axi^2 \mathbb{P}_{\neq 0} |\Wxi|^2 \, dx\Big\|_{C_t} \leq \frac{1}{80m^{1/4}} \delta_{q+2} \Theta^{-2}(t) e(t). \label{3.37}&
\end{flalign}
\begin{proof}
Applying Green's identity $L>0$ times to the functions $\axi^2$ and $\Delta^{-L}\mathbb{P}_{\neq 0}|\Wxi|^2$ furnishes
\begin{align*}
\sum_{\xi \in \Lambda} \Big\|\int_{\mathbb{T}^3}\axi^2 \mathbb{P}_{\neq 0} |\Wxi|^2 \, dx\Big\|_{C_t} \lesssim \sum_{\xi \in \Lambda} \|\Delta^L \,\axi^2\|_{C_tL^\infty}\Big\|\int_{\mathbb{T}^3} \Delta^{-L}\, \mathbb{P}_{\neq 0} |\Wxi|^2 \, dx\Big\|_{C_t}.
\end{align*}
Thanks to Leibniz rule, we get
\begin{align*}
\|\Delta^L \, \axi^2\|_{C_tL^\infty} &\lesssim \sum_{|\alpha|=2L} \sum_{\beta \leq \alpha} \|D^\beta \axi\|_{C_tL^\infty} \|D^{\alpha-\beta}\axi\|_{C_tL^\infty} \\
&\lesssim \sum_{n=0}^{2L} \|\axi\|_{C^n_{t,x}} \|\axi\|_{C^{2L-n}_{t,x}}.
\end{align*}
Moreover, \thref{Lemma 2.5}, Cauchy-Schwarz's inequality and \thref{Lemma 2.4} entail
\begin{align*}
\Big\|\int_{\mathbb{T}^3} \Delta^{-L} \, \mathbb{P}_{\neq 0} |\Wxi|^2 \,dx \Big\|_{C_t} &\leq \int_{\mathbb{T}^3} \Big\|\Big((-\Delta)^{-1/2} \mathbb{P}_{\geq \frac{r_\perp \lambda_{q+1}}{2}}\Big)^{2L}|\Wxi|^2\Big\|_{C_t}\,dx\\ 
&\lesssim\Big\|\Big((-\Delta)^{-1/2} \mathbb{P}_{\geq \frac{r_\perp \lambda_{q+1}}{2}}\Big)^{2L}|\Wxi|^2\Big\|_{C_tL^2}\\
& \lesssim (r_\perp \lambda_{q+1})^{-2L} \|\Wxi\|^2_{C_tL^4}.
\end{align*}
Therefore, by stipulating $L=5$, we conclude with help of \eqref{3.8b}, \eqref{3.13c} and \eqref{3.1b} 
\begin{align*}
\sum_{\xi \in \Lambda} \Big\|\int_{\mathbb{T}^3}\axi^2 \mathbb{P}_{\neq 0} |\Wxi|^2 \, dx\Big\|_{C_t} &
\lesssim \left(\frac{M}{4|\Lambda|}\right)^2  \ell^{-86} \delta_{q+1} m^{3/4}\bar{e} (r_\perp \lambda_{q+1})^{-10}r_\perp^{-1} r_\parallel^{-1/2}\\
& \lesssim  \lambda_{q+1}^{174\alpha-2/7} \lambda_1^{2 \beta}\\
&\lesssim \lambda_{q+1}^{-147\alpha- 2\beta b} \lambda_1^{2\beta} \\
&\lesssim \lambda_{q+1}^{-147\alpha} \delta_{q+2}e(t).
\end{align*}
Here we used the fact $174\alpha-\frac{2}{7}<-148\alpha<-147 \alpha-2\beta b$, which follows from the constraint $161\alpha<\frac{1}{7}$ and $2 \beta b<4\beta b^2<\alpha$, in the penultimate and we employed $4 \leq\underline{e}\leq e(t)$ in the last step.\\
In order to absorb the implicit constant, we choose $a,b$ sufficiently large and $\alpha,\beta$ sufficiently small enough. In other words, we have ascertained the existence of some $\hat{K}>0$, so that 
\begin{align}
\sum_{\xi \in \Lambda} \Big\|\int_{\mathbb{T}^3}\axi^2 \mathbb{P}_{\neq 0} |\Wxi|^2 \, dx\Big\|_{C_t}  \leq \hat{K} \lambda_{q+1}^{-147\alpha}\delta_{q+2}e(t) \label{3.38}
\end{align}
holds. Imposing $\hat{K}\lambda_{q+1}^{-147\alpha}\leq \frac{1}{80m^{3/4}}$ (cf. Section \ref{Section 3.1}) yield the desired bound.
\end{proof}
\end{adjustwidth}
\vspace{0.5cm}

\noindent Armed with these statements, we are now able to bound I, i.e. taking \eqref{3.32}, \eqref{3.33}, \eqref{3.34} and \eqref{3.37} into account, we find

\begin{align*}
&\Big\| \|\wpr\|_{L^2}^2-3(2\pi)^3 \eta_q \Big\|_{C_t}\leq 4\cdot\frac{1}{80} \delta_{q+2}\Theta^{-2}(t) e(t).
\end{align*}

\noindent \begin{tikzpicture}[baseline=(char.base)]
\node(char)[draw,fill=white,
  shape=rounded rectangle,
  drop shadow={opacity=.5,shadow xshift=0pt},
  minimum width=.8cm]
  {\Large II};
\end{tikzpicture} 
The second term can according to Cauchy-Schwarz's inequality, \eqref{3.6a}, \eqref{3.3b}, \eqref{3.3a}, \eqref{3.6b}, \eqref{3.4} and \eqref{3.1a} be bounded as
\begin{align*}
\Big\|\|v_q\|_{L^2}^2-\|v_\ell\|^2_{L^2} \Big\|_{C_t}&            \leq \|v_q-v_\ell\|_{C_tL^2} \big(\|v_q\|_{C_tL^2}+\|v_\ell\|_{C_tL^2}\big) \\
& \leq (2\pi)^{3/2} \ell  \lambda_q^5 m\bar{e}^{1/2} 6M_0m\bar{e}^{1/2}\\
&\leq 6(2\pi)^{3/2} M_0\lambda_{q+1}^{-\alpha} m^2\bar{e}.
\end{align*}
Invoking the assumption $6(2\pi)^{3/2}M_0\lambda_{q+1}^{-\alpha/2}m^2\bar{e}\leq  \frac{1}{5m^{1/2}}$ together with $e(t)\geq\underline{e}\geq 4$, $ \Theta^{-2}(t)\geq m^{-1/2}$ and $\alpha>4\beta b^2$ permits to conclude
\begin{align*}
\Big\|\|v_q\|_{L^2}^2-\|v_\ell\|^2_{L^2} \Big\|_{C_t}&\leq \frac{1}{5m^{1/2}}\lambda_{q+1}^{-\alpha/2}\leq \frac{1}{20m^{1/2}} \lambda_{q+1}^{-2\beta b}e(t)\leq \frac{1}{20} \delta_{q+2}\Theta^{-2}(t) e(t).
\end{align*}

\noindent \begin{tikzpicture}[baseline=(char.base)]
\node(char)[draw,fill=white,
  shape=rounded rectangle,
  drop shadow={opacity=.5,shadow xshift=0pt},
  minimum width=.8cm]
  {\Large III};
\end{tikzpicture}
For the third term we employ Hölder's inequality with exponents $r:=3$ and $p:=\frac{3}{2}$ and appeal to \eqref{3.3b}, \eqref{3.15a}, \eqref{3.1a} and \eqref{3.1b} to deduce
\begin{align*}
\|v_\ell\cdot\wpr\|_{C_tL^1}&\leq \|v_q\|_{C_tL^\infty}(2 \pi)^{3/r}\|\wpr\|_{C_tL^p}\lesssim 2\pi \frac{M}{4|\Lambda|} \lambda_{q+1}^{19\alpha-8/21}.
\end{align*}
Furthermore, it follows from the assumption $19\alpha-1/7<160\alpha-1/7<-\alpha$ and $\alpha>4\beta b^2>2\beta b$ that
\begin{align}\label{3.39}
2 \|v_\ell \cdot \wpr\|_{C_tL^1} \leq K^\ast \,2\pi \frac{M}{4|\Lambda|}\lambda_{q+1}^{-5/21} \lambda_{q+1}^{-2\beta b}
\end{align}
holds for some constant $K^\ast>0$. Imposing $K^\ast \, 2\pi \, \frac{M}{4|\Lambda|} \lambda_{q+1}^{-5/21} \leq \frac{1}{5m^{1/2}}$ and remembering that $e(t)\geq \underline{e}\geq 4$ and $ \Theta^{-2}(t)\geq m^{-1/2}$ we find
\begin{align*}
2\|v_\ell \cdot \wpr\|_{C_tL^1} \leq \frac{1}{20} \delta_{q+2} \Theta^{-2}(t) e(t).
\end{align*}

\noindent \begin{tikzpicture}[baseline=(char.base)]
\node(char)[draw,fill=white,
  shape=rounded rectangle,
  drop shadow={opacity=.5,shadow xshift=0pt},
  minimum width=.8cm]
  {\Large IV};
\end{tikzpicture} 
Next, we aim at estimating IV also as $\text{IV}\leq \frac{1}{20} \delta_{q+2}\Theta^{-2}(t)e(t)$. Namely, after applying Cauchy–Schwarz's inequality, we use \eqref{3.6b}, \eqref{3.15d}, \eqref{3.15b}, \eqref{3.15c} and \eqref{3.1b} to obtain
\begin{align*}
&2\|v_\ell \cdot(\wc+\wt)\|_{C_tL^1}+2\|\wpr\cdot(\wc+\wt)\|_{C_tL^1}\\
& \hspace{.5cm}\leq 2\Big(\|v_\ell\|_{C_tL^2}+\|\wpr\|_{C_tL^2}\Big)\Big(\|\wc\|_{C_tL^2}+\|\wt\|_{C_tL^2}\Big)\\
& \hspace{.5cm}\lesssim \left(M_0\lambda_{q+1}^\alpha+\frac{M}{4|\Lambda|}\lambda_{q+1}^\alpha \right)  \left(\frac{M}{4|\Lambda|}\lambda_{q+1}^{45\alpha-2/7}+\left(\frac{M}{4|\Lambda|}\right)^2\lambda_{q+1}^{34\alpha-1/7}\right)\\
&\hspace{.5cm} \lesssim \left(M_0+\frac{M}{4|\Lambda|} \right)  \left(\frac{M}{4|\Lambda|}+\left(\frac{M}{4|\Lambda|}\right)^2\right) \lambda_{q+1}^{46\alpha-1/7}.
\end{align*}
Due to $46\alpha-\frac{1}{14}<-2\beta b$, entailing from $160 \alpha- 1/7<-\alpha$ and $\alpha>4\beta b^2$, it holds
\begin{align} \label{3.40}
\begin{split}
&2\|v_\ell \cdot(\wc+\wt)\|_{C_tL^1}+2\|\wpr \cdot(\wc+\wt)\|_{C_tL^1}\\
&\hspace{.5cm} \leq K^\prime\left(M_0+\frac{M}{4|\Lambda|} \right)  \left(\frac{M}{4|\Lambda|}+\left(\frac{M}{4|\Lambda|}\right)^2\right)\lambda_{q+1}^{-\frac{1}{14}} \lambda_{q+1}^{-2\beta b}
\end{split}
\end{align}
for some constant $K^\prime>0$. Similarly as above we achieve
\begin{align*}
2\|v_\ell \cdot(\wc+\wt)\|_{C_tL^1}+2\|\wpr \cdot(\wc+&\wt)\|_{C_tL^1}\leq \frac{1}{20} \delta_{q+2}\Theta^{-2}(t) e(t)
\end{align*}
by assuming $ K^\prime\left(M_0+\frac{M}{4|\Lambda|} \right)  \left(\frac{M}{4|\Lambda|}+\left(\frac{M}{4|\Lambda|}\right)^2\right)\lambda_{q+1}^{-\frac{1}{14}}\leq \frac{1}{5m^{1/2}}$.

\noindent \begin{tikzpicture}[baseline=(char.base)]
\node(char)[draw,fill=white,
  shape=rounded rectangle,
  drop shadow={opacity=.5,shadow xshift=0pt},
  minimum width=.8cm]
  {\Large V};
\end{tikzpicture}
To bound the last term, we proceed similarly as in IV. Namely, thanks to \eqref{3.15b}, \eqref{3.15c}, \eqref{3.1b} and to the required $160\alpha-1/7<-\alpha$ and $\alpha>4\beta b^2$ it holds
\pagebreak
\begin{align*}
\|\wc+\wt\|^2_{C_tL^2} &\lesssim \left(\frac{M}{4|\Lambda|} \right)^2  \lambda_{q+1}^{90\alpha-4/7}+\left(\frac{M}{4|\Lambda|} \right)^4 \lambda_{q+1}^{68\alpha-2/7}\\
& \lesssim \left(\left(\frac{M}{4|\Lambda|} \right)^2+\left(\frac{M}{4|\Lambda|} \right)^4\right)  \lambda_{q+1}^{-2\beta b -1/7}.
\end{align*}
Moreover, we already know about the existence of some $K^{\prime \prime}>0$, such that 
\begin{align} \label{3.41}
\|\wc+\wt\|^2_{C_tL^2} \leq K^{\prime \prime} \left(\left(\frac{M}{4|\Lambda|} \right)^2+\left(\frac{M}{4|\Lambda|} \right)^4\right)  \lambda_{q+1}^{-1/7} \delta_{q+2}.
\end{align}
Choosing the parameters in such a way that $K^{\prime \prime} \left(\left(\frac{M}{4|\Lambda|} \right)^2+\left(\frac{M}{4|\Lambda|} \right)^4\right)  \lambda_{q+1}^{-1/7}\leq \frac{1}{5m^{1/2}}$, leads together with $e(t)\geq \underline{e}\geq 4$ and $\Theta^{-2}(t)\geq m^{-1/2}$ to
\begin{align*}
\|\wc+\wt\|^2_{C_tL^2}\leq \frac{1}{20}\delta_{q+2}\Theta^{-2}(t) e(t).
\end{align*}
Finally, plugging all the above estimates from I through to V into \eqref{3.30} proves \eqref{3.5} on the level $q+1$.

\subsection{Convergence of the Sequence}\label{Section 3.4.4}
Moreover, we claim $\left(v_q\right)_{q \in \mathbb{N}_0}$ to be a Cauchy sequences in $C\left((-\infty,\tau]; L^2\left( \mathbb{T}^3 \right)\right)$.\\
Thanks to \eqref{3.27} and \eqref{3.6a} we obtain
\begin{align}\label{3.42}
\|v_{q+1}-v_q\|_{C_tL^2}\leq \|\omega_{q+1}\|_{C_tL^2}+\|v_\ell-v_q\|_{C_tL^2} \leq \left(M_0+(2\pi)^{3/2}\right) \delta_{q+1}^{1/2}m\bar{e}^{1/2}
\end{align}
for any $t \in (-\infty,\tau]$ and $q \in \mathbb{N}_0$, which results together with $a^{\beta b}\geq 2$ in
\begin{align*}
\sum_{q\geq 0} \|v_{q+1}-v_q\|_{C_tL^2}& \leq  \left(M_0+(2\pi)^{3/2}\right) m \bar{e}^{1/2}\sum_{q\geq 0} a^{\beta b-(q+1)\beta b}\\
&\hspace{-0.3cm}\overset{\text{geom.}}{\underset{\text{series}}{=}} \left(M_0+(2\pi)^{3/2}\right) m \bar{e}^{1/2}\frac{1}{1-a^{-\beta b}}\\
& \leq 2\left(M_0+(2\pi)^{3/2}\right)m\bar{e}^{1/2}.
\end{align*}
As a consequence the sequence $\left(\sum\limits_{q=0}^n \|v_{q+1}-v_q\|_{C_tL^2}\right)_{n\in \mathbb{N}_0}$ converges and is therefore in particular Cauchy in $\mathbb{R}$, i.e. for each $\varepsilon>0$ there exists some $N\in \mathbb{N}_0$, such that 
\begin{align*}
 \|v_n-v_k\|_{C_tL^2}\leq \sum_{q=k}^{n-1} \|v_{q+1}-v_q\|_{C_tL^2} \leq \varepsilon
\end{align*}
holds for every $n\geq k\geq N$, verifying the claim.

\boldmath
\section{Decomposition of the Reynolds Stress $\mathring{R}_{q+1}$} \label{Section 3.5}
\unboldmath
In order to find an expression for the Reynolds error at level $q+1$ we plug $v_{q+1}$ into \eqref{3.2} and
exploit the formula $\divs \big(A \big)= \divs\big(\mathring{A} \big)+\frac{1}{3} \nabla \tr\big(A\big)$ together with $\tr(a \otimes b)=a \cdot b $, which holds for arbitrary quadratic matrices $A$ and vectors $a,\, b$, respectively, to derive

\begin{align*}
\divs(&\mathring{R}_{q+1})-\nabla p_{q+1}\\[15pt]
\begin{split}&:= \underbrace{\big[\partial_tv_q+\frac{1}{2}v_q-\Delta v_q+\Theta \divs(v_q \otimes v_q) \big]}_{= \divs(\mathring{R}_q)-\nabla p_q}\ast_t \varphi_\ell \ast_x \phi_\ell-[\Theta \divs(v_q \otimes v_q)]\ast_t \varphi_\ell \ast_x \phi_\ell\\
&\hspace{0.5cm}+\partial_t w_{q+1}+\frac{1}{2} w_{q+1} -\Delta w_{q+1}+ \Theta_\ell \divs(v_{q+1}\otimes v_{q+1})+(\Theta-\Theta_\ell) \divs(v_{q+1}\otimes v_{q+1})\end{split}\\[15pt]
&=\underbrace{\partial_t (\wpr+\wc)+\frac{1}{2} w_{q+1}-\Delta w_{q+1} +\Theta_\ell \divs(v_\ell \mathring{\otimes}w_{q+1}+w_{q+1}\mathring{\otimes}v_\ell)}_{=:\divs(\mathring{R}_\text{lin})}+ \nabla \underbrace{\frac{2}{3}\Theta_\ell(v_\ell \cdot w_{q+1})}_{=:p_\text{lin}}\\
&\hspace{0.5cm}+\underbrace{\partial_t \wt+ \divs(\Theta_\ell \wpr \otimes \wpr+\mathring{R}_\ell)-\nabla p_\ell}_{=:\divs(\mathring{R}_\text{osc})+\nabla p_\text{osc}}\\
&\hspace{0.5cm}+\divs\bigg(\underbrace{\Theta_\ell\Big[(\wc+\wt) \mathring{\otimes} w_{q+1}+ \wpr \mathring{\otimes}(\wc+\wt) \Big]}_{:=\mathring{R}_\text{cor}}\bigg)\\
&\hspace{.5cm}+\nabla \underbrace{ \frac{1}{3} \Theta_\ell \Big[(\wc+\wt)\cdot w_{q+1}+\wpr\cdot (\wc+\wt)\Big]}_{=:p_\text{cor}}\\
&\hspace{0.5cm}+ \divs\bigg(\underbrace{(\Theta-\Theta_\ell)v_{q+1} \mathring{\otimes} v_{q+1}+\Theta_\ell v_\ell \mathring{\otimes}v_\ell -(\Theta v_q\mathring{\otimes} v_q)\ast_t \varphi_\ell \ast_x \phi_\ell}_{:= \mathring{R}_\text{com}} \bigg)\\
&\hspace{.5cm}+\nabla \underbrace{\frac{1}{3}\bigg[(\Theta-\Theta_\ell)|v_{q+1}|^2+\Theta_\ell |v_\ell|^2-(\Theta |v_q|^2)\ast_t\varphi_\ell \ast_x \phi_\ell  \bigg]}_{=:p_\text{com}},
\end{align*}
where we have set $p_\ell:= p_q\ast_t \varphi_\ell \ast_x \phi_\ell$.

\noindent Keeping in mind that $\divs\big(\axi^2 \mathbb{P}_{\neq 0} (\Wxi \otimes \Wxi)\big)$ has zero mean, since $\axi^2 \mathbb{P}_{\neq 0} (\Wxi \otimes \Wxi)$ is a smooth function with periodic boundary conditions, we invoke \eqref{3.31} to rewrite the oscillation error as 
\begin{align} \label{3.43}
\begin{split}
\divs(\mathring{R}_{\text{osc}})+\nabla p_{\text{osc}} &=\partial_t \wt+\sum_{\xi \in \Lambda} \mathbb{P}_{\neq 0}\Big[  \mathbb{P}_{\neq 0}(\Wxi \otimes \Wxi)\nabla\axi^2\Big]\\&
\hspace{0.5cm} +\sum_{\xi \in \Lambda} \mathbb{P}_{\neq 0}\bigg[ \axi^2 \divs\Big(\mathbb{P}_{\neq 0}(\Wxi \otimes \Wxi)\Big) \bigg] +\nabla (\rho-p_\ell) , 
\end{split}
\end{align}
where 
\begin{align*}
\big(\divs( \Wxi \otimes \Wxi)\big)_i
&=\sum_{j=1}^3  2n_\ast r_\perp \lambda \xi_j \psi_{r_\parallel}^{\prime} \psixi\phi_{(\xi)}^2\xi_i\xi_j\\
&\hspace{0.5cm}+\sum_{j=1}^3 2 \, \psi_{(\xi)}^2 \big[n_\ast r_\perp \lambda (A_\xi)_j \partial_{y_1}\phi_{r_\perp}+n_\ast r_\perp \lambda (\xi \times A_\xi)_j \partial_{y_2}\phi_{r_\perp} \big] \phixi \xi_i\xi_j\\
&=\mu^{-1} \partial_t(\psi_{(\xi)}^2\phi_{(\xi)}^2 \xi_i),
\end{align*}
with $y_1:=n_\ast r_\perp \lambda (x-\alpha_\xi)\cdot A_\xi$ and $y_2:=n_\ast r_\perp\lambda (x-\alpha_\xi)\cdot (\xi \times A_\xi)$. The second step follows from the fact that $\{\xi,A_\xi,\xi \times A_\xi\}$ form an orthonormal basis of $\mathbb{R}^3$.\\
Therefore \eqref{3.43} boils, thanks to \eqref{3.25}, down to
\begin{align*}
&\divs(\mathring{R}_{\text{osc}})+\nabla p_{\text{osc}}\\
&\hspace{.5cm}=-\mu^{-1} \sum_{\xi \in \Lambda} \mathbb{P}\mathbb{P}_{\neq 0} \Big[ \partial_t \big(\axi^2\psi_{(\xi)}^2\phi_{(\xi)}^2 \xi \big) \Big]+ \sum_{\xi \in \Lambda} \mathbb{P}_{\neq 0} \Big[ \mathbb{P}_{\neq 0}\big(\Wxi \otimes \Wxi\big)\nabla \axi^2 \Big]\\
&\hspace{1cm}+ \mu^{-1} \sum_{\xi \in \Lambda} \mathbb{P}_{\neq 0} \Big[\partial_t \big(\axi^2\psi_{(\xi)}^2\phi_{(\xi)}^2 \xi\big) \Big]-\mu^{-1} \sum_{\xi \in \Lambda} \mathbb{P}_{\neq 0} \Big[\big(\partial_t \axi^2\big)\psi_{(\xi)}^2\phi_{(\xi)}^2 \xi \Big]+\nabla (\rho-p_\ell)
\end{align*}
and because of $\Id- \mathbb{P}= \nabla \Delta^{-1} \divs$ and \thref{Lemma 2.5} we may continue by writing
\begin{align*}
\divs(\mathring{R}_{\text{osc}})+\nabla p_{\text{osc}}
&=\nabla \bigg[\Delta^{-1} \divs\Big(\mu^{-1}\sum_{\xi \in \Lambda} \mathbb{P}_{\neq 0} \big[\partial_t (\axi^2\psi_{(\xi)}^2\phi_{(\xi)}^2 \xi ) \big]\Big)+\rho-p_\ell \bigg]\\
&\hspace{0.5cm}+ \sum_{\xi \in \Lambda} \mathbb{P}_{\neq 0} \Big[ \mathbb{P}_{\geq \frac{r_\perp \lambda_{q+1}}{2} }\big(\Wxi \otimes \Wxi\big)\nabla \axi^2 \Big]\\
&\hspace{0.5cm}-\mu^{-1} \sum_{\xi \in \Lambda} \mathbb{P}_{\neq 0} \Big[ \big(\partial_t \axi^2\big)\psi_{(\xi)}^2\phi_{(\xi)}^2 \xi \Big].\\
\end{align*}
\noindent
In order to find a specific representation of the stress terms $R_{\text{lin}}$ and $R_{\text{osc}}$ we need a right inverse of the divergence operator. We recall the one that emerged in \cite{DLS13}, which acts on the space of smooth, $\mathbb{R}^3$-valued vector fields on $ \mathbb{T}^3$ with zero mean and which takes values in $L^p\left(\mathbb{T}^3; \mathring{\mathbb{R}}_{\text{sym}}^{3\times 3}\right)$.

\begin{lemma} \thlabel{Lemma 3.8} 
The operator $\mathcal{R}$ defined by 
\begin{gather*}
\mathcal{R}\colon \left(C_{\neq 0}^\infty\left(\mathbb{T}^3; \mathbb{R}^3\right),\|\cdot\|_{L^p}\right) \to \left(L^p\left(\mathbb{T}^3; \mathring{\mathbb{R}}_{\text{sym}}^{3\times 3}\right), \|\cdot\|_{L^p}\right),\\
\mathcal{R}u:= \nabla \otimes (\Delta^{-1}u)+\Big(\nabla \otimes (\Delta^{-1}u)\Big)^T-\frac{1}{2} \Big((\nabla \otimes \nabla) \Delta^{-1}+ \Id \Big)\divs(\Delta^{-1}u)
\end{gather*}
is a right inverse of the divergence operator and is particularly bounded for $p \in (1,\infty)$.
\end{lemma}
\, \\
\noindent
We can moreover formulate the ensuing Lemma, which states that the composition of $\mathcal{R}$ with the differential operators $\Delta$, $\curl$ and with the Fourier-cut-off operator $\mathbb{P}_{\geq \kappa/2}$ are bounded operators as well.

\begin{lemma} \thlabel{Lemma 3.9} 
The composition of operators
\begin{itemize}
\item[i)] $
\mathcal{R}\curl \colon \left(C^\infty\left(\mathbb{T}^3;\mathbb{R}^3\right),\|\cdot\|_{L^p}\right) \to \left(L^p\left(\mathbb{T}^3; \mathring{\mathbb{R}}_{\text{sym}}^{3\times 3}\right), \|\cdot\|_{L^p}\right),
$

 \item[ii)]$
\mathcal{R}\Delta \colon \left(C^\infty\left(\mathbb{T}^3;\mathbb{R}^3\right),\|\cdot\|_{W^{1,p}}\right) \to \left(L^p\left(\mathbb{T}^3; \mathring{\mathbb{R}}_{\text{sym}}^{3\times 3}\right), \|\cdot\|_{L^p}\right),
$
\item[iii)]$
\mathcal{R}\mathbb{P}_{\geq\kappa}\colon \left( C^\infty\left(\mathbb{T}^3;\mathbb{R}^3\right) , \|\cdot\|_{L^p}\right) \to \left(L^p\left(\mathbb{T}^3; \mathring{\mathbb{R}}_{\text{sym}}^{3\times 3}\right), \|\cdot\|_{L^p}\right) 
$
\end{itemize}
are for $p \in (1,\infty)$ continuous ones. More precisely we find
\begin{align*}
\|\mathcal{R}\mathbb{P}_{\geq\kappa}\|_{L^p\to L^p} \lesssim \frac{1}{\kappa}.
\end{align*}
\end{lemma}
\vspace{.5cm}
\noindent Equipped with this knowledge it make sense to define 
 \begin{align*}
 \mathring{R}_{\text{lin}}&:=\mathcal{R}\partial_t (\wpr+\wc)+ \frac{1}{2}\mathcal{R} w_{q+1}-\mathcal{R}\Delta w_{q+1} +\Theta_\ell \Big(v_\ell \mathring{\otimes}w_{q+1}+w_{q+1}\mathring{\otimes}v_\ell\Big),\\
  \mathring{R}_{\text{cor}}&:=\Theta_\ell\Big[(\wc+\wt) \mathring{\otimes} w_{q+1}+ \wpr \mathring{\otimes}(\wc+\wt) \Big],\\
 \mathring{R}_{\text{osc}}&:= \sum_{\xi \in \Lambda} \mathcal{R}\mathbb{P}_{\neq 0} \Big[ \mathbb{P}_{\geq \frac{ r_\perp \lambda_{q+1}}{2}}\big(\Wxi \otimes \Wxi\big)\nabla \axi^2 \Big]-\mu^{-1} \sum_{\xi \in \Lambda}\mathcal{R} \mathbb{P}_{\neq 0} \Big[ \big(\partial_t \axi^2\big)\psi_{(\xi)}^2\phi_{(\xi)}^2 \xi \Big],\\
 \mathring{R}_{\text{com}}&:=(\Theta-\Theta_\ell)v_{q+1} \mathring{\otimes} v_{q+1}+\Theta_\ell v_\ell \mathring{\otimes}v_\ell -(\Theta v_q\mathring{\otimes} v_q)\ast_t \varphi_\ell \ast_x \phi_\ell
 \end{align*}
  with the corresponding pressure terms
 \begin{align*}
  p_{\text{lin}}&:=\frac{2}{3}\Theta_\ell(v_\ell \cdot w_{q+1}),\\
   p_{\text{cor}}&:= \frac{1}{3} \Theta_\ell \Big[(\wc+\wt)\cdot w_{q+1}+\wpr\cdot (\wc+\wt)\Big],\\
p_{\text{osc}}&:=\Delta^{-1} \divs\Big(\mu^{-1}\sum_{\xi \in \Lambda} \mathbb{P}_{\neq 0} \big[\partial_t (\axi^2\psi_{(\xi)}^2\phi_{(\xi)}^2 \xi ) \big]\Big)+\rho-p_\ell , \\
 p_{\text{com}}&:=\frac{1}{3}\Big[(\Theta-\Theta_\ell)|v_{q+1}|^2+\Theta_\ell |v_\ell|^2-(\Theta |v_q|^2)\ast_t\varphi_\ell \ast_x \phi_\ell \Big].
 \end{align*}
The Reynolds stress at level $q+1$ then becomes
\begin{align*}
\mathring{R}_{q+1}= \mathring{R}_\text{lin}+\mathring{R}_\text{cor}+\mathring{R}_\text{osc}+\mathring{R}_\text{com}
\end{align*}
and remains by construction symmetric and traceless.

\boldmath
\section{Inductive Estimates for the Reynolds Stress $\mathring{R}_{q+1}$} \label{Section 3.6}
\unboldmath
\boldmath
\subsection{Verifying the Key Bound on the Level $q+1$}\label{Section 3.6.1}
\unboldmath
Next, we aim at verifying \eqref{3.3c} on the level $q+1$. To do so, we need $r_\perp^{2/p-2}r_{\parallel}^{1/p-1}\leq \lambda_{q+1}^{\alpha}$, which can be achieved by taking $p \in (1,\frac{16}{16-7\alpha}] $. We start with a bound for the Linear error. 
\subsubsection{Linear Error}\label{Section 3.6.1.1}
Remembering that $\mathcal{R}\Delta$ and $\mathcal{R}\curl$ are according to \thref{Lemma 3.9} bounded operators on $C^\infty \left( \mathbb{T}^3;\mathbb{R}^3\right)$, we employ Hölder's inequality, \eqref{3.14} and \eqref{3.24} to obtain
\begin{align*}
\|\mathring{R}_\text{lin}\|_{C_tL^1} &\lesssim \sum_{\xi \in \Lambda} \ell^{-1}m^{5/8}\|\mathcal{R} \curl \curl (\axi \Vxi)\|_{C_tL^p}+m^{1/8}\sum_{\xi \in \Lambda} \|\mathcal{R} \partial_t\curl \curl (\axi \Vxi)\|_{C_tL^p}\\
&\hspace{.5cm}+\|w_{q+1}\|_{C_tL^p}+\|w_{q+1}\|_{C_tW^{1,p}}+m^{1/4}\|v_\ell \mathring{\otimes} w_{q+1}+w_{q+1} \mathring{\otimes} v_\ell\|_{C_tL^p}\\
& \lesssim \sum_{\xi \in \Lambda} \underbrace{\ell^{-1} m^{5/8}\|\curl(\axi \Vxi)\|_{C_tL^p}}_{=:\text{I}}+\sum_{\xi \in \Lambda} \underbrace{m^{1/8}\|\partial_t\curl(\axi \Vxi)\|_{C_tL^p}}_{=:\text{II}}+\underbrace{\|w_{q+1}\|_{C_tW^{1,p}}}_{=:\text{III}}\\
&\hspace{.5cm}+\underbrace{m^{1/4}\|v_\ell \mathring{\otimes} w_{q+1}+w_{q+1} \mathring{\otimes} v_\ell\|_{C_tL^p}}_{\text{=:IV}}.
\end{align*}

\pagebreak
\noindent \begin{tikzpicture}[baseline=(char.base)]
\node(char)[draw,fill=white,
  shape=rounded rectangle,
  drop shadow={opacity=.5,shadow xshift=0pt},
  minimum width=.8cm]
  {\Large I};
\end{tikzpicture}
We combine \eqref{3.8b} and \eqref{3.13c} and \eqref{3.1b} with $r_\perp^{2/p-2}r_\parallel^{1/p-1}\leq \lambda_{q+1}^\alpha$ to control the first term as follows

\begin{align*}
\ell^{-1} m^{5/8}\|\curl(\axi \Vxi)\|_{C_tL^p}&\leq \ell^{-1} m^{5/8}\bigg( \|\axi\|_{C_{t,x}^1}\|\Vxi\|_{C_tL^p}+\|\axi\|_{C_{t,x}^0} \sum_{|\gamma|=1}\|D^\gamma\Vxi\|_{C_tL^p}\bigg)\\
&\lesssim  \frac{M}{4 |\Lambda|} \left( \lambda_{q+1}^{34 \alpha- 22/7}+\lambda_{q+1}^{20\alpha-15/7} \right).
\end{align*}

\noindent  \begin{tikzpicture}[baseline=(char.base)]
\node(char)[draw,fill=white,
  shape=rounded rectangle,
  drop shadow={opacity=.5,shadow xshift=0pt},
  minimum width=.8cm]
  {\Large II};
\end{tikzpicture}
A bound for the second term can be deduced in the same manner: 
\begin{align*}
&m^{1/8}\|\partial_t \curl(\axi \Vxi)\|_{C_tL^p}\\
&\hspace{.5cm}\leq m^{1/8}\big(\|\axi\|_{C_{t,x}^2}\|\Vxi\|_{C_tL^p} +\|\axi\|_{C_{t,x}^1}\|\partial_t \Vxi\|_{C_tL^p}\big)\\
&\hspace{1cm}+m^{1/8}\big(\|\axi\|_{C_{t,x}^1} \sum_{|\gamma|=1}\|D^\gamma\Vxi\|_{C_tL^p}+\|\axi\|_{C_{t,x}^0} \sum_{|\gamma|=1}\|D^\gamma\partial_t\Vxi\|_{C_tL^p}\big)\\
&\hspace{.5cm}\lesssim \frac{M}{4|\Lambda|} \left(\lambda_{q+1}^{46\alpha-22/7}+\lambda_{q+1}^{32\alpha-8/7}+\lambda_{q+1}^{32\alpha-15/7}+\lambda_{q+1}^{18\alpha-1/7} \right).
\end{align*}

\noindent \begin{tikzpicture}[baseline=(char.base)]
\node(char)[draw,fill=white,
  shape=rounded rectangle,
  drop shadow={opacity=.5,shadow xshift=0pt},
  minimum width=.8cm]
  {\Large III};
\end{tikzpicture} Thanks again to the constrained $r_\perp^{2/p-2}r_\parallel^{1/p-1}\leq \lambda_{q+1}^\alpha$ the third term can owing to \eqref{3.17a}, \eqref{3.17b} and \eqref{3.1b} be estimated as
\begin{align*}
\|w_{q+1}\|_{C_tW^{1,p}}\lesssim \frac{M}{4 |\Lambda|} \lambda_{q+1}^{60 \alpha-1/7} +\left( \frac{M}{4 |\Lambda|}\right)^2 \lambda_{q+1}^{49\alpha-2/7}.
\end{align*}

\noindent \begin{tikzpicture}[baseline=(char.base)]
\node(char)[draw,fill=white,
  shape=rounded rectangle,
  drop shadow={opacity=.5,shadow xshift=0pt},
  minimum width=.8cm]
  {\Large IV};
\end{tikzpicture}
\noindent Keeping in mind that
\begin{align}\label{3.44}
\|v \mathring{\otimes} w\|_{\text{F}}=\sqrt{\|v \otimes w\|^2_{\text{F}} -\frac{1}{3} \tr\left(v \otimes w \right)^2} \leq \|v \otimes w\|_{\text{F}}= |v|\, |w|
\end{align}
holds for any two vector fields $v$ and $w$, the fourth term admits, according to \eqref{3.3b}, \eqref{3.15a}, \eqref{3.15b}, \eqref{3.15c}, \eqref{3.1a}, \eqref{3.1b} and $r_\perp^{2/p-2}r_\parallel^{1/p-1}\leq \lambda_{q+1}^\alpha$, the inequality
\begin{align*}
m^{1/4}\|v_\ell \mathring{\otimes} w_{q+1}+w_{q+1}\mathring{\otimes} v_\ell\|_{C_tL^p}&\lesssim m^{1/4}\|v_\ell\|_{C_{t,x}^0} \|w_{q+1}\|_{C_tL^p}\\
&
\lesssim   \frac{M}{4|\Lambda|} \lambda_{q+1}^{21\alpha-8/7}+\frac{M}{4|\Lambda|}\lambda_{q+1}^{49 \alpha-10/7}+\left(\frac{M}{4|\Lambda|} \right)^2 \lambda_{q+1}^{38\alpha-9/7} .
\end{align*}
\noindent
Thus, remembering the assumptions $161 \alpha<\frac{1}{7}$, $\alpha>4 \beta b^2$, $e(t)\geq \underline{e}>4$ and $\Theta^{-2}(t)\geq m^{-1/2}$ we can finally bound the linear error as
\begin{align}\label{3.45}
\|\mathring{R}_\text{lin}\|_{C_tL^1}&\leq S\left(\frac{M}{4|\Lambda|}+\left(\frac{M}{4|\Lambda|}\right)^2 \right)   \lambda_{q+1}^{60\alpha-1/7}\\
&\leq S\left(\frac{M}{4|\Lambda|}+\left(\frac{M}{4|\Lambda|}\right)^2 \right)   \lambda_{q+1}^{-100 \alpha-4\beta b^2}\notag\\
&\leq \frac{1}{6000} \delta_{q+3}\Theta^{-2}(t)e(t) \notag,
\end{align}
provided $S\left(\frac{M}{4|\Lambda
|}+\left(\frac{M}{4|\Lambda|}\right)^2 \right) \lambda_{q+1}^{-100\alpha}\leq \frac{1}{1500m^{1/2}}$ for some constant $S>0$ (cf. Section \ref{Section 3.1}).

\subsubsection{Corrector Error}\label{Section 3.6.1.2}
We apply \eqref{3.44} and Hölder's inequality twice in order to obtain
\begin{align*}
\|R_{\text{cor}}\|_{C_tL^1}
&\leq m^{1/4}\|\wc+\wt\|_{C_tL^{2p}}\|w_{q+1}\|_{C_tL^{2p}}+m^{1/4}\|\wc+\wt\|_{C_tL^{2p}}\|\wpr\|_{C_tL^{2p}}\\
&\leq m^{1/4}\|\wc+\wt\|_{C_tL^{2p}}^2+2m^{1/4}\|\wc+\wt\|_{C_tL^{2p}}\|\wpr\|_{C_tL^{2p}}.
\end{align*}
Appealing to \eqref{3.15b}, \eqref{3.15c}, \eqref{3.15a}, \eqref{3.1b} and $r_\perp^{2/p-2}r_\parallel^{1/p-1}\leq \lambda_{q+1}^{\alpha}$ it follows
\begin{align*}
\|\wc+\wt\|^2_{C_tL^{2p}} &\lesssim \left(\frac{M}{4|\Lambda|}\right)^2  \lambda_{q+1}^{91\alpha-4/7}+\left(\frac{M}{4|\Lambda|}\right)^4\lambda_{q+1}^{69\alpha-2/7}
\end{align*}
and 
\begin{align*}
\|\wc+\wt\|_{C_tL^{2p}} \|\wpr\|_{C_tL^{2p}}\lesssim \left(\frac{M}{4|\Lambda|}\right)^2  \lambda_{q+1}^{63\alpha-2/7}+\left(\frac{M}{4|\Lambda|}\right)^3\lambda_{q+1}^{52\alpha-1/7}.
\end{align*}
That means there exist a constant $\widetilde{S}>0$ such that
\begin{align}\label{3.46}
\|\mathring{R}_\text{cor}\|_{C_tL^1}&\leq \widetilde{S}\left(\left(\frac{M}{4|\Lambda|}\right)^2+\left(\frac{M}{4|\Lambda|}\right)^3+\left(\frac{M}{4|\Lambda|}\right)^4 \right)\lambda_{q+1}^{92\alpha-1/7}\\
&\leq\widetilde{S}\left(\left(\frac{M}{4|\Lambda|}\right)^2+\left(\frac{M}{4|\Lambda|}\right)^3+\left(\frac{M}{4|\Lambda|}\right)^4 \right)\lambda_{q+1}^{-69\alpha} \notag, 
\end{align} 
where we also took $m^{1/4}\leq \lambda_{q+1}^\alpha$ and $161\alpha<\frac{1}{7}$ into account. \\
In view of our choice of parameters (cf. Section \ref{Section 3.1}) it holds 
\begin{align*}
\widetilde{S}\left(\left(\frac{M}{4|\Lambda|}\right)^2+\left(\frac{M}{4|\Lambda|}\right)^3+\left(\frac{M}{4|\Lambda|}\right)^4 \right)\lambda_{q+1}^{-68 \alpha}\leq \frac{1}{1500m^{1/2}},
\end{align*}
so that together with $\alpha>4\beta b^2$, $e(t)\geq \underline{e}> 4$ and $\Theta^{-2}(t)\geq m^{-1/2}$ again we accomplish
\begin{align*}
\|\mathring{R}_\text{cor}\|_{C_tL^1}\leq  \frac{1}{1500m^{1/2}} \lambda_{q+1}^{-4\beta b^2}\leq \frac{1}{6000} \delta_{q+3}\Theta^{-2}(t)e(t) \notag.
\end{align*}

\subsubsection{Oscillation Error} \label{Section 3.6.1.3}
To control the first term of the oscillation error, we intend to apply \thref{Lemma A.4}. Thanks to Leibniz's formula, \eqref{3.8b} and \eqref{3.1b} we accomplish
\begin{align*}
\sum_{|\alpha|=n} \|D^\alpha (\nabla \axi^2)\|_{C_tL^\infty}& \lesssim \sum_{|\alpha|=n+1}  \sum_{\beta \leq \alpha} \|D^\beta \axi \|_{C_tL^\infty}\|D^{\alpha-\beta}\axi\|_{C_tL^\infty}\\
& \lesssim \sum_{k \leq n+1} \|\axi\|_{C^{k}_{t,x}} \|\axi\|_{C^{n+1-k}_{t,x}}\\
&\lesssim \left(\frac{M}{4|\Lambda|}\right)^2 \ell^{-24}\ell^{-7n}
\end{align*}
for any $n\geq 0$ and owing to the constraints $161\alpha<\frac{1}{7}$, $a\geq 3600$, $b\geq 7$ and \eqref{3.1b} also 
\begin{align*}
\ell^{-7}&\leq  \lambda_{q+1}^{14\alpha} <  \lambda_{q+1}^{-3/23} \lambda_{q+1}^{1/7} \leq  3600^{-7 \cdot \frac{3}{23}}r_\perp\lambda_{q+1} \leq  \frac{r_\perp\lambda_{q+1}}{2}
\intertext{and}
\ell^{-7N}& \leq (\lambda_{q+1}^{14 \alpha})^{ N}\leq  \lambda_{q+1}^{3/7-\frac{3}{23} N} \lambda_{q+1}^{\frac{1}{7}(N-3)}\\ &\leq  3600^{7 \cdot (3/7-\frac{3}{23} N)} (r_\perp  \lambda_{q+1})^{N-3} \leq \left(\frac{r_\perp \lambda_{q+1}}{2}\right)^{N-3}  ,
\end{align*}
as long as $\frac{3 \ln (2^{-1}\cdot 3600)}{\ln(2^{-1}\cdot 3600^{21/23})}<4 \leq N$. Therefore all requirements of \thref{Lemma A.4} are fulfilled and together with \eqref{3.1b}, \eqref{3.44} and \eqref{3.13c} it teaches us 
\begin{align*}
\Big\|\sum_{\xi\in \Lambda} \mathcal{R}\mathbb{P}_{\neq 0}\big[ \mathbb{P}_{\geq \frac{r_\perp \lambda_{q+1}}{2}}(\Wxi \otimes \Wxi)\nabla \axi^2\big] \Big\|_{C_tL^p}&\lesssim \left(\frac{M}{4|\Lambda|}\right)^2\lambda_{q+1}^{48 \alpha} \frac{\|\Wxi \otimes \Wxi\|_{C_tL^p}}{r_\perp \lambda_{q+1}}\\
&= \left(\frac{M}{4|\Lambda|}\right)^2\lambda_{q+1}^{48 \alpha} \|\Wxi\|^2_{C_tL^{2p}}r_\perp^{-1}\lambda_{q+1}^{-1}\\
&\lesssim \left(\frac{M}{4|\Lambda|}\right)^2\lambda_{q+1}^{49 \alpha-1/7},
\end{align*}
where we make again use of $r_\perp^{2/p-2}r_\parallel^{1/p-1}\leq \lambda_{q+1}^\alpha$ in the last step.\\
The second term of the oscillation error can, according to the boundedness of $\mathcal{R}\mathbb{P}_{\neq 0}$ on \linebreak $\left(C^\infty\left( \mathbb{T}^3; \mathbb{R}^3\right),\|\cdot\|_{L^p}\right)$, \thref{Lemma 3.4}, \eqref{3.8b}, \eqref{3.13a}, \eqref{3.13b}, \eqref{3.1b} and $r_\perp^{2/p-2}r_\parallel^{1/p-1}\leq \lambda_{q+1}^\alpha$ be estimated as
\begin{align*}
&\Big\| \mu^{-1} \sum_{\xi \in \Lambda} \mathcal{R} \mathbb{P}_{\neq 0} \big[(\partial_t \axi^2) \psi_{(\xi)}^2 \phi_{(\xi)}^2 \xi \big] \Big\|_{C_tL^p}\\ &\hspace{.5cm}\lesssim \mu^{-1} \sum_{\xi \in \Lambda} \|\partial_t \axi\|_{C_tL^\infty}\| \axi\|_{C_tL^\infty} \|\psi_{(\xi)}^2 \phi_{(\xi)}^2\|_{C_tL^p}\\
&\hspace{.5cm}\lesssim \mu^{-1} \sum_{\xi \in \Lambda} \| \axi\|_{C_{t,x}^1}\| \axi\|_{C_{t,x}^0} \|\psi_{(\xi)}\|^2_{C_tL^{2p}} \| \phi_{(\xi)}\|^2_{L^{2p}}\\
&\hspace{.5cm}\lesssim \left( \frac{M}{4|\Lambda|}\right)^2  \lambda_{q+1}^{49\alpha-9/7}.
\end{align*}
Under the assumptions $161\alpha<\frac{1}{7}$ and $\alpha>4\beta b^2$, we therefore get by Hölder's inequality
\begin{align}\label{3.47}
\|\mathring{R}_\text{osc}\|_{C_tL^1}\leq \hat{S}\left( \frac{M}{4|\Lambda|}\right)^2  \lambda_{q+1}^{49 \alpha-1/7} \leq \hat{S}\left( \frac{M}{4|\Lambda|}\right)^2  \lambda_{q+1}^{-111 \alpha}\lambda_{q+1}^{-2\beta b^2}
\end{align}
for some $\hat{S}>0$. To absorb this implicit universal constant, we impose $\hat{S}\left( \frac{M}{4|\Lambda|}\right)^2  \lambda_{q+1}^{-111 \alpha}\leq$ $\frac{1}{1500m^{1/2}}$ by possibly increasing $a$. Finally $e(t)\geq \underline{e}> 4$ and $\Theta^{-2}(t)\geq m^{-1/2}$ entail
\begin{align*}
\|\mathring{R}_\text{osc}\|_{C_tL^1}\leq \frac{1}{1500m^{1/2}}\delta_{q+3}\leq \frac{1}{6000} \delta_{q+3}\Theta^{-2}(t) e(t).
\end{align*}

 \subsubsection{Commutator Error}\label{Section 3.6.1.4}
We will estimate each term of 
\begin{align} \label{3.48}
\begin{split}
\|\mathring{R}_\text{com}\|_{C_tL^1} &\leq \underbrace{\|(\Theta-\Theta_\ell) v_{q+1} \mathring{\otimes} v_{q+1}\|_{C_tL^1}}_{=:\text{I}}+\underbrace{\|\Theta_\ell v_\ell \mathring{\otimes} v_\ell-\Theta_\ell v_q \mathring{\otimes} v_\ell \|_{C_tL^1}}_{=:\text{II}}\\
&\hspace{.5cm}+\underbrace{\|\Theta_\ell v_q \mathring{\otimes} v_\ell -\Theta v_q \mathring{\otimes}v_\ell\|_{C_tL^1}}_{_{=:\text{III}}}+\underbrace{\|\Theta v_q \mathring{\otimes}v_\ell-\Theta v_q \mathring{\otimes}v_q \|_{C_tL^1}}_{_{=:\text{IV}}}\\
&\hspace{.5cm}+\underbrace{\|\Theta v_q \mathring{\otimes}v_q-(\Theta v_q \mathring{\otimes}v_q )\ast_t \varphi_\ell \ast_x \phi_\ell\|_{C_tL^1}}_{_{=:\text{V}}}
\end{split}
\end{align}
separately.

\noindent \begin{tikzpicture}[baseline=(char.base)]
\node(char)[draw,fill=white,
  shape=rounded rectangle,
  drop shadow={opacity=.5,shadow xshift=0pt},
  minimum width=.8cm]
  {\Large I};
\end{tikzpicture}
To find a bound for $\|\Theta- \Theta_\ell \|_{C_t}$ we proceed in a similar  way as in \eqref{3.36} by using Itô's formula. More precisely we find
\begin{align} \label{3.49}
\|\Theta- \Theta_\ell \|_{C_t} \leq \frac{3}{2}m^{1/4} \ell^\iota.
\end{align} 
Together with \eqref{3.44} it furnishes
\begin{align*}
\text{I} & \leq \|\Theta- \Theta_\ell \|_{C_t} \sup_{s \in (-\infty,t]} \int_{\mathbb{T}^3} |v_{q+1}(s,x)|^2\, dx \leq \frac{3}{2}(2\pi)^{3/2}m^{1/4} \ell^\iota \|v_{q+1}\|_{C_{t,x}^1} \|v_{q+1}\|_{C_tL^2}.
\end{align*}

\noindent \begin{tikzpicture}[baseline=(char.base)]
\node(char)[draw,fill=white,
  shape=rounded rectangle,
  drop shadow={opacity=.5,shadow xshift=0pt},
  minimum width=.8cm]
  {\Large II \& IV};
\end{tikzpicture}
Keeping Cauchy-Schwarz's inequality in mind, we combine \eqref{3.44} with \eqref{3.6a} and \eqref{3.6b} to find 
\begin{align*}
\text{II} &\leq m^{1/4} \sup_{s \in (-\infty,t]} \int_{\mathbb{T}^3} |(v_\ell-v_q)(s,x)|\, |v_\ell (s,x) |\,dx
\leq m^{1/4} \|v_\ell-v_q\|_{C_tL^2}\|v_\ell\|_{C_tL^2}\\
& \leq (2\pi)^{3/2}m^{1/4} \ell \|v_q\|_{C^1_{t,x}} \|v_q\|_{C_tL^2}. 
\end{align*}
Note that the final bound is also valid for IV.\\

\noindent \begin{tikzpicture}[baseline=(char.base)]
\node(char)[draw,fill=white,
  shape=rounded rectangle,
  drop shadow={opacity=.5,shadow xshift=0pt},
  minimum width=.8cm]
  {\Large III};
\end{tikzpicture}
Moreover, we invoke \eqref{3.49}, \eqref{3.44} and \eqref{3.6b} to get 
\begin{align*}
\text{III} 
\leq \|\Theta_\ell- \Theta\|_{C_t}\sup_{s\in(-\infty,t]} \int_{\mathbb{T}^3} |v_q(s,x)|\, |v_\ell(s,x)| \,dx \leq \frac{3}{2}(2\pi)^{3/2}m^{1/4} \ell^\iota \|v_q\|_{C_{t,x}^1}\|v_q\|_{C_tL^2}.
\end{align*}

\noindent \begin{tikzpicture}[baseline=(char.base)]
\node(char)[draw,fill=white,
  shape=rounded rectangle,
  drop shadow={opacity=.5,shadow xshift=0pt},
  minimum width=.8cm]
  {\Large V};
\end{tikzpicture}
Thanks to the normalizations of mollifiers we have 
\begin{align*}
\text{V} 
&\leq \underbrace{\sup_{s \in (-\infty,t]}  \sup_{u \in [0,\ell]} \int_{\mathbb{T}^3}\|\Theta(s)(v_q \mathring{\otimes}v_q)(s,x)-\Theta(s-u)(v_q \mathring{\otimes}v_q)(s-u,x)\|_F \, dx }_{=:V_1}\\
&\hspace{.5cm}+\underbrace{\sup_{s \in (-\infty,t]}  \sup_{u \in [0,\ell]} \sup_{|y|\leq \ell }\int_{\mathbb{T}^3}\|\Theta(s-u)(v_q \mathring{\otimes}v_q)(s-u,x)-\Theta(s-u)(v_q \mathring{\otimes}v_q)(s-u,x-y)\|_F \, dx }_{=:V_2},
\end{align*}
where $V_1$ boils owing to \eqref{3.44} and Cauchy-Schwarz's inequality down to 
\begin{align*}
\text{V}_1 &\leq \sup_{s \in (-\infty,t]}  \sup_{u \in [0,\ell]} \int_{\mathbb{T}^3}\big\|[\Theta(s)v_q(s,x)- \Theta (s-u)v_q(s-u,x)] \mathring{\otimes}v_q(s,x)\big\|_F \, dx\\
&\hspace{.5cm}+\sup_{s \in (-\infty,t]}  \sup_{u \in [0,\ell]} \int_{\mathbb{T}^3}\big\| \Theta (s-u)v_q(s-u,x) \mathring{\otimes}[v_q(s,x)-v_q(s-u,x)]\big\|_F \, dx\\
&\leq \sup_{s \in (-\infty,t]}  \sup_{u \in [0,\ell]} \int_{\mathbb{T}^3} \underbrace{|\Theta(s)v_q(s,x)- \Theta (s-u)v_q(s-u,x)|}_{\leq 3m^{1/4} \ell^\iota \|v_q\|_{C_{t,x}^1} \, (\text{\ding{93}})} \,|v_q(s,x)| \, dx\\ 
&\hspace{.5cm}+\sup_{s \in (-\infty,t]}  \sup_{u \in [0,\ell]} \int_{\mathbb{T}^3} | \Theta (s-u)v_q(s-u,x)|\, |v_q(s,x)-v_q(s-u,x)| \, dx\\
&\leq 3m^{1/4} \ell^\iota\|v_q\|_{C_{t,x}^1} (2\pi)^{3/2} \|v_q\|_{C_tL^2}\\
&\hspace{.5cm}+m^{1/4} \|v_q\|_{C_tL^2}\sup_{s \in (-\infty,t]}  \sup_{u \in [0,\ell]} \left(\int_{\mathbb{T}^3} |v_q(s,x)-v_q(s-u,x)|^2 \, dx \right)^{1/2}\\
&= 3(2\pi)^{3/2}m^{1/4} \ell^\iota \|v_q\|_{C_{t,x}^1}  \|v_q\|_{C_tL^2}+m^{1/4} \|v_q\|_{C_tL^2}\sup_{s \in (-\infty,t]}  \sup_{u \in [0,\ell]} \bigg(\int_{\mathbb{T}^3} \Big| \int_0^1 \underbrace{\partial_\varsigma v_q(s-\varsigma u,x)}_{\substack{=-u\partial_s v_q(s-\varsigma u,x)\\ \hspace{-.9cm}\leq \ell \|v_q\|_{C_{t,x}^1}}}\, d\varsigma\Big|^2 \, dx \bigg)^{1/2}\\
&\leq 4(2\pi)^{3/2}m^{1/4} \ell^\iota\|v_q\|_{C_{t,x}^1}  \|v_q\|_{C_tL^2}.
\end{align*}
Here (\ding{93}) follows similarly to \eqref{3.36}, this time however with $e(s)$ replaced by $v_q(s,x)$ and $\bar{e}$, $\widetilde{e}$ by $\|v_q\|_{C_{t,x}^1}$.\\
One more time we employ \eqref{3.44} and stick to standard mollification estimates in order to control V$_2$ as follows 
\begin{align*}
\text{V}_2 &\leq m^{1/4} \sup_{s \in (-\infty,t]} \sup_{|y|\leq \ell }\int_{\mathbb{T}^3}\big\|[v_q(s,x)-v_q(s,x-y)] \mathring{\otimes}v_q(s,x)\big\|_F \, dx \\
&\hspace{.5cm}+ m^{1/4} \sup_{s \in (-\infty,t]}  \sup_{|y|\leq \ell }\int_{\mathbb{T}^3}\big\|v_q(s,x-y) \mathring{\otimes}[v_q(s,x)-v_q(s,x-y)]\big\|_F \, dx \\
&\leq m^{1/4} \sup_{s \in (-\infty,t]}   \sup_{|y|\leq \ell }\int_{\mathbb{T}^3} |v_q(s,x)-v_q(s,x-y)| \, |v_q(s,x)| \, dx \\
&\hspace{.5cm}+ m^{1/4} \sup_{s \in (-\infty,t]}   \sup_{|y|\leq \ell }\int_{\mathbb{T}^3}|v_q(s,x-y)|\,|v_q(s,x)-v_q(s,x-y)| \, dx \\
& \leq 2m^{1/4} \|v_q\|_{C_tL^2}\sup_{s \in (-\infty,t]}  \sup_{|y|\leq \ell }\left(\int_{\mathbb{T}^3} |v_q(s,x)-v_q(s,x-y)|^2 \, dx \right)^{1/2}\\
&= 2m^{1/4} \|v_q\|_{C_tL^2}\sup_{s \in (-\infty,t]}  \sup_{|y|\leq \ell }\left(\int_{\mathbb{T}^3} \Big| \int_0^1 \partial_\varsigma v_q(s,x-\varsigma y) \, d\varsigma\Big|^2 \, dx \right)^{1/2}\\
&= 2m^{1/4} \|v_q\|_{C_tL^2}\sup_{s \in (-\infty,t]}  \sup_{|y|\leq \ell }\left(\int_{\mathbb{T}^3} \Big| \int_0^1 \sum_{i=1}^3 y_i \partial_{x_i} v_q(s,x-\varsigma y) \, d\varsigma\Big|^2 \, dx \right)^{1/2}\\
&\leq 2 (2\pi)^{3/2}m^{1/4} \ell \|v_q\|_{C_{t,x}^1} \|v_q\|_{C_tL^2}.
\end{align*}
Thence 
\begin{align*}
\text{V}\leq 6 (2\pi)^{3/2} m^{1/4}\ell^\iota \|v_q\|_{C_{t,x}^1} \|v_q\|_{C_tL^2}.
\end{align*}
Plugging the above bounds into \eqref{3.48} and taking \eqref{3.3b}, \eqref{3.3a}, \eqref{3.4} and \eqref{3.1a} into account, leads to 
\begin{align*}
\|R_{\text{com}}\|_{C_tL^1} &\leq (2\pi)^{3/2}m^{1/4} \ell^\iota \left(\frac{3}{2}\|v_{q+1}\|_{C_{t,x}^1} \|v_{q+1}\|_{C_tL^2}+\frac{19}{2}\|v_q\|_{C_{t,x}^1} \|v_q\|_{C_tL^2} \right)\\
& \leq 33 (2\pi)^{3/2}M_0 m^{9/4}  \bar{e}  \lambda_{q+1}^{-\alpha \iota}.
\end{align*}
Finally, combining $33 (2\pi)^{3/2}M_0 m^{9/4} \bar{e}  \lambda_{q+1}^{-\frac{\alpha \iota}{2}} \leq \frac{1}{1500m^{1/2}}$ with $\alpha \iota >4\beta b^2$, $e(t)\geq \underline{e}>4$ and $\Theta^{-2}\geq m^{-1/2}$ permits to achieve 
\begin{align*}
\|R_{\text{com}}\|_{C_tL^1} \leq \frac{1}{1500m^{1/2}} \lambda_{q+1}^{-2\beta b^2} \leq \frac{1}{6000} \delta_{q+3} \Theta^{-2}(t)e(t).
\end{align*}
\newline
\newline
So altogether this proves \eqref{3.3c} at level $q+1$.

\subsection{Convergence of the Sequence}\label{Section 3.6.2}
Furthermore note that it follows instantly from \eqref{3.3c} that $\big(\mathring{R}_q \big)_{q \in \mathbb{N}_0}$ is a zero sequence in $C\left((-\infty,\tau]; L^1\left( \mathbb{T}^3\right) \right)$.

\section{Adapted and Deterministic}\label{Section 3.7}

\begin{figure}[H]
\begin{center}
\includegraphics[scale=0.62]{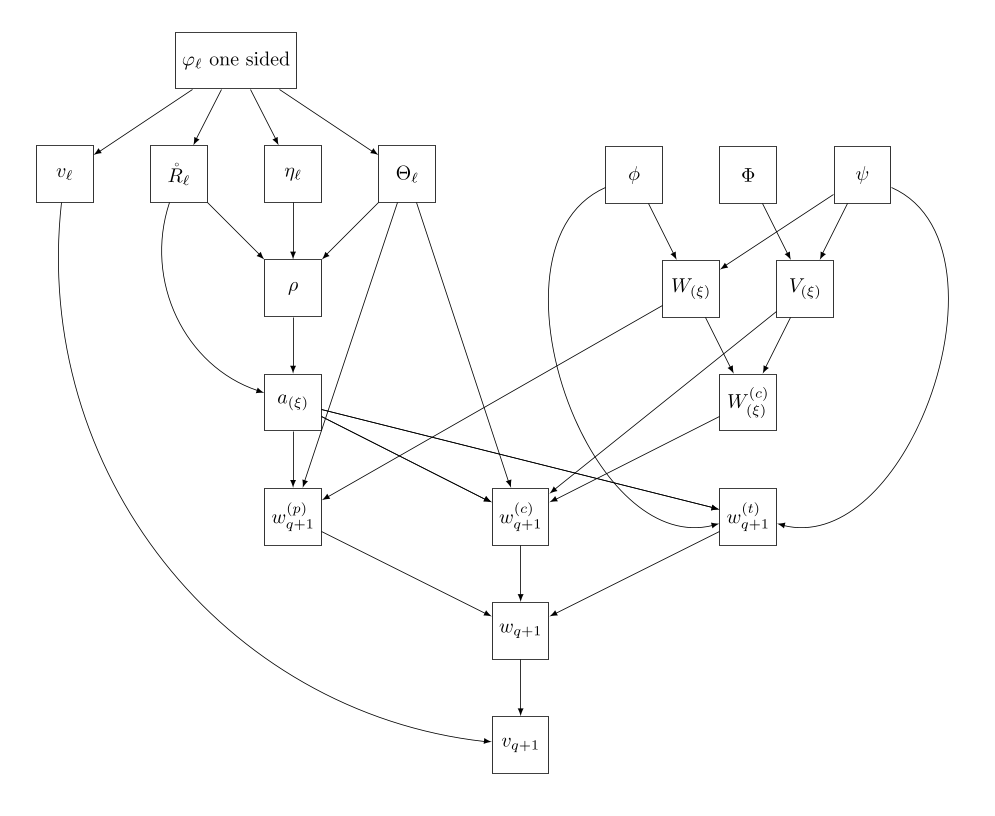}
\caption{Implication scheme of the adaptedness of each process up to the desired iteration step $v_{q+1}$.}
\label{Figure 3.6}
\end{center}
\end{figure}

\subsection{Adapted} \label{Section 3.7.1}
Now, we aim at proving the adaptedness of the next iteration step $(v_{q+1},\mathring{R}_{q+1})$. As already mentioned in Section \ref{Section 3.3.1} we emphasize again that the mollified velocity field $v_\ell$ as well as $\mathring{R}_\ell$ and $\Theta_\ell$ remain $\left(\mathcal{F}_t\right)_{t\in (-\infty, \tau]}$-adapted. Since the energy $e$ is deterministic, the adaptedness of $\eta_\ell$ follows in the same way as for $v_\ell$.\\
These facts in turn yield the adaptedness of $\rho$ and thus also of the amplitude functions $\axi$ as composition of them.\\
Furthermore note, that the building blocks $\phi,\, \Phi$ and $\psi$ of the intermittent jets do not depend on $\omega$, so consequently $\Vxi,\, \Wcxi$ and the intermittent jets themselves are deterministic as well. Together with the fact that any partial derivative of $\axi$ in space remains $\left(\mathcal{F}_t\right)_{t \in (-\infty, \tau]}$-adapted, establishes the adaptedness of $\wpr$ and $\wc$. Moreover, the Leray projection $\mathbb{P}$ and obviously also the projection $\mathbb{P}_{\neq 0}$ onto functions with zero mean preserve adaptedness, which confirms the adaptedness of the temporal corrector $\wt$.\\
To summarize our above considerations: we have proven that $v_\ell$ and the total perturbation $w_{q+1}=\wpr+\wc+\wt$ are, at any time $t \in (-\infty, \tau]$, $\mathcal{F}_t$-measurable, which finally verifies the adaptedness of $v_{q+1}=v_\ell+w_{q+1}$ and $\mathring{R}_{q+1}$.

 \subsection{Deterministic} \label{Section 3.7.2}
We will use an induction argument to verify, that the constructed sequence $\big(v_q(t,x),\mathring{R}_q(t,x)\big)_{q\in \mathbb{N}_0}$ is deterministic for each $t\leq 0$ and $x \in \mathbb{T}^3$.\\ Obviously the starting point $\big(v_0,\mathring{R}_0 \big)=(0,0)$ does not depend on $\omega$ at any time.\\
So assuming $v_q(t,x)$ for all $t\leq 0$ and $x \in \mathbb{T}^3$ to be deterministic, it can be easily seen that 
\begin{align*}
v_\ell(t,x,\omega)=\int_0^\ell \int_{|y|\leq \ell}v_q(t-s,x-y,\omega)\varphi_\ell(s) \phi_\ell(y)\,dy \,ds,
\end{align*}
$\mathring{R}_\ell(t,x)$ and $\Theta_\ell(t)$ are deterministic as well. Combined with the fact, that $\eta_q(t,\cdot)$ is deterministic, $\rho(t,\cdot)$ and hence $\axi(t,\cdot)$ does not depend on $\omega$ either. Based on this we are able to conclude that each part of the total perturbation $w_{q+1}(t,\cdot)$, is deterministic, that is to say $\wpr,\, \wc$ and $\wt$ do not depend on $\omega$ at each time up to time $0$. The functions $\phi,\,\Phi$ and $\psi$ and hence the intermittent jets $W_{(\xi)}$ as well as its incompressibility corrector $W^{(c)}_{(\xi)}$ and $V_{(\xi)}$ are by definition deterministic. \\
On the one hand, this together with the independence of $v_\ell(t,\cdot)$ of $\omega$ results in the deterministic behavior of $v_{q+1}$ at time $t\leq 0$. On the other hand it also follows that $\mathring{R}_\text{lin}(t,\cdot)$, $\mathring{R}_\text{cor}(t,\cdot)$, $\mathring{R}_\text{osc}(t,\cdot)$, $\mathring{R}_\text{com}(t,\cdot)$ and thereby $\mathring{R}_{q+1}(t,\cdot)$ are deterministic as well, yielding the assertion.

\chapter{End of the Proof of Theorem 1.1} \label{Chapter 4}
Summarizing our previous results, we formulate the following Proposition.

\begin{prop}[Main iteration] \thlabel{Proposition 4.1}
For an $(\mathcal{F}_t)_{t\in (-\infty,\tau]}$-adapted solution $\big(v_q,\mathring{R}_q\big)$ to \eqref{3.2} on $(-\infty,\tau]$, which admits the bounds in \eqref{3.3} and \eqref{3.5} for some $q\in \mathbb{N}_0$, an  $\left(\mathcal{F}_t\right)_{t \in (-\infty,\tau]}$-adapted process $\big(v_{q+1},\mathring{R}_{q+1}\big)$ can be constructed, so that this pair also solves \eqref{3.2} on $(-\infty,\tau]$ and obeys \eqref{3.3} and \eqref{3.5} at level $q+1$. Moreover the sequences   $\left(v_q\right)_{q\in \mathbb{N}_0}$ and $\big(\mathring{R}_q\big)_{q\in \mathbb{N}_0}$ are Cauchy in $C\left((-\infty,\tau];L^2\left( \mathbb{T}^3 \right)\right)$ and $C\left((-\infty,\tau];L^1\left( \mathbb{T}^3 \right)\right)$, respectively; more precisely $\big(\mathring{R}_q\big)_{q\in \mathbb{N}_0}$ converges to zero and \eqref{3.42} holds.\\ Additionally $v_q(t,x)$ and $\mathring{R}_q(t,x)$ are deterministic for all $t\leq 0,\, x \in \mathbb{T}^3$ and $q\in \mathbb{N}_0$.
\end{prop}

\vspace{.5cm}
\noindent
Now we have everything in hand to finish the proof of \thref{Theorem 1.2}.
\begin{proof}[Proof of \thref{Theorem 1.2}]
\item \underline{Existence:}~\\
According to Proposition \ref{Proposition 4.1} there exists a sequence $\big(v_q,\mathring{R}_q \big)_{q\in \mathbb{N}_0}$, so that $\big(v_q(t),\mathring{R}_q(t) \big)$ is for every $q\in \mathbb{N}_0 $ and $t\leq 0$ deterministic and so that $\big(v_q, \mathring{R}_q \big)$ solves \eqref{3.2} on $(-\infty,\tau]$ in the weak sense. That means for any divergence free test function $\varphi \in C^\infty\left(\mathbb{T}^3;\mathbb{
R}^3\right)$ the pair $\big(v_q,\mathring{R}_q \big)$ satisfies  \pagebreak
\begin{align} \label{4.1}
&\underbrace{\int_{\mathbb{T}^3}\Big(v_q(t,x)-v_q(0,x)\Big)\cdot\varphi(x) \,dx}_{=:\text{I}}+\frac{1}{2}\underbrace{\int_0^t\int_{\mathbb{T}^3}v_q(s,x)\cdot\varphi(x)\,dx\,ds}_{=:\text{II}}-\underbrace{\int_0^t \int_{\mathbb{T}^3} v_q(s,x)\cdot \Delta\varphi(x)\,dx\,ds}_{=:\text{III}}\notag \\
&\hspace{.3cm}+\underbrace{\int_0^t \Theta(s)\int_{\mathbb{
T}^3}(v_q \otimes v_q)(s,x):\nabla \varphi^T(x)\,dx \,ds}_{=: \text{IV}}+\underbrace{\int_0^t\int_{\mathbb{T}^3}  p_q(s,x)\cdot\divs(\varphi(x))\, dx \,ds}_{=:\text{V}}  \\
&= \underbrace{\int_0^t \int_{\mathbb{T}^3} \mathring{R}_q(s,x): \nabla \varphi^T(x) \,dx \,ds}_{=:\text{VI}} \notag
\end{align}
in particular for every $t\in [0,\tau]$, where the fifth term vanishes due to the fact that we work with solenoidal test functions.\\
Moreover the sequence can be chosen in such a way that $\big(\mathring{R}_q\big)_{q \in \mathbb{N}_0}$ converges to zero in \linebreak $C\left((-\infty,\tau] ;L^1\left(\mathbb{T}^3\right)\right)$ and such that $(v_q)_{q\in \mathbb{N}_0}$ is Cauchy in the Banach space $C\left((-\infty,\tau];L^2\left( \mathbb{T}^3\right)\right)$. Integration by parts and Cauchy-Schwarz's inequality permits therefore to deduce

\noindent \begin{tikzpicture}[baseline=(char.base)]
\node(char)[draw,fill=white,
  shape=rounded rectangle,
  drop shadow={opacity=.5,shadow xshift=0pt},
  minimum width=.8cm]
  {\Large VI};
\end{tikzpicture} 

\begin{align*}
&\bigg|\int_0^t \int_{\mathbb{T}^3} \mathring{R}_q(s,x): \nabla \varphi^T(x) \,dx \,ds \bigg|\\
 &\hspace{.5cm}\leq \int_0^t\int_{\mathbb{T}^3} \left( \sum_{i,j=1}^3\big|  \big(\mathring{R}_q(s,x) \big)_{i,j} \big|^2\right)^{1/2} \, \left(\sum_{i,j=1}^3 \big|\partial_{ x_j} \varphi_i(x)\big|^2\right)^{1/2} \,dx \,ds \\
&\hspace{.5cm}\leq \sqrt{3}\int_0^t\int_{\mathbb{T}^3} \|\mathring{R}_q(s,x)\|_F\|\varphi\|_{C^1} \,dx \,ds\\
&\hspace{.5cm}\leq \sqrt{3} \|\varphi\|_{C^1}   \|\mathring{R}_q\|_{C_tL^1} \underbrace{\int_0^t\,ds}_{\leq \tau \leq 1} \overset{q \to \infty}{\longrightarrow} 0.
\end{align*}
Furthermore we may assume the existence of a limit $\lim\limits_{q \to \infty}v_q=:v \in C\left((-\infty,\tau];L^2\left( \mathbb{T}^3\right)\right)$, which causes that each term on the left hand side of \eqref{4.1} will converge for $q \to \infty$ pointwise. In fact integrating by parts again and Cauchy-Schwartz's inequality furnish

\noindent \begin{tikzpicture}[baseline=(char.base)]
\node(char)[draw,fill=white,
  shape=rounded rectangle,
  drop shadow={opacity=.5,shadow xshift=0pt},
  minimum width=.8cm]
  {\Large I};
\end{tikzpicture} 
\noindent \begin{align*}
&\bigg|\int_{\mathbb{T}^3}\big[ v_q(t,x)-v_q(0,x)\big] \cdot\varphi(x) \,dx-\int_{\mathbb{T}^3}\big[v(t,x)-v(0,x)\big]\cdot\varphi(x) \,dx \bigg|\\
&\hspace{.5cm}\leq (2\pi)^{3/2} \|\varphi\|_C\Big[\|v_q(t)-v(t)\|_{L^2} +\|v_q(0)-v(0)\|_{L^2}\Big]\overset{q \to \infty}{\longrightarrow} 0,
\end{align*}

\noindent \begin{tikzpicture}[baseline=(char.base)]
\node(char)[draw,fill=white,
  shape=rounded rectangle,
  drop shadow={opacity=.5,shadow xshift=0pt},
  minimum width=.8cm]
  {\Large II};
\end{tikzpicture} 
\begin{align*}
&\bigg|\int_0^t\int_{\mathbb{T}^3}v_q(s,x)\cdot\varphi(x)\,dx\,ds-\int_0^t\int_{\mathbb{T}^3}v(s,x)\cdot\varphi(x)\,dx\,ds\bigg| \leq (2\pi)^{3/2} \|\varphi\|_C\|v_q-v\|_{C_tL^2}\overset{q \to \infty}{\longrightarrow} 0,
\end{align*}
\pagebreak 
\, 

\noindent \begin{tikzpicture}[baseline=(char.base)]
\node(char)[draw,fill=white,
  shape=rounded rectangle,
  drop shadow={opacity=.5,shadow xshift=0pt},
  minimum width=.8cm]
  {\Large III};
\end{tikzpicture} 
\begin{align*}
&\bigg|\int_0^t \int_{\mathbb{T}^3} v_q(s,x)\cdot \Delta\varphi(x)\,dx\,ds-\int_0^t \int_{\mathbb{T}^3}  v(s,x)\cdot\Delta\varphi(x)\,dx\,ds \bigg|\\
&\hspace{.5cm}\leq (2\pi)^{3/2} \|\varphi\|_{C^2}\|v_q-v\|_{C_tL^2}\overset{q \to \infty}{\longrightarrow} 0
\end{align*}
and by virtue of Cauchy-Schwartz's inequality and \eqref{3.44} also

\noindent \begin{tikzpicture}[baseline=(char.base)]
\node(char)[draw,fill=white,
  shape=rounded rectangle,
  drop shadow={opacity=.5,shadow xshift=0pt},
  minimum width=.8cm]
  {\Large IV};
\end{tikzpicture}

\begin{align*}
 &\bigg|\int_0^t \Theta(s)\int_{\mathbb{
T}^3}(v_q \otimes v_q)(s,x):\nabla \varphi^T(x)\,dx \,ds-\int_0^t \Theta(s)\int_{\mathbb{
T}^3}(v \otimes v)(s,x):\nabla \varphi^T(x)\,dx \,ds\bigg|\\
&\hspace{.5cm}\leq m^{1/4}  \int_0^t \int_{\mathbb{T}^3} \left(\sum_{i,j=1}^3 \big|(v_q \otimes v_q-v\otimes v)_{ij}(s,x)\big|^2\right)^{1/2} \left(\sum_{i,j=1}^3 \big|\partial_{x_j} \varphi_i(x)\big|^2\right)^{1/2} \,dx \,ds \\
&\hspace{.5cm}\leq \sqrt{3}m^{1/4}  \|\varphi\|_{C^1} \underbrace{\int_0^t\,ds}_{\leq \tau \leq 1}  \sup_{s \in [0,t]} \int_{\mathbb{T}^3} \|v_q \otimes  (v_q-v)-(v-v_q)\otimes v\|_F(s,x)\,dx  \\
&\hspace{.5cm}\leq \sqrt{3}m^{1/4}  \|\varphi\|_{C^1}  \sup_{s \in [0,t]} \int_{\mathbb{T}^3} |v_q (s,x)|\, |v_q(s,x)-v(s,x)|\,dx\\
&\hspace{1cm}+\sqrt{3}m^{1/4} \|\varphi\|_{C^1}\sup_{s \in [0,t]} \int_{\mathbb{T}^3}|v(s,x)-v_q(s,x)| \, | v(s,x)|\,dx  \\
& \hspace{.5cm}\leq \sqrt{3}m^{1/4}  \|\varphi\|_{C^1} \Big( \|v_q\|_{C_tL^2} \|v_q-v\|_{C_tL^2}+ \|v_q-v\|_{C_tL^2}\|v\|_{C_tL^2}\Big)\overset{q \to \infty}{\longrightarrow} 0.
\end{align*}
So passing to the limit on both sides of \eqref{4.1} shows that $v$ is an analytically weak solution to \eqref{1.2} with deterministic initial condition. Therefore $u:=e^Bv$ solves \eqref{1.1} on $[0,\tau]$ in the probabilistically strong and analytically weak sense; of course also with a deterministic initial condition $u_0$.
\item \underline{Regularity:}~\\
It remains to verify that the convergence of $(v_q)_{q\in \mathbb{N}_0}$ to $v$ even takes place in $ C\left((-\infty,\tau];H^\gamma\left(\mathbb{T}^3\right) \right)$. For this end we combine Hölder's inequality with exponents $\frac{1}{\gamma}$ and $\frac{1}{1-\gamma}$ with Plancherel's theorem to accomplish
\begin{align*}
\sum_{q\geq 0} \|v_{q+1}(t)-v_q(t)\|_{H^\gamma}&=\sum_{q \geq 0} \big\|\mathcal{F}^{-1}\big[(1+|\cdot|^2)^{\gamma/2} \mathcal{F}(v_{q+1}(t)-v_q(t))\big] \big\|_{L^2}\\
&\hspace{-0.5cm}\overset{\text{Plancherel}}{=}\sum_{q \geq 0} \big\|(1+|\cdot|^2)^{\gamma/2} \mathcal{F}(v_{q+1}(t)-v_q(t)) \big\|_{\ell^2(\mathbb{Z}^3)}\\
&=\sum_{q \geq 0} \big\|(1+|\cdot|^2)^{\gamma} \big[\mathcal{F}(v_{q+1}(t)-v_q(t))\big]^2 \big\|^{1/2}_{\ell^1(\mathbb{Z}^3)}\\
&=\sum_{q \geq 0} \big\|(1+|\cdot|^2)^{\gamma} \big[\mathcal{F}(v_{q+1}(t)-v_q(t))\big]^{2\gamma} \big[\mathcal{F}(v_{q+1}(t)-v_q(t))\big]^{2(1-\gamma)} \big\|^{1/2}_{\ell^1(\mathbb{Z}^3)}\\
&\hspace{-0.26cm}\overset{\text{Hölder}}{\leq}  \sum_{q \geq 0} \big\|(1+|\cdot|^2)\big[ \mathcal{F}(v_{q+1}(t)-v_q(t))\big]^2 \big\|_{\ell^1(\mathbb{Z}^3)}^{\gamma/2} \big\|\big[\mathcal{F}(v_{q+1}(t)-v_q(t))\big]^2\big\|_{\ell^1(\mathbb{Z}^3)}^{(1-\gamma)/2}\\
& \hspace{-0.5cm}\overset{\text{Plancherel}}{=} \sum_{q\geq 0} \|v_{q+1}(t)-v_q(t)\|_{H^1}^\gamma \|v_{q+1}(t)-v_q(t)\|^{1-\gamma}_{L^2}\\
&\lesssim \sum_{q\geq 0} \|v_{q+1}-v_q\|_{C_{t,x}^1}^\gamma \|v_{q+1}-v_q\|^{1-\gamma}_{C_tL^2}.
\end{align*}
for any $t \in (-\infty,\tau]$. Invoking \eqref{3.3b} and \eqref{3.42} then yields
\begin{align*}
\sum_{q\geq 0} \|v_{q+1}(t)-&v_q(t)\|_{H^\gamma}\lesssim 2^\gamma\bigg(M_0+(2\pi)^{3/2} \bigg)^{1-\gamma}m\bar{e}^{1/2}\lambda_1^{\beta(1-\gamma)}\sum_{q\geq 0} a^{\big[5\gamma-\beta(1-\gamma)\big]b^{q+1}}.
\end{align*}
So if we impose $\gamma \in \Big(0, \frac{\beta}{5+\beta} \Big)$, the above series will converge as it can be bounded by a geometric series. Thus the sequence $\left(\sum\limits_{q\geq 0}^n \|v_{q+1}-v_q\|_{C_tH^\gamma}\right)_{n\in \mathbb{N}_0}$ is as a convergent sequence particularly Cauchy in $\mathbb{R}$, which means that for all $\varepsilon>0$ there exists some $N\in \mathbb{N}_0$ so that
\begin{align*}
 \|v_n-v_k\|_{C_tH^\gamma}\leq \sum_{q=k}^{n-1} \|v_{q+1}-v_q\|_{C_tH^\gamma} \leq \varepsilon
\end{align*}
holds for every $n\geq k\geq N$. In other words $\left(v_q\right)_{q \in \mathbb{N}_0}$ is also a Cauchy sequence in the Banach space $C\left((-\infty,\tau];H^\gamma\left(\mathbb{T}^3\right)\right)$ for $\gamma \in \Big(0, \frac{\beta}{5+\beta} \Big)$, furnishing the existence of the limit in $C\left((-\infty,\tau];H^\gamma\left(\mathbb{T}^3\right)\right)$, which equals, by virtue of the uniqueness of the limit, to $v$.
\item \underline{Bounded:}\\ Furthermore it follows 
\begin{align*}
\|v(t,\omega)\|_{H^\gamma} 
&\leq \|v(t,\omega)-v_{q+1}(t,\omega)\|_{H^\gamma}+ \sum_{k=0}^q\|v_{k+1}(t,\omega)-v_k(t,\omega)\|_{H^\gamma}\\
&\lesssim \varepsilon+ 2^\gamma\bigg(M_0+(2\pi)^{3/2} \bigg)^{1-\gamma}m\bar{e}^{1/2}\lambda_1^{\beta(1-\gamma)} \left(1-a^{\big[5\gamma-\beta(1-\gamma)\big]b}\right)^{-1}
\end{align*}
for any $\varepsilon>0$, sufficiently large $q\geq 0$ and every $t \in [0,\tau]$ and $\omega \in \Omega$. Note that the constant on the right hand side neither depends on $t$ nor on $\omega$.
\item \underline{Kinetic energy:}\\
Moreover, remembering that $\left(v_q\right)_{q \in \mathbb{N}_0}$ is a Cauchy sequence in $C\left((-\infty,\tau];L^2\left(\mathbb{T}^3\right) \right)$, we may deduce 
\begin{align*}
\|u(t)\|_{L^2}^2=e(t)
\end{align*}
from \eqref{3.5}. 
\item \underline{Consistency:}\\
Let $e_1$ and $e_2$ be two energies in $C_b^1\left((-\infty,1];[\underline{e},\infty) \right)$ respecting
\begin{align*}
\|e_i\|_{C}\leq \bar{e}, \qquad \Big\|\frac{d}{dt} e_i\Big\|_{C}\leq \widetilde{e}
\end{align*}
and $e_i(t)=e_i(0)$ for every $t\leq 0$ and $i=1,2$, which coincide on $[0,t]$ for some $t \in [0,1]$ and let $\left(v^1_q\right)_{q \in \mathbb{N}_0}$, $\left(v^2_q\right)_{q \in \mathbb{N}_0}$ be the corresponding sequences, constructed in Proposition \ref{Proposition 4.1}. Then for each $q \in \mathbb{N}_0$ and $i=1,2$, $v^i_{q+1}$ consists of the previous, mollified iteration $v_\ell^i:=v_q^i\ast_t\varphi_\ell \ast_x \phi_\ell$ and the perturbation $w_{q+1}^i$. If we decompose the perturbations $w_{q+1}^1$ and $w_{q+1}^2$ in a way that is presented in Figure \ref{Figure 3.6}, one can see that they coincide on $(-\infty,t \wedge\tau]$, if $\mathring{R}_\ell^1$ and $\mathring{R}_\ell^2$ and also $\eta_\ell^1$ and $\eta_\ell^2$ do. The functions $\phi,\, \Phi$ and $\psi$ can be chosen for both sequences in the same way. \\
Taking into account that $e_1=e_2$ on $(-\infty,t]$ together with $v_q^1=v_q^2$ on $(-\infty,t \wedge \tau]$ implies $\eta_q^1=\eta_q^2$ on $(-\infty,t \wedge \tau]$, it indeed holds
\begin{align*}
\eta_\ell^1(s,x,\omega)=\int_{|y|\leq \ell}\int_{0}^{\ell} \eta_q^1(s-u,x-y,\omega)\varphi_\ell(u)\phi_\ell(y)\, du\, dy =\eta_\ell^2(s,x,\omega)
\end{align*}
for all $s\in (-\infty, t\wedge \tau]$, $x \in \mathbb{T}^3$ and $\omega \in \Omega$.\\
In the same manner we deduce $v_\ell^1(s)=v_\ell^2(s)$ and $\mathring{R}_\ell^1(s)=\mathring{R}_\ell^2(s)$, furnishing $v_{q+1}^1(s)=v_{q+1}^2(s)$  and thus also $\mathring{R}_{q+1}^1(s)=\mathring{R}_{q+1}^2(s)$ at any time $s\in(-\infty, t\wedge \tau]$. That means by induction we have ascertained that the sequences $\left(v^1_q\right)_{q \in \mathbb{N}_0}$ and $\left(v^2_q\right)_{q \in \mathbb{N}_0}$ are the same on $(-\infty, t\wedge \tau]$; hence so are their limits
\begin{align*}
v^1:=\lim_{q\to \infty }v_q^1=\lim_{q\to \infty }v_q^2=:v^2.
\end{align*}
As a consequence the solutions $u_1:=e^Bv^1$ and $u_2:=e^Bv^2$ to \eqref{1.1} on $[0,\tau]$, associated to the energies $e_1$ and $e_2$, respectively, coincide on $[0, t\wedge \tau]$, completing the proof of \thref{Theorem 1.2}.
\end{proof}

\appendix

\chapter{Appendix} \label{Appendix} \label{Appendix A.1} 
\begin{lemma}\thlabel{Lemma A.1}
For any $\varepsilon>0$ there exist a homomorphism $\varphi \colon B_\varepsilon\big([x]\big)\subseteq \mathbb{R}^3\setminus \left(2\pi \mathbb{Z}\right)^3 \to B_\varepsilon(x) \subseteq \mathbb{R}^3$, so that a $\mathbb{T}^3$-periodic function $u \colon \mathbb{R}^3\to \mathbb{R}^3$ solves \eqref{1.1}, if and only if $u \circ \varphi \colon \mathbb{R}^3\setminus \left(2\pi \mathbb{Z}\right)^3 \to \mathbb{R}^3$ does.
\end{lemma}

\begin{lemma}\thlabel{Lemma A.2} 
For $d\geq 3$ and a finite set of directions $\Lambda \subseteq \mathbb{R}^d $, there exists some $(\alpha_\xi)_{\xi \in \Lambda}$ and $\rho>0$ such that the periodic tubes 
\begin{align*}
\left(B_\rho (0)+\alpha_\xi+\{s \xi\}_{s\in \mathbb{R}}+2\pi\mathbb{Z}^d\right)_{\xi \in \Lambda}
\end{align*}
are mutually disjoint.
\end{lemma}
\begin{proof}
See \cite{BMS21}, p.9, Lemma 3.
\end{proof}

\begin{lemma}  \thlabel{Lemma A.3}
Take $p \in \{1,2\}$ and let $g$ be a $\left(\frac{\mathbb{T}}{\kappa}\right)^3$-periodic function with $\kappa \in \mathbb{N}$ and $f$ be a $\mathbb{T}^3$-periodic function, satisfying
\begin{align*}
\sum_{|\alpha|\leq n}\|D^\alpha f\|_{L^p} \leq C_f \zeta^{n}
\end{align*}
for some constants $C_f>0,\, \zeta \geq 1 $ and every $0\leq n\leq N+4$ with $N\in\mathbb{N}$. Then it holds
\begin{align*}
\|fg\|_{L^p}\lesssim C_f \|g\|_{L^p},
\end{align*}
provided
\begin{align*}
3\, \frac{2\pi}{\kappa}\zeta \leq \frac{1}{41} \quad \text{and} \quad 16\exp(12)\zeta^4 \left(3\, \frac{2\pi }{\kappa}2\exp(3)\zeta \right)^N \leq 1.
\end{align*}
\end{lemma}

\begin{proof}
See \cite{BV19b}, p.12, Lemma 3.7.
\end{proof}

\begin{lemma} \thlabel{Lemma A.4}
Assume $1\leq \zeta <\kappa$ to be parameters, satisfying $\kappa^{3-N}\zeta^N\leq 1$ for some non-negative integer $N \geq 2$, let $f\in L^p\left(\mathbb{T}^3 \right)$ with $p \in (1,2]$ and let $a \in C^N\left(\mathbb{T}^3\right)$ be a function which obeys
\begin{align*}
\sum_{|\alpha|= n}\|D^\alpha a\|_{L^\infty} \leq C_a\zeta^{n}
\end{align*}
for some constant $C_a>0$ and $n \in \{0,N\}$. It then holds that
\begin{align*}
\| \, \mathcal{R}\mathbb{P}_{\neq 0}(a \mathbb{P}_{\geq\kappa}f)\|_{L^p}\lesssim C_a\frac{\|f\|_{L^p}}{\kappa}, 
\end{align*}
where the implicit constant depends on $p$ and $N$.
\end{lemma}
\begin{proof}
See \cite{BV19b}, p.32, Lemma B.1.
\end{proof}

\thispagestyle{empty}


\addcontentsline{toc}{chapter}{References}
\begin{thebibliography}{999999999999}







\bibitem[AT74]{AT74}
A.A. Aleksandrov, M.S. Takhtengerts, Viscosity of water at temperatures of $-20$ to $150^\circ$C, vol. 27 in Journal of Engineering Physics and Thermophysics (pp. 1235–1239), \url{https://doi.org/10.1007/BF00864022}, 1974.

\bibitem[Be23]{Be23}
S. Berkemeier, A Toolbox for the Methods of Convex Integration, (not published).

\bibitem[BDLIS15]{BDLIS15}
T. Buckmaster, C. De Lellis, P. Isett, and L. Sz\'ekelyhidi \newblock Jr., Anomalous dissipation for 1/5-Hölder Euler flows, vol. 182, Issue 1 in Annals of Mathematics (pp. 127–172), \url{https://doi.org/10.4007/annals.2015.182.1.3}, 2015.




\bibitem[BDLSV17]{BDLSV17}
T. Buckmaster, C. De Lellis, L. Székelyhidi Jr., V. Vicol, Onsager's conjecture for admissible weak solutions, arXiv:1701.08678, 2017. 




\bibitem[BV19a]{BV19a}
T. Buckmaster, V. Vicol, Convex integration and phenomenologies in turbulence, arXiv:1901.09023, 2019.


\bibitem[BV19b]{BV19b}
T. Buckmaster, V. Vicol, Nonuniqueness of weak solutions to the Navier-Stokes equation, vol. 189, no. 1 in Annals of Mathematics (pp. 101-144), \url{https://doi.org/10.4007/annals.2019.189.1.3}, 2019.

\bibitem[BMS21]{BMS21}
J. Burczak, S. Modena, L. Sz\'ekelyhidi Jr., Non Uniqueness of power-law flows, vol. 388 in Communications in Mathematical Physics (pp.199-243), \url{https://doi.org/10.1007/s00220-021-04231-7}, 2021.



\bibitem[CDZ10]{CDZ10}
W. Chen, Z. Dong, X. Zhu, Sharp non-uniqueness of solutions to stochastic Navier-Stokes equations, arXiv:2208.08321, 2022.


\bibitem[DLS09]{DLS09}
C. De Lellis, L. Sz\'ekelyhidi Jr., The Euler equations as a differential inclusion, vol. 170, no. 3 in Annals of Mathematics (pp. 1417–36), \url{https://www.jstor.org/stable/25662181}, 2009. 


\bibitem[DLS10]{DLS10}
C. De Lellis, L. Sz\'ekelyhidi Jr., On admissibility criteria for weak solutions of the Euler equations, vol. 195 in Archive for Rational Mechanics and Analysis (pp. (225–260), \url{https://doi.org/10.1007/s00205-008-0201-x}, 2010.



\bibitem[DLS13]{DLS13}
C. De Lellis, L. Sz\'ekelyhidi Jr., Dissipative continuous Euler flows, vol. 193 in Inventiones mathematicae (pp. 377–407), \url{ https://doi.org/10.1007/s00222-012-0429-9}, 2013.


\bibitem[DNPV11]{DNPV11}
E. Di Nezza, G. Palatucci, E. Valdinoci, Hitchhiker's guide to the fractional Sobolev spaces, arXiv:1104.4345, 2011.






\bibitem[Ev10]{Ev10}
L.C. Evans, Partial differential equations, vol. 19 in Graduate Studies in Mathematics, American Mathematical Society, Providence, 2nd ed., 2010.


\bibitem[HLP22]{HLP22}
M. Hofmanová, T. Lange, U. Pappalettera, Global Existence and Non-Uniqueness of $3D$ Euler Equations Perturbed by Transport Noise, (to appear). 



\bibitem[HZZ19]{HZZ19}
M. Hofmanová, R. Zhu, X. Zhu, Non-uniqueness in law of stochastic 3D Navier-Stokes equations, arXiv:1912.11841, 2019.





\bibitem[HZZ21a]{HZZ21a}
M. Hofmanová, R. Zhu, X. Zhu, Global-in-time probabilistically strong and Markov solutions to stochastic 3D Navier-Stokes equations: Existence and non-uniqueness. arXiv:2104.09889, 2021 (to appear in Annals of Probability).


\bibitem[HZZ21b]{HZZ21b}
M. Hofmanová, R. Zhu, X. Zhu, Global existence and non-uniqueness for 3D Navier-Stokes equations with space-time white noise. arXiv:2112.14093, 2021.

\bibitem[HZZ21c]{HZZ21c}
M. Hofmanová, R. Zhu, X. Zhu, On ill- and well-posedness of dissipative martingale solutions to stochastic 3D Euler Equations, vol. 75, Issue 11 in Pure and Applied Mathematics (pp. 2446-2510), \url{https://doi.org/10.1002/cpa.22023}, 2021.



\bibitem[HZZ22a]{HZZ22a}
M. Hofmanová, R. Zhu, X. Zhu, Non-unique ergodicity for deterministic and stochastic 3D Navier-Stokes and Euler equations, arXiv:2208.08290v1, 2022.

\bibitem[HZZ22b]{HZZ22b}
M. Hofmanová, R. Zhu, X. Zhu, A class of supercritical/critical singular stochstic PDEs: Existence, non-uniqueness, non-gaussianity, non-unique ergodicity, arXiv:2205.13378v1, 2022. 


\bibitem[Is16]{Is16}
P. Isett, A Proof of Onsager's Conjecture, arXiv:1608.08301, 2016.



\bibitem[LZ22]{LZ22}
H. Lü, X. Zhu, Global-in-times probabilistically strong solutions to stochastic power-law equations: Existence and non-uniqueness, arXiv:2209.02531, 2022.

\bibitem[RS22]{RS22}
M. Rehmeier, A. Schenke, Non-Uniqueness in Law for Stochastic Hypodissipative Navier-Stokes Equations, arXiv:2104.10798v2, 2022.

\bibitem[RZZ13]{RZZ13}
M. Röckner, R. Zhu, X. Zhu, Local existence and non-explosion of solutions for stochastic fractional partial differential equations driven by multiplicative noise, arXiv:1307.4392, 2013.

\bibitem[Tr83]{Tr83}
H. Triebel, Theory of function spaces, vol. 78 in Monographs in Mathematics, Birkhäuser Verlag, Basel, Boston, Stuttgart, 1983.

\bibitem[Tr92]{Tr92}
H. Triebel, Theory of function spaces II, vol. 84 in Monographs in Mathematics, Birkhäuser Verlag, Basel, Boston, Berlin, 1992.

\bibitem[Ya21a]{Ya21a}
K. Yamazaki, Non-Uniqueness in Law of Three-Dimensional Navier-Stokes Equations Diffused via a Fractional Laplacian with Power Less than One Half, arXiv:2104.10294v1, 2021.

\bibitem[Ya21b]{Ya21b}
K. Yamazaki, Non-Uniqueness in Law of Three-Dimensional Magnetohydrodynamics System Forced by Random Noise, arXiv:2109.07015v1, 2021.


\bibitem[Ya22]{Ya22}
K. Yamazaki, Non-Uniqueness in Law of the Two-Dimensional Surface Quasi-Geostrophic Equations Forced by Random Noise, arXiv:2208.05673v2, 2022.

\end{thebibliography}
\end{document}